\documentclass[10pt]{article}

\usepackage{geometry}
\RequirePackage[OT1]{fontenc}
\RequirePackage{amsthm,amsmath,amssymb}
\RequirePackage[colorlinks,citecolor=blue,urlcolor=blue]{hyperref}
\usepackage{stmaryrd}
\usepackage[sans]{dsfont}
\usepackage{authblk}

\usepackage[english]{babel}
\usepackage{latexsym}
\usepackage{bbm}
\usepackage[mathcal]{euscript}
\usepackage{graphicx}
\usepackage{pstricks}
\usepackage{pstricks-add}
\usepackage{pst-plot}
\usepackage{enumerate}
\usepackage{multicol}
\usepackage{subfigure}
\usepackage{mathrsfs}
\newcommand{\B}{\mathscr{B}}    

\newcommand{\e}{\varepsilon}   
\newcommand{\EE}{\mathbb{E}}     
\newcommand{\F}{\mathcal{F}}    
 
\newcommand{\ind}{\mathbbm{1}} 
\newcommand{\J}{\mathcal{J}}
\renewcommand{\L}{\mathcal{L}}
\newcommand{\M}{\mathcal{M}}    %
\newcommand{\NN}{\mathbb{N}}
\newcommand{\RR}{\mathbb{R}}    

\newcommand{\V}{\hbox{Var}}

\newcommand{\ZZ}{\mathbb{Z}}

\numberwithin{equation}{section}
\theoremstyle{plain}
\newtheorem{theorem}{Theorem}[section]
\newtheorem{lemma}[theorem]{Lemma}
\newtheorem{proposition}[theorem]{Proposition}
\newtheorem{corollary}[theorem]{Corollary}
\newtheorem*{remark}{Remark}
\newtheorem*{example}{Example}
\theoremstyle{definition}
\newtheorem{definition}{Definition}[section]

\begin{document}

\title{Proximinality and uniformly approximable sets in $L^p$}

\author{Guillaume Grelier\thanks{Departamento de Matemáticas, Universidad de Murcia, 
Campus de Espinardo, 30100 Espinardo, Murcia, Spain email: g.grelier@um.es}  \& Jaime San Mart\'in\thanks{CMM-DIM;  Universidad de Chile; UMI-CNRS 2807; 
Casilla 170-3 Correo 3 Santiago; Chile. \break email: jsanmart@dim.uchile.cl}}

\maketitle

\begin{abstract}
For any $p\in[1,\infty]$, we prove that the set of simple functions taking at most $k$ different values is proximinal in $L^p$ for all $k\geq 1$. 
We introduce the class of uniformly approximable subsets of $L^p$, which is larger than the class of uniformly integrable sets. This new class is 
characterized in terms of the $p$-variation if $p\in[1,\infty)$ and in terms of covering numbers if $p=\infty$. We study properties of uniformly 
approximable sets. In particular, we prove that the convex hull of a uniformly approximable bounded set is also uniformly approximable and that 
this class is stable under H\"older transformations. We also prove that, for $p\in [1,\infty)$, the unit ball of $L^p$ is uniformly approximable 
if and only if $L^p$ is finite-dimensional, while for $p=\infty$ the unit ball is always uniformly approximable.  
\end{abstract}



\section{Introduction}
In this paper we study the approximation of measurable functions by simple functions taking at most $k$ values, for $k\in \NN$. 
This problem has important consequences in multiple applications, where for example, one seeks for reduction of dimensionality, 
among many others. For example, the embedding of metric spaces into finite-dimensional normed spaces with small dimension is one of 
the main issue in non-linear analysis (see \cite{Bourgain,JL,Matousek}). These results have deep consequences in order to design approximation algorithms, 
for instance for the Sparsest Cut problem  (see \cite{Naor}). When we aproximate a given function $f\in L^p(\Omega,\F,\mu)$ by simple functions, 
the number of terms in those approximations growths to infinity 
in general. Here, a main concern is what we can say if we restrict the number of terms in the approximations. In particular, what we can say about subsets of 
$L^p(\Omega,\F,\mu)$ that can be uniformly approximated by simple functions taking $k$ values, as $k$ growth to $\infty$. As we shall see, this new concept
is more general than uniform integrability or compactness, and we fully characterize it in terms of a new measure of variation defined for functions 
in $L^p(\Omega,\F,\mu)$ for $p\in [1,\infty)$, and in terms of covering numbers in the case of $p=\infty$.

Let us fix some notations we need to explain the main results of this paper. Consider $(\Omega,\F,\mu)$ a measure space. For any $k\geq 1$, we denote by 
$\mathscr{G}_{p,k}(\Omega,\F, \mu)$, or simply $\mathscr{G}_{p,k}$ when the measure space $(\Omega,\F, \mu)$ is 
clear from the context, the set of simple functions given by
$$
\mathscr{G}_{p,k}=\left\{\sum\limits_{i=1}^l a_i \ind_{A_i}\in L^p(\Omega,\F,\mu) :\ 
\{A_i\}_{1\le i\le l} \hbox{ measurable partition of } \Omega,\ a_i\in \RR 
\hbox{ for all $i$, } l\le k\right\}.
$$
\begin{remark} Assume $\mu$ is a finite measure. Then $\mathscr{G}_{p,k}=\mathscr{G}_{1,k}$ for all $p\in [1,\infty]$ is just the 
set of simple measurable functions that takes
at most $k$ values. If $\mu$ is an infinite measure, then $h=\sum_{i=1}^k a_i \ind_{A_i}$, where $\{A_i\}_{1\leq i\leq k}$ is a 
measurable partition, belongs to $\mathscr{G}_{p,k}$, for 
$p\in [1,\infty)$, if and only if  $\mu(A_i)=\infty$ implies $a_i=0$. So, again for all $p\in [1,\infty)$ it holds that 
$\mathscr{G}_{p,k}=\mathscr{G}_{1,k}\subset \mathscr{G}_{\infty,k}$, and 
the latter is the set of all simple measurable functions that takes at most $k$ values.
\end{remark}

We recall some notions from approximation theory. Let $X$ be a Banach space and let $K$ be a closed subset of $X$. The \textit{metric projection} on $K$ 
is the multi-valued mapping $P_K:X\rightrightarrows K$ defined by $P_K(x)=\{y\in K\ :\ \|x-y\|=d(x,K)\}$ (where $d(A,B)$ is the 
distance between two subsets $A$ and $B$ of $X$). If $P_K(x)$ is not empty for all 
$x\in X$, we say that $K$ is \textit{proximinal}. If $P_K(x)$ is a singleton for all $x\in X$, we say that $K$ is \textit{Chebyshev}. 
Section \ref{sec:minimizing} is devoted to show the following result.

\begin{theorem}
\label{mainresult}
Let $(\Omega,\F,\mu)$ be a measure space and $p\in[1,+\infty]$. Then $\mathscr{G}_{p,k}$ is proximinal in $L^p(\Omega,\F,\mu)$ for all $k\geq 1$.
\end{theorem}

In other words, the distance of a function $f$ to $\mathscr{G}_{p,k}$ is attained at some $g\in\mathscr G_{p,k}$. Note that most of the classical 
results on the existence of a solution cannot be used in this case since $\mathscr{G}_{p,k}$ is obviously not compact in the strong topology, nor convex, and as
we will see, it is not closed in the weak topology, in general. The proof of this result is divided into several steps. We first deal with the case $p\in[1,\infty)$ 
and we also prove that a minimum can be chosen to have a particular form (see Theorem \ref{pro:minimum} when $\mu$ is 
finite and Theorem \ref{pro:minimum2} if not). The proof is rather technical since we deal with any kind of measure 
(not only finite or $\sigma$-finite). In case $\mu$ is finite we also give conditions to ensure that there is a unique minimizer (see Theorem \ref{pro:uniqueness}).
In general, the set of minimizers is not a singleton, it can even exists a continuum of minimizers. Then, it makes sense to study
if the metric projection $P_{\mathscr{G}_{p,k}}$ has a continuous selection. In general, there is no continuous selection, unless $L^p(\Omega,\F,\mu)$
is finite dimensional (see the Remark before Section 3).

When $p=\infty$,  in Proposition \ref{pro:minimum3} we prove that $\mathscr{G}_{\infty,k}$ is proximinal. The proofs we provide are somehow
constructive in nature, but still there is a long way to go for obtaining useful algorithms, 
which in itself, we think, will be important in many applications.
\smallskip

An important role in this section is played by $\M_p(f,A)$, the $p$-th mean of $f$ on a set $A$ (see Definition \ref{eq:M_p}).
In particular, for $p=2$ we have $\M_2(f,A)=\frac{1}{\mu(A)}\int_A f(x) \ d\mu(x)$. A well-known approximation associated to a finite measurable
partition $\mathcal{P}=\{A_i\}_{1\leq i\leq k}$ is given by
$$
\EE^{\mathcal{P}}(f)=\sum\limits_{i=1}^k \ \M_2(f,A_i) \ind_{A_i},
$$
which corresponds to the conditional expectation of $f$ over the $\sigma$-field generated by $\mathcal{P}$. 
\medskip

In Section \ref{sec:pvariation}, we introduce the $p$-variation $\V_{p,k}(f)$ of a function $f\in L^p(\Omega,\F,\mu)$, for $p\in [1,\infty)$ 
and we studied some of its properties. 
The $p$-variation of a function allow us to control the distance of $f$ to the sets $\mathscr{G}_{p,k}$, up to a factor of $2$ (see Proposition \ref{variation}). 
This notion will be a useful tool to characterize the uniform approximability of sets in the following section and whose definition is the following:

\begin{definition}\label{def_UA}
Let $(\Omega,\F,\mu)$ be a measure space and $p\in[1,+\infty]$. Let $\mathscr{A}\subset L^p(\Omega,\F,\mu)$. For $\varepsilon>0$, we define 
$$
N_{p,\varepsilon}(\mathscr{A})=\inf\{k\geq 1\ :\ \forall f\in
\mathscr{A},\ \exists h\in\mathscr{G}_{p,k}\ \ \|f-h\|_p\le \e\}.
$$ 
As usual if the set  where the infimum is taken is empty we set
$N_{p,\varepsilon}(\mathscr{A})=\infty$. 
We say that $\mathscr{A}$ is
\textit{uniformly approximable} (in short UA) in $L^p(\Omega,\F,\mu)$ if
$N_{p,\varepsilon}(\mathscr{A})<\infty$ for any $\varepsilon>0$. 
\end{definition}

Concretely a set $\mathscr{A}$ is UA in $L^p(\Omega,\F,\mu)$ if for any
$\varepsilon>0$ there exists $k\geq 1$ such that any function in $\mathscr{A}$ 
can be $\varepsilon$-approximated in $L^p(\Omega,\F,\mu)$ by simple functions taking less than $k$ 
different values. Notice that $\mathscr{A}$ is UA if and only if 
$$
\lim\limits_{k\to \infty}\sup\limits_{f\in \mathscr{A}}\inf\{\|f-g\|_p:\ g\in \mathscr{G}_{p,k}\}=0.
$$
We point out that a similar quantity leads to relatively compactness of $\mathscr{A}$. Indeed, if $1\le p<\infty$, 
a result inspired by M. Riesz (see Theorem 4.7.28 in \cite{Bogachev}) says that $K\subset L^p(\Omega,\F,\mu)$ 
is relatively compact if and only if $K$ is bounded in $L^p(\Omega,\F,\mu)$ and
$$
\inf\limits_{\mathcal{P}}\sup\limits_{f\in K} \|f-\EE^{\mathcal{P}}(f)\|_p=0.
$$
We point out that for every finite measurable partition $\mathcal{P}$, with at most $k$ atoms it holds 
$$
\inf\{\|f-g\|_p:\ g\in \mathscr{G}_{p,k}\}\le \|f-\EE^{\mathcal{P}}(f)\|_p, 
$$
so relatively compactness implies UA, a fact
that can be easily proved directly.
\medskip

The last part of the paper, Section \ref{sec:UAp}, is dedicated to the study of uniformly approximable sets. We will give some examples of UA sets and prove that 
it is a larger class than the class of uniformly integrable sets. We also characterize this property in terms of covering numbers 
if $p=\infty$ and in terms of the $p$-variation if $p<\infty$. The covering numbers $\mathcal{N}(f,\varepsilon)$ of a 
function $f$ is simply defined as the covering number of its range, up to measure $0$. We will prove the following two results:

\begin{theorem}
Let $(\Omega,\F,\mu)$ be a measure space and let $\mathscr{A}\subset L^\infty(\Omega,\F,\mu)$. 
The following assertions are equivalent:
\begin{enumerate}
    \item[(i)] $\mathscr{A}$ is UA;
    \item[(ii)] $\sup_{f\in\mathscr{A}}\mathcal{N}(f,\varepsilon)<\infty$ for all $\varepsilon>0$.
\end{enumerate}
\end{theorem}

\begin{theorem}
Let $(\Omega,\F,\mu)$ be a measure space, $p\in[1,\infty)$ and let $\mathscr{A}\subset L^p(\Omega,\F,\mu)$. Then, the following are equivalent
\begin{itemize}
\item[(i)] $\mathscr{A}$ is UA in $L^p(\Omega,\F,\mu)$;
\item[(ii)] $\lim\limits_{k\to \infty} \sup_{f\in \mathscr{A}} \V_{p,k}(f) =0$.
\end{itemize}
\end{theorem}

Then we investigate when the unit ball of $L^p(\Omega,\F,\mu)$ is UA. If $1\leq p<\infty$, this happens, as one can expect, if and only if $L^p(\Omega,\F,\mu)$ 
is finite dimensional (see Theorem \ref{ballUA}).  We conclude this section by establishing some stability properties of the class of UA sets. 
In particular, a nice use of the Rademacher type allows us to prove that if $\mathscr A$ is a bounded UA set in $L^p(\Omega,\F,\mu)$ 
for $p\in(1,\infty)$ then its closed convex hull also is UA (see Theorem \ref{convex_hull}). For more information about Rademacher type and 
cotype, we refer the reader  to \cite{AlbiacKalton} (chapter 6).\\

In what follows all the measures considered are assumed non trivial, that is, different from the $0$ measure, unless it is explicitly stated.
We believe that our notation is quite standard. For example, the closure of a set $A$ is denoted by $\overline{A}$ 
and the distance between two subsets $A$ and $B$ in a metric space is denoted by $d(A,B)$. The complement of a set $A$ is denoted by $A^c$. 
In some of the results we will need to consider diffuse and atomic measures. For that reason we fix some notations at this respect. 
We recall that an \textit{atom} in a measure space $(\Omega,\F,\mu)$ is a measurable set
$A$ that satisfies: $\mu(A)>0$ and if $B\subset A$ is a measurable set such that $\mu(B)<\mu(A)$ then $\mu(B)=0$. Notice 
that if $A_1, A_2$ are two atoms with finite measure, then either $\mu(A_1\cap A_2)=0$ or they differ on a set of measure $0$, that is, 
$\mu(A_1\Delta A_2)=0$ (where $\Delta$ is the symmetric difference). A measurable space is said to be \textit{atomic} if every measurable 
set of positive measure contains an atom. An atomic space
is said to have a finite number of atoms of finite measure, up to measure $0$, if there exists 
a finite collection (eventually empty) $\mathcal{A}$ of atoms of finite measure such that for any atom $B$ either $\mu(B)=\infty$ or 
there exists $A \in \mathcal{A}$ such that $\mu(A\Delta B)=0$. 
A \textit{diffuse} measure, is a measure that has no atoms. Notice that the measure $\mu\equiv 0$ 
is by definition diffuse, and we refer to this case as the trivial one. More information about measure theory can be found in \cite{Bogachev}.

\section{Minimizing the distance to the sets $\mathscr{G}_{p,k}$}
\label{sec:minimizing}

The main objective of this section is to prove that $\mathscr{G}_{p,k}$ proximinal, i.e. given some $f\in L^p(\Omega,\F,\mu)$, the distance from $f$ 
to $\mathscr{G}_{p,k}$ is reached at some function $g\in\mathscr{G}_{p,k}$ (see Theorem \ref{mainresult}). We denote by
$$
\mathscr{D}_{p,k}(f)=\inf\{\|f-h\|_p:\ h\in \mathscr{G}_{p,k}\}.
$$ for all $p\in[1,\infty]$, that is the distance between $f$ and $\mathscr{G}_{p,k}$. A function $g\in P_{\mathscr{G}_{p,k}}$ 
will be called a \textit{minimizer}. As we mentioned in the introduction, 
the classical results of optimization do not apply in this case since $\mathscr{G}_{p,k}$ is not convex nor compact. Even in the 
reflexive case (that is $1<p<\infty$), it is not clear if the problem admits a solution. However, if $1<p<\infty$ and $\mathscr{G}_{p,k}$ 
is weakly closed, it is easy to see that there exists a minimizer. In fact, let $(g_n)_n\subset\mathscr{G}_{p,k}$ such that 
$\|g_n-f\|\to \mathscr{D}_{p,k}(f)$. In particular, $(g_n)_n$ is bounded and then admits a subsequence $(g_{n'})_{n'}$ 
that weakly converges to some $g\in\mathscr{G}_{p,k}$. Since the norm is weakly lower semicontinuous, we obtain that 
$$
\mathscr{D}_{p,k}(f)\leq\|f-g\|_p\leq\lim_{n'}\|f-g_{n'}\|_p=\mathscr{D}_{p,k}(f),
$$ 
implying that $\mathscr{D}_{p,k}(f)=\|f-g\|_p$. Unfortunately, as the following discussion will show, $\mathscr{G}_{p,k}$ 
is not weakly closed in general, a fact that depends strongly on the measure space. On the one hand, in the case of the $\ell_p$ 
spaces for $1\le p<\infty$, every $\mathscr{G}_{p,k}$ is closed
under the weak topology. This follows directly from the fact that if $(f_n)_n\subset \mathscr{G}_{p,k}$ converges to $f$ weakly in $\ell_p$, 
then $(f_n)_n$ converges pointwise to $f$. From this fact it follows that $f(\NN)$ is a finite set with cardinality at most $k$ and therefore 
$f\in \mathscr{G}_{p,k}$. On the other extreme we have the following result:

\begin{proposition} Consider $([0,1],\L,dx)$ the Lebesgue measure and let $p\in[1,\infty)$. Then $\mathscr{G}_{p,k}$ 
is weakly dense in $L^p([0,1],\L,dx)$ for all $k\ge 2$.
\end{proposition}
\begin{proof} It is enough to prove the case $k=2$. Consider an integer $r\ge 2$. Every $x\in [0,1)$ has a unique expansion
$$
x=\sum\limits_{n=1}^\infty \zeta^r_n(x) r^{-n}.
$$
where $\zeta^r_n(x)\in \{0,...,r-1\}$ and $(\zeta^r_n(x))_n$ is not eventually constant $r-1$. For $x=1$, we define $\zeta^r_n(x)=r-1$ for all $n\geq 1$.

Let us prove that for every $A\in \L$, the sequence $(\ind_{A\cap\{\zeta^2_n=1\}})_n$ converges weakly to the function $f=\frac12\ind_A$. 
Indeed, assume first that $A=[0,1]$. For $n\ge 1$, let $\Psi_n:[0,1]\to [0,1]$ be the bi-measurable and measure preserving 
transformation which flips the $n$-th binary digit. Then
for all continuous functions $g:[0,1]\to \RR$ it holds
$$
\int_{\{\zeta^2_n=1\}} g(x) \ dx=\int_{\{\zeta^2_n=0\}} g(x)\ dx+R_n,
$$
where $R_n=\int_{\{\zeta^2_n=1\}} g(x)-g(\Psi_n(x)) \ dx$. The continuity of $g$, allow us to prove that $R_n$ 
converges to zero. This shows $\ind_{\{\zeta^2_n=1\}}$ converges weakly to $\frac12 \ind_{[0,1]}$. 
Thus, for all $h\in L^q$, where $q$ is the conjugated index of $p$, and all $A\in \L$ we have
$$
\lim\limits_{n\to \infty} \int \ind_{\{\zeta^2_n=1\}}(x) \ind_A(x) h(x) \ dx=\frac12\int \ind_A(x) h(x) \ dx,
$$
showing that $(\ind_{A\cap\{\zeta^2_n=1\}})_n$ converges weakly to $\frac12 \ind_A$.

In a similar way, it is shown that for all $A\in\L$, any integer number $r\ge 2$, any $m\in\{1,...,r\}$ and all $0\le t_1<t_2...<t_m\le r-1$, the sequence
$$
f_n=\ind_{A\cap \cup_{j=1}^m \{\zeta^r_n=t_j\}}=\sum\limits_{j=1}^m \ind_{A\cap \{\zeta^r_n=t_j\}}\in \mathscr{G}_{p,2},
$$
converges weakly to $\frac{m}{r}\ind_A$. 

Now, for any $\ell\ge 1$, any partition $\{A_j\}_{1\leq j\leq\ell}$ of measurable sets,   
any collection $\{r_j\}_{1\leq j\leq\ell}$ of integer numbers greater or equal than 2, any collection $\{m_j\}_{1\leq j\leq\ell}$ such that 
$m_j\in\{1,...,r_j\}$ and any collection of integer numbers $\{t_{j,i}\ :\ 1\leq i\leq m_j,\ 1\leq j\leq\ell\}$ such that
$0\le t_{j,1}<...<t_{j,m_j}\le r_j-1$, we obtain that the sequence
$$
f_n=\sum_{j=1}^\ell \ind_{A_j\cap \cup_{i=1}^{m_j} \{\zeta^{r_j}_n=t_{j,i}\}}=\sum_{j=1}^\ell \sum_{i=1}^{m_j} \ind_{A_j\cap \{\zeta^{r_j}_n=t_{j,i}\}},
$$
converges weakly to $\sum\limits_{j=1}^\ell\frac{m_j}{r_j} \ind_{A_j}$. We notice that $f_n=\ind_{B_n}$, where 
$$
B_n=\bigcup_{j=1}^\ell \bigcup_{i=1}^{m_j} A_j \cap \{\zeta^{r_j}_n=t_{j,i}\},
$$
so $f_n\in \mathscr{G}_{p,2}$. This shows that the weak closure of $\mathscr{G}_{p,2}$ contains all the simple functions of the form
$
f=\sum\limits_{j=1}^\ell \alpha_j \ind_{A_j},
$
where $\ell\geq 1$, $\{A_j\}_{1\leq j\leq \ell}$ is any finite measurable partition and $ \alpha_j\in[0,1]$ for all $j\in\{1,...,\ell\}$. 
Moreover, any such simple function is the weak limit of a sequence $(\ind_{F_n})_n$ for some sequence $(F_n)_n$ of
measurable sets. From here it follows that 
the weak closure of $\mathscr{G}_{p,2}$ contains all the simple functions. Indeed, consider a simple function
$
f=\sum\limits_{j=1}^\ell a_j \ind_{A_j},
$
with $\ell\geq 1$ and $a_j\in \RR$ for all $j\in\{1,...,\ell\}$. By adding a large constant $C$, we have $f+C=\sum\limits_{j=1}^\ell b_j \ind_{A_j}$, where 
$b_j=a_j+C>0$ for all $j\in\{1,...,\ell\}$. Letting $D=\max_{1\leq j\leq l}b_j$, we deduce that $\frac{1}{D}(f+C)$ is the weak limit of a sequence
$(\ind_{F_n})_n$ for some sequence of measurable sets $(F_n)_n$. Then
$$
f_n:=D\ind_{F_n}-C=(D-C)\ind_{F_n}-C\ind_{F_n^c}\in \mathscr{G}_{p,2},
$$
converges weakly to $f$. The density of the simple functions in $L^p$, in the strong topology, shows the result.
\end{proof}

The previous result implies obviously that $\mathscr{G}_{p,k}$ is not weakly closed in general, and the usual 
optimization methods do not work in this context, we have to find a minimizer by a more constructive way.\\

\begin{definition}\label{eq:M_p} 
In what follows, for a measurable set $A$ of positive and finite measure, we consider $\M_p(f,A)$
as one of the \textit{$p$-th means} of $f$ on $A$ where $p\in[1,\infty)$. The function 
$$
a\mapsto \int_{A} |f(x)-a|^p \, d\mu(x)
$$ 
is convex, nonnegative and finite on $\RR$, which converges to $\infty$ as $a\to \pm \infty$. 
Therefore, this function has at least one global minimum. For $p=1$, the set of minima is a bounded interval 
with extremes $a^*$ and $b^*$ and it is customary to take, the median, as 
$$
\M_1(f,A)=\frac{a^*+b^*}{2}.
$$
For $p>1$ the minimum is unique due to strict convexity and we denote it by $\M_p(f,A)$. For example, for $p=2$
$$
\M_2(f,A)=\frac{1}{\mu(A)} \int_A f(x) \, d\mu(x),
$$
is the mean of $f$ over the set $A$. If a set has measure $0$, we simply put $\M_p(f,A)=0$.
\end{definition}

The next concept will play an important role in what follows.

\begin{definition}
\label{def:special}
Assume $f\in L^p(\Omega,\F,\mu)$, $p\in [1,\infty)$. A function $g\in \mathscr{G}_{p,k}$ 
$$
g=\sum\limits_{i=1}^{q} a_i \ind_{C_i},
$$
with $1\le q\le k$, is said in \textit{$f$-special form} if there exist $-\infty\le r_1<...<r_k<r_{k+1}\le \infty$ such that
\begin{itemize}
    \item $C_i=f^{-1}([r_{i},r_{i+1}))$ for all $i\in\{1,...,q-1\}$, $C_q=f^{-1}([r_{q},r_{q+1}])$ and $\{C_i\}_{1\leq i\leq q}$ is a partition of $\Omega$;
    \item $-\infty<a_1<...<a_q<\infty$;
    \item for all $i\in\{1,...,q\}$ such that $\mu(C_i)<\infty$, it holds $a_i$ is a $p$-th mean of $f$ on $C_i$.
\end{itemize}
\end{definition}

Suppose that $g=\sum\limits_{i=1}^{q} a_i \ind_{C_i}$ is in $f$-special form. Note that if $\mu$ is an infinite measure there exists a unique $1\le s\le q$ such that $a_s=0$ and $\mu(C_i)<\infty$ 
for all $i\neq s$. We also have that $a_i=\M_p(f,C_i)$ for all $i\in\{1,...,q\}$ if $p>1$. Moreover notice that $g=h\circ f$, where $h=\sum_{i=1}^{q-1} a_i \ind_{[r_i,r_{i+1})}+a_q \ind{[r_q,r_{q+1}]}$ is a Borel function
and $g$ is $f$-measurable, that is, $g$ is measurable with respect $\sigma(f)=f^{-1}(\B)$, where $\B$ is the Borel $\sigma$-field in $\RR$.

\subsection{The case of a finite measure, $p\in[1,\infty)$}
\label{sec:finite measure}

If the measure if finite, we start by proving that there exists an approximation sequence which is uniformly bounded:

\begin{lemma}\label{unif_bounded}
Let $(\Omega,\F,\mu)$ be a finite measure space and $p\in [1,\infty)$. Let $f\in L^p(\Omega,\F,\mu)$ and $k\ge 1$. 
Then there exists a uniformly bounded sequence $(g_n)_n\subset \mathscr{G}_{p,k}$ such that $$\|f-g_n\|_p\to \mathscr{D}_{p,k}(f).$$
\end{lemma}

\begin{proof}
Let $(h_n)_n\in \mathscr{G}_{p,k}$ be a sequence such that $\|f-h_n\|_p\to \mathscr{D}_{p,k}(f)$. Assume that 
$h_n=\sum_{i=1}^{m(n)} c_{i,n} \ind_{A_{i,n}}$, where $(c_{i,n})_{1\leq i\leq m(n)}$ are all different, $\{A_{i,n}\}_{1\leq i\leq m(n)}$
is a measurable partition with sets of positive measure and $m(n)\le k$. We assume that $m(n)=m$ is constant by passing to a subsequence 
if necessary. We modify this
approximating sequence by considering $a_{i,n}=\M_p(f,A_{i,n})$ any of the $p$-th means of $f$ in $A_{i,n}$. By definition
of the $p$-th means we have, for all $i\in\{1,...,m\}$
$$
\int_{A_{i,n}} |f(x)-a_{i,n}| \ d\mu(x)\le \int_{A_{i,n}} |f(x)-c_{i,n}| \ d\mu(x),
$$
showing that $\tilde h_n=\sum_{i=1}^{m} a_{i,n} \ind_{A_{i,n}}\in \mathscr{G}_k$ is a minimizing sequence since
$$
\mathscr{D}_{p,k}(f)\le \|f-\tilde h_n\|_p\leq\|f- h_n\|\to \mathscr{D}_{p,k}(f)
$$
If $m<k$, we define $a_{i,n}=0$ and $A_{i,n}=\emptyset$ for $i\in\{m+1,...,k\}$. We assume 
that $\{A_{i,n}\}_{1\leq i\leq k}$ are ordered in decreasing order according to their measure
$$
\mu(A_{1,n})\ge \mu(A_{2,n})\ge ... \ge \mu(A_{k,n})\ge 0.
$$
In this way, the vector $v_n=(\mu(A_{1,n}),\mu(A_{2,n}), ...,\mu(A_{k,n}))$ belongs to the compact set in $\RR^k$
$$
\Delta=\left\{x\in \RR^k:\ x_1\ge x_2\ge ... \ge x_k\ge 0,\\ \sum_i x_i=\mu(\Omega)\right\}
$$
By passing to a subsequence if necessary, we can assume that $(v_n)_n$ converges to some
vector $v=(v_1,v_2,...,v_k)\in \Delta$. If $q$ is the largest index such that $v_q>0$
($q$ could be exactly $k$) then, we have $q\ge 1$ and $v_1\ge ... \ge v_{q}>0= v_{q+1}=...=v_k.$
We notice that $q\le m$. 
Now, define $B_n=\bigcup\limits_{i=q+1}^k A_{i,n}$ for all $n\in\mathbb N$, that we take as the empty set if $q=k$, so
$$
\lim\limits_{n\to \infty} \mu(B_n)=\lim\limits_{n\to \infty} \sum\limits_{i=q+1}^k \mu(A_{i,n})=0.
$$
On the other hand, for all $i\in\{1,...,q\}$ we have
$$
\lim\limits_{n\to \infty} \mu(A_{i,n})=v_i>0,
$$
and so, passing to a further subsequence we can assume there exists a finite constant $\Gamma$ such that for all $n$
and all $i\in\{1,...,q\}$ it holds
\begin{equation}
\label{eq:boundC}
\frac{1}{\mu(A_{i,n})}\le \Gamma
\end{equation}
The finite measure $\nu$ defined by
$$
\nu(A)=\int_A |f(x)|^p \ d\mu(x),
$$
is absolutely continuous with respect to $\mu$, which means that, for all $\rho>0$ there exists a
$\delta>0$ such that, for any measurable set $A$ if $\mu(A)\le \delta$ then 
$\nu(A)=\int_A |f(x)|^p \ d\mu(x)\le \rho$. This property
shows that
$$
\lim\limits_{n\to \infty} \int_{B_n} |f(x)|^p \ d\mu(x)=0.
$$
Now, we modify further the approximation sequence by defining
\begin{equation}
\label{eq:def_b_in}
b_{i,n}=\begin{cases}  \vspace{0.1cm} a_{i,n} &\hbox{for } i\in\{1,...,q\}\\
                        0       &\hbox{for } i\in\{q+1,...,k\}
        \end{cases},
\end{equation}
and define
\begin{equation}
\label{eq:newg_n}    
g_n=\sum\limits_{i=1}^k b_{i,n} \ind_{A_{i,n}}=\sum\limits_{i=1}^q \M_p(f,A_{i,n}) \ind_{A_{i,n}}
+ 0 \ind_{B_n}\in \mathscr{G}_{p,k}.
\end{equation}
We need to show that $(g_n)_n$ is a good approximation sequence and it is uniformly bounded. For the 
first claim notice that for $i\in\{q+1,...,k\}$, we have 
$$
\int_{A_{i,n}} |f(x)-\tilde h_n(x)|^p\ d\mu(x)=\int_{A_{i,n}} |f(x)-\M_p(f,A_{i,n})|^p\ d\mu(x)
\le \int_{A_{i,n}} |f(x)|^p\ d\mu(x),
$$
where we have used the optimality of $\M_p(f,A_{i,n})$ in the last inequality. This shows that
$$
\begin{array}{ll}
\mathscr{D}_{p,k}(f)^p &\hspace{-0.2cm}\le \|f-\tilde h_n\|_p^p=\sum\limits_{i=1}^{m} \int_{A_{i,n}} |f(x)-\M_p(f,A_{i,n})|^p\ d\mu(x)\\
&\hspace{-0.2cm}\le \sum\limits_{i=1}^q \int_{A_{i,n}} |f(x)-\M_p(f,A_{i,n})|^p\ d\mu(x)+
\int_{B_n} |f(x)|^p \ d\mu(x)=\|f-g_n\|_p^p\\
&\hspace{-0.2cm}\le \sum\limits_{i=1}^{m} \int_{A_{i,n}} |f(x)-\M_p(f,A_{i,n})|^p\ d\mu(x)+\int_{B_n} |f(x)|^p \ d\mu(x)\\
&\hspace{-0.2cm}\le \|f-\tilde h_n\|_p^p+\int_{B_n} |f(x)|^p \ d\mu(x)\vspace{0.1cm}\to \mathscr{D}_{p,k}(f)^p
\end{array}
$$
Now, we prove that $(g_n)_n$ is uniformly bounded. We notice that $g_n=0$ on $B_n$, so we must
study $g_n$ on $B_n^c$. For $i\in\{1,...,q\}$ and $x\in A_{i,n}$ we have $g_n(x)=\M_p(f,A_{i,n})$
and so
$$
\|\M_p(f,A_{i,n})\ind_{A_{i,n}}\|_p\le
\|(f-\M_p(f,A_{i,n}))\ind_{A_{i,n}}\|_p+\|f\ind_{A_{i,n}}\|_p
\le 2\| f \ind_{A_{i,n}}\|_p\le 2 \|f\|_p,
$$
where we have used again the optimality of $\M_p(f,A_{i,n})$.
This shows that
$$
|\M_p(f,A_{i,n})| \le 2 \frac{\|f\|_p}{\mu(A_{i,n})^{\frac{1}{p}}}\le 2\|f\|_p \Gamma^{\frac{1}{p}},
$$
where $\Gamma$ is the constant obtained in \eqref{eq:boundC}.
\end{proof}

The next result proves that $\mathscr{G}_{p,k}$ is proximinal in case of finite measure spaces. Remember that $P_K$ is the metric projection over $K$.

\begin{theorem}
\label{pro:minimum}
Let $(\Omega,\F,\mu)$ be a finite measure space, $p\in [1,\infty)$ and $k\geq 1$. Then $\mathscr{G}_{p,k}$ is proximinal.

Moreover, if $f\in L^p(\Omega,\F,\mu)$ and $g=\sum\limits_{i=1}^q b_i \ind_{A_i}\in P_{\mathscr{G}_{p,k}}(f)$ is a minimizer with 
$q\le k$, $-\infty<b_1<...<b_q<\infty$ and $\{A_i\}_{1\leq i\leq q}$ a partition of $\Omega$ with sets of positive measure, there exists a minimizer 
$\widetilde g\in P_{\mathscr{G}_{p,q}}(f)$ in $f$-special form:  
$$
\widetilde g=\sum\limits_{i=1}^{q} \M_p(f,f^{-1}(C_i))\ \ind_{f^{-1}(C_i)}
$$
where 
\begin{itemize}
    \item $r_1=-\infty, r_{q+1}=\infty$ and  $r_i=\frac{b_{i-1}+b_{i}}{2}$ for all $i\in\{2,...,q\}$;
    \item $C_i=f^{-1}([r_i,r_{i+1}))$ for all $i\in\{1,...,q-1\}$ and $C_q=f^{-1}([r_q,r_{q+1}])$;
    \item $b_i$ is a $p$-th mean of $f$ on $f^{-1}(C_i)$ for all $i\in\{1,...,q\}$ such that $\mu(f^{-1}(C_i))>0$.
\end{itemize}
If $q$ is the smallest among all minimizers, then $\mu(C_i)>0$ for all $i\in\{1,...,q\}$.
\end{theorem}

\begin{proof} 
By Lemma \ref{unif_bounded}, let $(g_n)_n\subset \mathscr{G}_{p,k}$ be a uniformly bounded sequence such that $\|f-g_n\|_p\to \mathscr{D}_{p,k}(f)$. Let $C>0$ 
such that $|g_n|<C$ for all $n\in\mathbb N$. We write $g_n=\sum\limits_{i=1}^k b_{i,n} \ind_{A_{i,n}}$ where $\{A_{i,n}\}_{1\leq i\leq k}$ is a 
partition of $\Omega$ and $-C\le b_{1,n}\le...\le b_{k,n}\le C$. The vector $u_n=(b_{1,n},...,b_{k,n})$ belongs to the compact set $[-C,C]^k$ and therefore, 
by taking a subsequence if necessary, we can assume that $(u_n)_n$ converges to some $u=(b_1,...,b_k)\in [-C,C]^k$ with $b_1\le ...\le b_k$. 
Some of the entries in $u$ can be equal, for
that we consider $z_1<...<z_l$ the distinct entries in $u$ where $1\le l \le k$.
We define
$r_1=-\infty, r_{l+1}=\infty$ and $r_j=\frac{z_{j-1}+z_j}{2}$ for $j\in\{2,...,l\}$. Consider the intervals $I_j=[r_j,r_{j+1})$ for
$j\in\{1,...,l-1\}$ and $I_l=[r_l,r_{l+1}]$. For $j\in\{1,...,l\}$, we also define $L_j=\{i\in\{1,...,k\} :\ b_i=z_j\}$, which is  a 
partition of $\{1,...,k\}$. For all $n\in\mathbb N$, consider the function 
$$
\tilde g_n=\sum_{i=1}^k b_i \ind_{A_{i,n}}.
$$
Then, we have 
$$
\begin{array}{ll}
\|f-\tilde g_{n}\|_p&\hspace{-0.2cm}\le\|f-g_{n}\|_p+\|g_{n}-\tilde g_{n}\|_p\le \|f-g_{n}\|_p+
\max_{1\leq i\leq k}|b_{i,n}-b_{i}| \mu(\Omega)^\frac{1}{p}\to \mathscr{D}_{p,k}(f),
\end{array}
$$
proving that $(\tilde g_{n})_{n}$ is also a minimizing sequence. Finally, our candidate for minimizer is the function $g=\sum_{j=1}^l z_j \ind_{f^{-1}(I_j)}\in\mathscr{G}_{p,k}$. 
For all $i\in\{1,...,k\}$, all $j\in\{1,...,l\}$ and all $n$, we
have
$$
\int_{f^{-1}(I_j)\cap A_{i,n}} |f(x)-z_j|^p \ d\mu(x)
\le \int_{f^{-1}(I_j)\cap A_{i,n}} |f(x)-b_i|^p \ d\mu(x).
$$
This is clear if $i\in L_j$ because in that case $z_j=b_i$. Now, if $i\in L_{j'}$ with $j'\neq j$, we have $b_i=z_{j'}$ and
for all $x\in f^{-1}(I_j)$ it holds $|f(x)-z_j|\le |f(x)-z_{j'}|=|f(x)-b_i|$. Now, summing over $i,j$ we get for all $n$ that
$$
\begin{array}{ll}
\mathscr{D}_{p,k}(f)^p\le \|f-g\|_p^p&\hspace{-0.2cm}=\sum\limits_{i,j} \int_{f^{-1}(I_j)\cap A_{i,n}} |f(x)-z_j|^p \ d\mu(x)\le
\sum\limits_{i,j} \int_{f^{-1}(I_j)\cap A_{i,n}} |f(x)-b_i|^p \ d\mu(x)\\
&\hspace{-0.2cm}\le \|f-\tilde g_{n}\|_p^p\to \mathscr{D}_{p,k}(f)^p,
\end{array}
$$
proving that $g\in P_{\mathscr{G}_{p,k}}(f)$.

Now, we prove the last part of the Theorem. 
Assume that $g=\sum_{i=1}^q b_i\ind_{A_i}\in P_{\mathscr{G}_{p,k}}(f)$ is a minimizer,  
with $b_1<...<b_q$, $\{A_i\}_{1\leq i\leq q}$ a 
partition of $\Omega$ where all the sets $A_i$ have positive measure and $q\le k$.
Let $r_1=-\infty, r_{q+1}=\infty, r_i=\frac{b_{i-1}+b_i}{2}$ for $i=\{2,...,q\}$ and 
$$
C_i=f^{-1}([r_i,r_{i+1}))\ \text{for}\ i\in\{1,...q-1\},\ C_q=f^{-1}([r_i,r_{i+1}]).
$$
For all $i\in\{1,...,q\}$, we modify the sets $A_i$ as
\begin{equation}
\label{eq:especial_minimizer}
\widetilde{A}_i=\left(A_i\cup f^{-1}(\{r_i\})\right)\setminus 
f^{-1}(\{r_{i+1}\}).
\end{equation}
Let us prove that $\mu(\widetilde{A}_1 \Delta C_1)=\mu(A_j\cap E_2)=0$ for all $j>2$ where $E_2=f^{-1}(\{r_2\})$. 
Define 
$$
g'=b_1\ind_{\widetilde{A}_1}+\sum_{i=3}^q b_i\ind_{A_i\setminus E_2}+b_2\ind_{A_2\cup E_2}\in \mathscr{G}_{p,k}
$$ 
and note that $\{\widetilde{A}_1,\{A_j\setminus E_2\}_{j>2}, A_2\cup E_2\}$ is a partition of $\Omega$.
Consider the following decomposition
$$
\begin{array}{l}
\|f-g\|_p^p=\sum\limits_j \int_{A_j} |f(x)-b_j|^p\ d\mu(x)=\int_{\widetilde{A}_1} |f(x)-b_1|^p\ d\mu(x)+
\sum\limits_{j> 2} \int_{A_j\setminus E_2} |f(x)-b_j|^p\ d\mu(x)\\
+\sum\limits_{j>2} \int_{A_j\cap E_2} |f(x)-b_j|^p\ d\mu(x)+
\int_{A_1\cap E_2} |f(x)-b_1|^p \ d\mu(x)+\int_{A_2} |f(x)-b_2|^p \ d\mu(x)
\\
\ge \int_{\widetilde{A}_1} |f(x)-b_1|^p\ d\mu(x)
+\sum\limits_{j>2} \int_{A_j\setminus E_2} |f(x)-b_j|^p\ d\mu(x)+\int_{A_2\cup E_2} |f(x)-b_2|^p \ d\mu(x)\\
+ (b_3-b_2)^p\sum\limits_{j>2}\mu(A_j\cap E_2)\\
=\|f-g'\|_p^p+(b_3-b_2)^p\sum\limits_{j>2}\mu(A_j\cap E_2),
\end{array}
$$
where the second equality follows from the fact that $\widetilde{A}_1=A_1\setminus E_2$ (up to a set of measure zero). The inequality 
is proved noting that, for $x\in A_j\cap E_2$ with $j>2$ it holds $|f(x)-b_2|=b_2-r_2<b_3-r_2\leq b_j-r_2=|f(x)-b_j|$, which implies 
$|f(x)-b_j|\geq|f(x)-b_2|+b_3-b_2$, and for $x\in A_1\cap E_2$ it holds $|f(x)-b_1|=|f(x)-b_2|$.
So, since $g$ is a minimizer we deduce that
$\mu(A_j\cap E_2)=0$ for all $j>2$. Thus, we get
$$
\|f-g\|_p^p=\int_{\widetilde{A}_1} |f(x)-b_1|^p\ d\mu(x)+\sum\limits_{j>2} \int_{A_j} |f(x)-b_j|^p\ d\mu(x)+
\int_{A_2\cup f^{-1}(\{r_2\})} |f(x)-b_2|^p \ d\mu(x),
$$
showing that 
$$
b_1\ind_{\widetilde{A}_1}+b_2\ind_{A_2\cup f^{-1}(\{r_2\})} +\sum_{j>2} b_j \ind_{A_j}\in P_{\mathscr{G}_{p,k}}(f).
$$
A similar argument shows that $\mu(A_j\cap C_1)=0$ and $\mu(\widetilde{A}_1 \cap C_j)=0$ for all $j\ge 2$.
Since $\{A_i\}_{1\leq i\leq q}$ is a partition we conclude that $\mu(C_1)=\sum_i \mu(A_i\cap C_1)=\mu(C_1\cap A_1)=\mu(C_1\cap \widetilde{A}_1)$, 
proving that $C_1\subset \widetilde{A}_1$ except for a set of measure $0$. On the other hand, using again that 
$\{\widetilde{A}_1, A_2\cup f^{-1}(\{r_2\}), \{A_j\setminus f^{-1}(\{r_2\})\}_{j>2}\}$ is also a partition,
we conclude that $\widetilde{A}_1 \subset C_1$ except for a set of measure $0$. 
In a similar way, we prove $\mu(\widetilde{A}_i\cap C_j)=\mu(\widetilde{A}_i\Delta C_i)=0$, for all $i\neq j$. 

At this point we should
mention that some of the $\tilde A_i$ could have measure $0$. For example this occurs if $A_1=f^{-1}(\{r_2\})$.
In any case, we have
$$
\|f-g\|_p^p=\sum\limits_{i=1}^q \int_{\widetilde{A}_i} |f(x)-b_i|^p\ d\mu(x)=
\sum\limits_{i=1}^q \int_{C_i} |f(x)-b_i|^p\ d\mu(x),
$$
showing that
$$
\hat g=\sum\limits_{i=1}^q b_i \ind_{C_i}\in P_{\mathscr{G}_{p,k}}(f),
$$
is a minimizer.
On the other hand, if $\mu(C_i)>0$ we have 
$\int_{C_i} |f(x)-\M_p(f,C_i)|^p \ d\mu(x)\le \int_{C_i} |f(x)-b_i|^p \ d\mu(x)$. The inequality
cannot be strict, otherwise we contradict the minimality of $\hat g$, showing that
$b_i$ is a $p$-th means of $f$ on $C_i$, and therefore,
$$
\widetilde g= \sum\limits_{i=1}^q \M_p(f,C_i) \ind_{C_i}\in P_{\mathscr{G}_{p,k}}(f),
$$
is a minimizer in $f$-special form, as we wanted to prove. 
In case that $q$ is the minimal among all minimizers, we conclude that $\mu(C_i)>0$ for all $i$.
\end{proof}

\begin{remark}
In the last part of the Theorem, for any minimizer $g$, we have constructed a minimizer $\widetilde g$ in
$f$-special form, but it may happens that some of the sets $(C_i)_i$ have measure $0$, which can be 
discarded to get a minimizer with fewer terms. An interesting question is if this procedure applied 
to any minimizer gives always a minimizer with the smallest possible number of terms (see Proposition \ref{pro:f_inG_k}).
\end{remark}

Recall that given $f\in L^p(\Omega,\F,\mu)$, the distribution of $f$ is the measure $\mu_f$ defined on $(\RR,\B)$ given by, for all $B\in \B$
$$
\mu_f(B)=\mu(f^{-1}(B)).
$$
Let $g$ be a minimizer of $f$ in $\mathscr{G}_{p,k}$ in  $f$-special form provided by Theorem \ref{pro:minimum} 
$$
g=\sum\limits_{i=1}^q a_i \ind_{f^{-1}([r_i,r_{i+1}))}, 
$$
So, $g=\ell\circ f$ with
$$
\ell=\sum\limits_{i=1}^q a_i \ind_{[r_i,r_{i+1})},
$$
and
$$
\begin{array}{ll}
\|f-g\|_p^p&\hspace{-0.2cm}=\int_\Omega |f(x)-g(x)|^p \ d\mu(x)=\int_\Omega |f(x)-\ell(f(x))|^p \ d\mu(x)=\int_\RR |y-\ell(y)|^p \ d\mu_f(y)\\
\\
&\hspace{-0.2cm}=\|\hbox{id}-\ell\|_{L^p(\RR,\B,\mu_f)}^p.
\end{array}
$$
Thus, the problem of finding a minimizer for $f$ is equivalent to find a minimizer for the identity function $\hbox{id}$ in $\mathscr{G}_{p,k}(\RR,\B,\mu_f)$.
The following result shows that when $\mu_f$ is continuous, this search can be done over the subclass of simple functions in $f$-special form. 
Before stating the result, let us fix some notations. The cumulative distribution associated to $\mu_f$ is the function $F_f(x)=\mu_f((-\infty,x])$ 
Notice that $F_f(-\infty)=0$ and $F_f(\infty)=\mu(\Omega)$. The convex support of $\mu_f$ is the interval $[\mathfrak{a}_f,\mathfrak{b}_f]$, where
$$
\mathfrak{a}_f=\sup\{z:\ F_f(z)=0\}, \mathfrak{b}_f=\inf\{z:\ F_f(z)=F_f(\infty)\}.
$$

The following lemma is needed to study the uniqueness of minimizers, where $p$-th means are characterized as roots of certain equations, 
suitable for our purposes. We include a proof, inspired by exercise 1.4.23 in \cite{Stroock}, for the sake of completeness.

\begin{lemma}
\label{characterization_mean}
Let $(\Omega,\F,\mu)$ be a finite measure space and $f\in L^p(\Omega,\F,\mu)$, for $p\in [1,\infty)$. 
For $p=1$, we also assume that $F_f$ is continuous and strictly increasing on $[\mathfrak{a}_f,\mathfrak{b}_f]$. 
Let $I\subset\RR$ be an interval with extremities $c,d \in \overline{\RR}$
such that $\mu_f(I)>0$.  Then, the $p$-th mean
$m=\M_p(f,f^{-1}(I),\mu)=\M_p(\hbox{id},I,\mu_f)$ is characterized as the unique solution of the equation
\begin{equation}
\label{eq:pmean}
\int_{I\cap (-\infty,m]} (m-x)^{p-1} \ d\mu_f(x)=\int_{I\cap (m,\infty)} (x-m)^{p-1} \ d\mu_f(x),
\end{equation}
which for $p=1$ is equivalent to 
\begin{equation}
\label{eq:1-mean}
F_f(m)-F_f(c)=\frac12(F_f(d)-F_f(c)).
\end{equation}
\end{lemma}

\begin{proof} 
The proof is based on the following equality for all $x,b\in\RR$
$$
|x-b|^p=p\int_{-\infty}^b (t-x)^{p-1} \ind_{\{x\le t\}} \ dt+p \int_b^\infty (x-t)^{p-1} \ind_{\{t< x\}} \ dt,
$$
which implies that
$$
|x-b|^p-|x-a|^p=p\int_{a}^b \left((t-x)^{p-1} \ind_{\{x\le t\}} -(x-t)^{p-1} \ind_{\{t< x\}}\right)\ dt.
$$
Fix $m\in\mathbb R$. Define a function $L:\RR\to\RR$ by 
$$
L(b)=\int_I |x-b|^p \ d\mu_f(x)-\int_I |x-m|^p \ d\mu_f(x).
$$ 
It is clear that $m$ is a minimum of $L$ if and only if $m$ is a $p$-th mean. 
Using Fubini's Theorem and the previous equality, we obtain that for all $b\in\RR$
$$
L(b)=p\int_m^b \int_{I\cap (-\infty,t]} (t-x)^{p-1} \ d\mu_f(x) dt-
p\int_m^b \int_{I\cap (t,\infty)}  (x-t)^{p-1}  d\mu_f(x) dt
$$ 
Note that $L$ is convex, coercive and continuous and then reaches a minimum.

Suppose $p>1$. The functions 
$t \mapsto \int_{I\cap [-\infty,t]} (t-x)^{p-1} \ d\mu_f(x)$ and $t \mapsto \int_{I\cap (t,\infty)}  (x-t)^{p-1} \ d\mu_f(x)$
are continuous, and therefore $L$ is strictly convex and 
continuously differentiable, which proves that $L'(b)=0$ is the equation for the unique minima, that is,
$$
L'(b)=p\int_{I\cap (-\infty,b]} (b-x)^{p-1} \ d\mu_f(x)-
p\int_{I\cap (b,\infty)} (x-b)^{p-1} \ d\mu_f(x)= 0.
$$
It follows that $m$ is the $p$-th mean if and only if $m$ fulfills \eqref{eq:pmean}. 

For $p=1$, using that $F_f$ is continuous, we have 
$$
L(b)=\int_m^b F_f(t)-F_f(c)-(F_f(d)-F_f(t)) \ dt.
$$
Again, since $F_f$ is continuous we obtain that $L$ is continuously differentiable. Then, if $b$ is any minima for $L$, it holds that
$L'(b)=0$, that is, $F_f(b)-F_f(c)=\frac{1}{2}(F_f(d)-F_f(c))$. Since $F_f$ is assumed to be 
strictly increasing, this equation has a unique solution, and then $L$ has exactly one minimum. Then, $m$ is a $1$-th if and only if $F_f(m)-F_f(c)=\frac{1}{2}(F_f(d)-F_f(c))$.
\end{proof}

Notice that in the previous Lemma we can replace $(-\infty,m]$ by $(-\infty,m)$ and
$(m,\infty)$ by $[m,\infty)$ in \ref{eq:pmean}, because $x=m$ does not add to the integrals. 
In the case $F_f$ is just increasing, $I=(c,d]$ and $p=1$, all the $1$-th means satisfy the equations $L^{'+}(m)\ge 0$ and $L^{'-}(m)\le 0$, 
which are equivalent to
$$
F_f(m)-F_f(c)\ge \frac{1}{2}(F_f(d)-F_f(c)), \ F_f(m-)-F_f(c)\le \frac{1}{2}(F_f(d)-F_f(c)),
$$
and the solution set is, in general, an interval.\\

The next result shows that when $F_f$ is continuous all minimizers are in $f$-special form.

\begin{corollary}
\label{cor:all_specialform}
Let $(\Omega,\F,\mu)$ be a finite measure space, $p\in [1,\infty)$ and $k\geq 1$. Let $f\in L^p(\Omega,\F,\mu)$ and assume that $F_f$ is continuous. 
Then any minimizer $g\in P_{\mathscr{G}_{p,k}}(f)$ is of the $f$-special form 
\begin{equation}
\label{eq:special}    
g=\sum\limits_{i=1}^k a_i \ind_{f^{-1}([r_i,r_{i+1}))},
\end{equation}
where 
\begin{itemize}
    \item $[r_i,r_{i+1})$ has positive $\mu_f$-measure for all $i\in\{1,...,k\}$;
    \item $\mathfrak{a}_f=r_1<...<r_k<r_{k+1}=\mathfrak{b}_f$ and $r_i=\frac{a_i+a_{i+1}}{2}$ for all $i\in\{2,...,k\}$;
    \item $a_i$ is a $p$-th mean of $\hbox{id}$ on $[r_i,r_{i+1})$ under $\mu_f$ for all $i\in\{1,...,k\}$. Moreover, if $F_f$ 
    is strictly increasing on $[\mathfrak{a}_f,\mathfrak{b}_f]$, then 
    $a_i=\M_p(\hbox{id},[r_i,r_{i+1}),\mu_f)$ for all $i\in\{1,...,k\}$.
\end{itemize}
\end{corollary}

\begin{proof} Notice first that $f\notin \mathscr{G}_{p,k}$, because the image of $f$ cannot be a finite set a.e., since $\mu_f$ is not atomic. This implies 
that there is no minimizer in $\mathscr{G}_{p,q}$, with $q<k$  (see Proposition \ref{pro:f_inG_k} below). So, any minimizer has the structure
$$
g=\sum\limits_{i=1}^k a_i \ind_{A_i}
$$
where $a_1<...<a_k$, $\{A_i\}_{1\leq i\leq k}$ is a partition with sets of positive measure and $a_i$ is a $p$-th mean of $f$ in $A_i$, for all $i\in\{1,...,k\}$. 
In the previous proof, we then modify this minimizer to get one in $f$-special
form. If one goes over that proof and using the fact that $F_f$ is continuous, one realizes that in equation \eqref{eq:especial_minimizer}, we get $\widetilde A_i=A_i$ a.e. and then $A_i=C_i$ a.e.,
proving that $A_i=f^{-1}([r_i,r_{i+1}))$ a.e. 

The fact that $a_i$ is a $p$-th mean is just the fact that $g$ is a minimizer. For $p>1$, the uniqueness of the $p$-th mean shows that 
$a_i=\M_p(\hbox{id},[r_i,r_{i+1}),\mu_f)$. This is also true for $p=1$, when $F_f$ is continuous and strictly increasing in $[\mathfrak{a}_f,\mathfrak{b}_f]$
(see Lemma \ref{characterization_mean}).
\end{proof}

\begin{remark} The previous result could be used as the basis of an algorithm to approximate a minimizer. Assume that $\mu_f$ is a continuous
distribution. For any $s\in \RR$, were $s$ plays the role of $r_2$
in the representation \eqref{eq:special}, we define $r_1(s)=-\infty, r_2(s)=s$ and $a_1=a_1(s)=\M_p(\hbox{id},(-\infty,s),\mu_f)$. Then, we define
$a_2(s)=2r_2(s)-a_1(s)$, which is a relation that should satisfy any minimizer. Then, compute $r_3(s)$ so that
$$
a_2(s)=\M_p(\hbox{id},[r_2(s),r_3(s)),\mu_f).
$$
and continue in this way defining $a_3(s),r_4(s),..., a_k(s), r_{k+1}(s)$. It may happens that at some iteration $r_{i+1}(s)$ 
is not well defined for some $i\le k-1$ because, for all $t\in [r_i(s),\infty]$
$$
a_i(s)>\M_p(\hbox{id},[r_i(s),t),\mu_f),
$$
which shows that there is no minimizer starting with $r_2=s$. So, we say $s$ is admissible if $r_{k+1}(s)$ is well defined. 
For every admissible $s$ we have a candidate
$$
\ell_s=\sum\limits_{i=1}^{k-1} a_i(s) \ind_{[r_i(s),r_{i+1}(s))}+ \M_p(\hbox{id},[r_k(s),\infty),\mu_f)\ind_{[r_k(s),\infty)}
$$
and we can compute $R(s)=\int |x-\ell_s(x)|^p \ d\mu_f(x)$. For $s$ which is not admissible put $R(s)=\infty$. Then a minimizer of $R$ gives a minimizer
for $f$. One expects that the set of admissible values of $s$ is an interval. We shall work on this algorithm in a forcoming paper.

For example, if $\mu_f$ is a normal $\mathcal N(0,1)$, $p=2$ and $k=3$, this algorithm gives the following approximation 
$$
h\approx -1.2\ \ind_{(-\infty,-0.6)}+0\ \ind_{[-0.6,0.6)}+1.2\ \ind_{[0.6,\infty)}.
$$
Notice that $\|\hbox{id}\|_2=1$ and $(\mathscr{D}_{2,3}(\hbox{id}))^2\approx 0.18$, which means that, in the language of statistics, $82\%$ of the variance of $f$ 
is explained by a simple function taking $3$ values.
\end{remark}

Uniqueness of minimizers is a much harder problem. Here, we present a partial result in case $\mu_f$ satisfies a certain monotone likelihood ratio
property.

\begin{theorem} 
\label{pro:uniqueness}
Let $(\Omega,\F,\mu)$ be a finite measure space, $p\in [1,\infty)$ and $f\in L^p(\Omega,\F,\mu)$. Assume $\mu_f$ has a 
density with respect to the Lebesgue measure $\Psi:(\mathfrak{a}_f,\mathfrak{b}_f)\to (0,\infty)$, which we 
extend by $0$ outside this interval. Consider for $s\in(0,\mathfrak{b}_f-\mathfrak{a}_f)$ 
the function $G_s:(\mathfrak{a}_f, \mathfrak{b}_f)\to [0,\infty)$
given by $G_s(y)=\frac{\Psi(y+s)}{\Psi(y)}$ and assume that $G_s$ is decreasing. Moreover, we suppose that one of the following hypotheses hold:
\begin{itemize}
\item[(H1)] either $\mathfrak{a}_f$ or $\mathfrak{b}_f$ is finite;
\item[(H2)] $G_s((\mathfrak{a}_f,\mathfrak{b}_f))$ is an infinite set;
\item[(H3)] $\Psi$ is continuous.
\end{itemize}
Then there exists a unique minimizer for $f$ in $\mathscr{G}_{p,k}$ for all $k\ge 1$.
\end{theorem}

\begin{proof} The case $k=1$ is direct from the fact that $F_f$ is is strictly increasing (because $\Psi$ is strictly 
positive on $(\mathfrak{a}_f,\mathfrak{b}_f)$) and therefore the $p$-th means are unique. So, we assume that $k\ge 2$. 

According to Corollary \ref{cor:all_specialform}, and since $F_f$ is strictly increasing all minimizers for $f$ have the $f$-special form given in 
\eqref{eq:special}. 
Fix one of them $g=\ell\circ f$, where
$$
\ell=\sum\limits_{i=1}^k a_i \ind_{I_i}
$$
with $r_1=\mathfrak{a}_f, r_{k+1}=\mathfrak{b}_f, r_i=\frac{a_{i-1}+a_i}{2}$ for $i\in\{2,...,k\}$, 
$I_i=[r_i,r_{i+1})$ for $i\in\{1,...,k-1\}$, $I_k=[r_k,r_{k+1}]$, $\mu_f(I_i)>0$ for $i\in\{1,...,k\}$  
and $a_i=\M_p(\hbox{id},I_i,\mu_f)$ for $i\in\{1,...,k\}$. 
Assume there exists another minimizer $\widetilde g=\widetilde \ell \circ f$ with
$$
\widetilde\ell=\sum\limits_{i=1}^{k} \widetilde a_i \ind_{\widetilde I_i}
$$
where $\widetilde r_1=\mathfrak{a}_f, \widetilde r_{k+1}=\mathfrak{b}_f, \widetilde r_i=\frac{\widetilde a_{i-1}+\widetilde a_i}{2}$ for $i\in\{2,...,k\}$, 
$\widetilde I_i=[\widetilde r_i,\widetilde r_{i+1})$ for $i\in\{1,...,k-1\}$, $\widetilde I_k=[\widetilde r_k, \widetilde r_{k+1}]$, 
$\mu_f(\widetilde I_i)>0$ for $i\in\{1,...,k\}$ and 
$\widetilde a_i=\M_p(\hbox{id},\widetilde I_i,\mu_f)$ for $i\in\{1,...,k\}$. We need to prove that $\ell=\widetilde\ell$. 
Consider $s=\widetilde r_2-r_2$. 
Switching $g$ and $\widetilde g$ if necessary, we can suppose without loss of generality that $s\ge 0$, and since 
$r_2,\widetilde r_2 \in (\mathfrak{a}_f,\mathfrak{b}_f)$ then $s<\mathfrak{b}_f-\mathfrak{a}_f$.
Define $\delta_i=\widetilde a_i-a_i$ for all $i\in \{1,...,k\}$ and
$\eta_i=\widetilde r_i-r_i$ for all $i\in \{2,...,k\}$. 

\underline{Case 1:} Suppose that $s=0$. We shall prove that $\ell=\widetilde\ell$. Notice that $I_1=\widetilde I_1$ and $\widetilde a_1=\M_p(\hbox{id},I_1,\mu_f)$ 
and by uniqueness of the $p$-th mean we deduce that $a_1=\widetilde a_1$. But $\widetilde a_1$, $\widetilde a_2$ and $\widetilde r_2$ are related by  
$\widetilde a_2=2\widetilde r_2-\widetilde a_1=2r_2-a_1=a_2$, showing that 
$\widetilde a_2=a_2$. Using the fact that $a_2=\M_p(\hbox{id},I_2,\mu_f)=\M_p(\hbox{id},\widetilde I_2,\mu_f)$ and 
Lemma \ref{characterization_mean}, we have that
\begin{align*}
\int_{a_2}^{\widetilde r_3} (x-a_2)^{p-1}\Psi(x)dx&=\int_{\widetilde a_2}^{\widetilde r_3} (x-a_2)^{p-1}\Psi(x)dx=
\int_{\widetilde r_2}^{\widetilde a_2} (a_2-x)^{p-1}\Psi(x)dx\\
&=\int_{r_2}^{a_2}(a_2-x)^{p-1} \Psi(x)dx=\int_{a_2}^{r_3}(x-a_2)^{p-1}\Psi(x)dx.
\end{align*}
Since $\widetilde r_3\le \mathfrak{b}_f$ and since $\Psi$ is strictly positive, we conclude that $\widetilde r_3=r_3$. Repeating this argument
we conclude that $\widetilde r_i=r_i$ for all $i\in\{2,...,k\}$ and $\widetilde a_i=a_i$ for all $i\in\{1,...,k\}$. Thus $\widetilde \ell=\ell$.

\underline{Case 2:} Suppose that $s>0$ and let us arrive to a contradiction if we suppose that (H1), (H2) or (H3) holds. 
This part will be divided in several steps.

\underline{Step 1:} We are going to show that the following properties hold:
\begin{enumerate}
    \item[(a)] $\delta_{k}\geq\eta_{k}\geq\delta_{k-1}\geq\eta_{k-1}\geq... \geq\eta_2\geq\delta_1$;
    \item[(b)] if one of these inequalities is strict then all the inequalities on the left are also strict;
    \item[(c)] all of these inequalities are in fact equalities if and only if $\mathfrak{a}_f=-\infty$ and
    for all $i\in \{1,...,k-1\}$ it holds
    $$
    \forall y \in (r_i,r_{i+1})\ \ \frac{\Psi(y+s)}{\Psi(y)}=\frac{\Psi(a_i+s)}{\Psi(a_i)}.
    $$
\end{enumerate}
Define a function $\phi:z\mapsto \int_{\mathfrak{a}_f}^{a_1+z}(z+a_1-x)^{p-1}\Psi(x)dx-\int_{a_1+z}^{\widetilde r_2}(x-a_1-z)^{p-1}\Psi(x)dx$. It is clear that $\phi$ is strictly increasing. 
Recall that $\widetilde a_1=\M_p(\hbox{id},[\mathfrak{a}_f,\widetilde r_2),\mu_f)$, which is characterized by 
$$
\int_{\mathfrak{a}_f}^{\widetilde a_1} (\widetilde a_1-x)^{p-1}\Psi(x) \ dx=\int_{\widetilde a_1}^{\widetilde r_2}(x-\widetilde a_1)^{p-1} \Psi(x) \ dx,
$$
so $\phi(\delta_1)=0$. Note also that
\begin{align*}
\phi(0)&=\int_{\mathfrak{a}_f}^{a_1}(a_1-x)^{p-1}\Psi(x)dx-\int_{a_1}^{\widetilde r_2}(x-a_1)^{p-1}\Psi(x)dx \\
&=\int_{a_1}^{r_2}(x-a_1)^{p-1}\Psi(x)dx-\int_{a_1}^{\widetilde r_2}(x-a_1)^{p-1}\Psi(x)dx<0,
\end{align*} 
since $\widetilde r_2-r_2=s>0$. Moreover, we have
\begin{align*}
\int_{\mathfrak{a}_f}^{a_1+s} (a_1+s-x)^{p-1}\Psi(x)dx&\ge  \int_{\mathfrak{a}_f+s}^{a_1+s} (a_1+s-x)^{p-1} \Psi(x)dx=
\int_{\mathfrak{a}_f}^{a_1} (a_1-x)^{p-1} \frac{\Psi(x+s)}{\Psi(x)} \Psi(x)dx\\
&\ge \frac{\Psi(a_1+s)}{\Psi(a_1)} \int_{\mathfrak{a}_f}^{a_1} (a_1-x)^{p-1} \Psi(x)dx=
\frac{\Psi(a_1+s)}{\Psi(a_1)} \int_{a_1}^{r_2} (x-a_1)^{p-1} \Psi(x)dx\\
&\ge \int_{a_1}^{r_2} (x-a_1)^{p-1} \frac{\Psi(x+s)}{\Psi(x)} \Psi(x) \ dx=\int_{a_1+s}^{\widetilde r_2} (x-a_1-s)^{p-1} \Psi(x)dx
\end{align*}
proving that $\phi(s)\geq 0$. It follows that $0<\delta_1\le s=\eta_2$.
The only way that $\delta_1=\eta_2=s$ is that the previous inequalities are only equalities, which means
that $\mathfrak{a}_f=-\infty$ and $\frac{\Psi(y+s)}{\Psi(y)}=\frac{\Psi(a_1+s)}{\Psi(a_1)}$ holds
for all $y \in (\mathfrak{a}_f,r_2)$ $dy$-a.e., but since $G_{s}$ is decreasing this property holds for all $y \in (\mathfrak{a}_f,r_2)$. We summarize
this condition for future reference
\begin{equation}
\label{eq:equality1}
\mathfrak{a}_f=-\infty \hbox{ and } \forall y \in (\mathfrak{a}_f,r_2)\ \ \frac{\Psi(y+s)}{\Psi(y)}=\frac{\Psi(a_1+s)}{\Psi(a_1)}.
\end{equation}
On the other hand, since $\widetilde r_2=\frac{\widetilde a_1+\widetilde a_2}{2}$, we deduce 
$$
\widetilde a_2=2\widetilde r_2-\widetilde a_1=2r_2-a_1+2\eta_2-\delta_1=a_2+2\eta_2-\delta_1,
$$
from where we deduce that $\delta_2=2\eta_2-\delta_1\ge \eta_2$, with equality $\delta_2=\eta_2$ if and only
if $\delta_2=\eta_2=\delta_1=s$.
Now, if there exists $\widetilde r_3\le \mathfrak{b}_f$ such that 
$$
\M_p(\hbox{id},[\widetilde r_2,\widetilde r_3),\mu_f)=\widetilde a_2,
$$
we deduce that $\eta_3\ge \delta_2$. Indeed, this follows from the inequalities 
\begin{align*}
\int_{\widetilde a_2}^{\widetilde r_3} (x-\widetilde a_2)^{p-1}\Psi(x)dx &=
\int_{\widetilde r_2}^{\widetilde a_2} (\widetilde a_2-x)^{p-1}\Psi(x)dx = \int_{r_2+\eta_2}^{a_2+\delta_2}(\widetilde a_2-x)^{p-1} \Psi(x) \ dx  \\
&\geq \int_{r_2+\delta_2}^{a_2+\delta_2} (\widetilde a_2-x)^{p-1}\Psi(x)dx= \int_{r_2}^{a_2}(a_2-x)^{p-1} \Psi(x+\delta_2)dx \\
&=\int_{r_2}^{a_2} (a_2-x)^{p-1}\frac{\Psi(x+\delta_2)}{\Psi(x)} \Psi(x)dx \ge \frac{\Psi(a_2+\delta_2)}{\Psi(a_2)}\int_{r_2}^{a_2} (a_2-x)^{p-1}\Psi(x)dx \\
&=\frac{\Psi(a_2+\delta_2)}{\Psi(a_2)} \int_{a_2}^{r_3} (x-a_2)^{p-1}\Psi(x)dx\ge \int_{a_2}^{r_3} (x-a_2)^{p-1}\frac{\Psi(x+\delta_2)}{\Psi(x)}\Psi(x)dx\\
&=\int_{a_2}^{r_3} (x-a_2)^{p-1}\Psi(x+\delta_2)dx=\int_{a_2+\delta_2}^{r_3+\delta_2} (x-\widetilde a_2)^{p-1}\Psi(x)dx\\
&=\int_{\widetilde a_2}^{r_3+\delta_2} (x-\widetilde a_2)^{p-1}\Psi(x)dx
\end{align*}
proving that $\widetilde r_3\geq r_3+\delta_2$, i.e. $\eta_3\ge \delta_2$. Also, we notice that $\eta_3=\delta_2$ if and only if 
\begin{equation}
\label{eq:equality2}    
\eta_2=\delta_2 \hbox{ and } \forall y \in (r_2,r_3)\ \ \frac{\Psi(y+\delta_2)}{\Psi(y)}=\frac{\Psi(a_2+\delta_2)}{\Psi(a_2)},
\end{equation}
which in particular implies that $\eta_3=\eta_2=\delta_2=\delta_1=s$ and \eqref{eq:equality1} holds. Iterating this idea, we complete Step 1.

\underline{Step 2:} Since $\widetilde g$ has $f$-special form, we have that 
$\widetilde a_k=\M_1(\hbox{id},[\widetilde r_k,\mathfrak{b}_f],\mu_f)$. Since $\delta_k\geq\eta_k$ by $(a)$ in Step 1, we obtain that
\begin{align*}
\int_{\widetilde a_k}^{\mathfrak{b}_f} (x-\widetilde a_k)^{p-1}\Psi(x)dx &=
\int_{\widetilde r_k}^{\widetilde a_k} (\widetilde a_k-x)^{p-1}\Psi(x)dx = \int_{r_k+\eta_k}^{a_k+\delta_k}(\widetilde a_k-x)^{p-1} \Psi(x) \ dx  \\
&\geq \int_{r_k+\delta_k}^{a_k+\delta_k} (\widetilde a_k-x)^{p-1}\Psi(x)dx= \int_{r_k}^{a_k}(a_k-x)^{p-1} \Psi(x+\delta_k)dx \\
&=\int_{r_k}^{a_k} (a_k-x)^{p-1}\frac{\Psi(x+\delta_k)}{\Psi(x)} \Psi(x)dx \ge \frac{\Psi(a_k+\delta_k)}{\Psi(a_k)}\int_{r_k}^{a_k} (a_k-x)^{p-1}\Psi(x)dx \\
&=\frac{\Psi(a_k+\delta_k)}{\Psi(a_k)} \int_{a_k}^{\mathfrak{b}_f} (x-a_k)^{p-1}\Psi(x)dx
\ge \frac{\Psi(a_k+\delta_k)}{\Psi(a_k)}\int_{a_k}^{\mathfrak{b}_f-\delta_k}(x-a_k)^{p-1}\Psi(x)dx\\
&\ge \int_{a_k}^{\mathfrak{b}_f-\delta_k}(x-a_k)^{p-1}\frac{\Psi(x+\delta_k)}{\Psi(x)} \Psi(x)dx
= \int_{a_k+\delta_k}^{\mathfrak{b}_f}(x-\widetilde a_k)^{p-1} \Psi(x)dx\\
&=\int_{\widetilde a_k}^{\mathfrak{b}_f}(x-\widetilde a_k)^{p-1}\Psi(x)dx
\end{align*}
It follows that all the inequalities are in fact equalities and then the following properties hold:
\begin{enumerate}
    \item[(d)] $\mathfrak{b}_f=\infty$;
    \item[(e)] $\eta_k=\delta_k$;
    \item[(f)] $\forall y \in (r_k,\infty)\ \ \frac{\Psi(y+\delta_k)}{\Psi(y)}=\frac{\Psi(a_k+\delta_k)}{\Psi(a_k)}$.
\end{enumerate}
We notice that $(e)$ implies that all inequalities in $(a)$ are equalities and then $(c)$ holds. This together with $(d)$ and $(f)$ implies that
if $~\widetilde g$ is a minimizer then
\begin{enumerate}
    \item[(g)] $\delta_{k}=\eta_{k}=\delta_{k-1}=\eta_{k-1}=...=\eta_2=\delta_1=s$;
    \item[(h)] $\mathfrak{a}_f=-\infty, \mathfrak{b}_f=\infty$ and for all $i\in \{1,...,k\}$ it holds
    $$
    \forall y \in (r_i,r_{i+1})\ \ \frac{\Psi(y+s)}{\Psi(y)}=\frac{\Psi(a_i+s)}{\Psi(a_i)}.
    $$
\end{enumerate}

\underline{Step 3:} Conclusion. Clearly under (H1) or (H2) the function $\widetilde g$ cannot be a minimizer.
It remains to consider that (H3) holds. From $(h)$ and the continuity of $\Psi$ it holds that
$$
\forall y\in (-\infty,\infty)\ \ \frac{\Psi(y+s)}{\Psi(y)}=\frac{\Psi(a_1+s)}{\Psi(a_1)}=C\in (0,\infty)
$$
Iterating this equality we have 
$\Psi(y+2s)=\frac{\Psi(y+2s)}{\Psi(y+s)}\frac{\Psi(y+s)}{\Psi(y)}\Psi(y)=C^2\Psi(y)$, and then for all $n\in \ZZ$ and all $y$
$$
\Psi(y+ns)=C^n\Psi(y).
$$
Then, if $C\ge 1$, we have
$$
C^n\int_0^{1} \Psi(y)\ dy=\int_0^{1} \Psi(y+ns)\ dy=\int_{ns}^{ns+1} \Psi(y) \ dy \underset{n\to\infty}{\longrightarrow} 0, 
$$
which is a contradiction. A similar contradiction is obtained if $C<1$, because 
$\Psi(y-ns)=C^{-n}\Psi(y)$ and then
$$
C^{-n}\int_{-1}^0\Psi(y)\ dy =\int_{-1}^{0} \Psi(y-ns)\ dy=\int_{-(1+ns)}^{-ns} \Psi(y) \ dy\underset{n\to\infty}{\longrightarrow} 0,
$$
proving that $\widetilde g$ cannot be a minimizer, and the result is shown.

\end{proof}

\begin{remark} Examples of distributions that satisfies the hypothesis of the previous Proposition are the exponential distribution 
$\mu_f(dx)=e^{-x} \ dx$ for $x\ge 0$,
the normal distribution $\mathcal N(0,1)$ and the uniform distribution $\mu_f(dx)=dx$ for $x\in[0,1]$. In the uniform case, we obtain an 
explicit solution for the minimizer
of $f\in L^p(\Omega,\F,\mu)$. For all $k\ge 1$ this unique minimizer is $g=\ell\circ f$, where
$$
\ell=\sum\limits_{i=1}^k \frac{2i-1}{2k}\ \ind_{[\frac{i-1}{k},\frac{i}{k})},
$$
independently of $p\in [1,\infty)$.
\end{remark}

\subsection{The case of an infinite measure, $p\in[1,\infty)$}
\label{sec:infinite measure}

The case of infinite measure needs an extra work and use some ideas already developed in the finite measure case.

\begin{theorem} 
\label{pro:minimum2} 
Let $(\Omega,\F,\mu)$ be an infinite measure space, $p\in [1,\infty)$ and $k\geq 1$. Then $\mathscr{G}_{p,k}$ is proximinal.

Moreover, if $f\in L^p(\Omega,\F,\mu)$ and $g=\sum\limits_{i=1}^q b_i \ind_{A_i}\in P_{\mathscr{G}_{p,k}}(f)$ is a minimizer, with 
$q\le k$, $-\infty<b_1<...<b_q<\infty$, $\{A_i\}_{1\leq i\leq q}$ a partition of $~\Omega$ 
such that $\mu(A_i)>0$ for all $i\in\{1,...,q\}$ and a unique 
$1\le s\le q$ such that $b_s=0$. Then, there exists a minimizer $\widetilde g\in P_{\mathscr{G}_{p,q}}(f)$ in $f$-special form
$$
\widetilde g=\sum\limits_{i=1, i\neq s}^q \M_p(f,f^{-1}(C_i)) \ind_{f^{-1}(C_i)}+0\ind_{f^{-1}(C_s)},
$$
where 
\begin{itemize}
    \item $r_1=-\infty, r_{q+1}=\infty$ and $r_i=\frac{b_{i-1}+b_{i}}{2}$ for all $i\in\{2,...,q\}$ (notice that $r_s<0<r_{s+1}$);
    \item $C_i=f^{-1}([r_i,r_{i+1}))$ for $i\in\{1,...,q-1\}$ and $C_q=f^{-1}([r_q,r_{q+1}])$;
    \item if $\mu(f^{-1}(C_i))>0$ and $i\neq s$, then $b_i$ is a $p$-th mean of $f$ on $f^{-1}(C_i)$;
\end{itemize}
If $q$ is the smallest among all minimizers, then $\mu(C_i)>0$ for all $i$.
\end{theorem}

\begin{proof} For $k=1$ the result is obvious since $\mathscr{G}_{p,1}=\{0\}$. So for the rest of the proof we assume that $k\ge 2$.  

Let $f\in L^p(\Omega,\F,\mu)$ and consider a sequence $(g_n)_n\in \mathscr{G}_{p,k}$ such that 
$g_n=\sum_{i=1}^{q(n)} a_{i,n}\ind_{A_{i,n}}$, where $a_{1,n}<...<a_{q(n),n} \in \RR$, $\{A_{i,n}\}_{1\leq i\leq q(n)}$ is a measurable partition
with sets of positive measure and $q(n)\le k$ for all $n\in\mathbb N$, and such that 
$$
\|f-g_n\|_p \to \mathscr{D}_{p,k}(f).
$$
Since $g_n \in L^p(\Omega,\F,\mu)$ there exists a unique $1\le s(n)\le q(n)$ such that $a_{s(n),n}=0$ and we have that $\mu(A_{i,n})<\infty$ for
all $i\neq s(n)$. Passing to a subsequence, we can assume that $1\le s(n)=s\le q(n)=q\le k$. Define $r_1=r_{1,n}=-\infty$, $r_{q+1,n}=\infty$ and
$r_{i,n}=\frac{a_{i-1,n}+a_{i,n}}{2}$ for $i\in\{2,...,q\}$. 
We point out that if $q=1$, then $g_n=0$, for all $n$ and so
$h=0$ is a minimizer. Then for the rest of the proof, we assume $q\ge 2$.
\smallskip

Now, consider $I_{i,n}=[r_{i,n},r_{i+1,n})$ 
and the corresponding $C_{i,n}=f^{-1}(I_{i,n})$ for all $i\in\{1,...,q\}$. For all $n\in\mathbb N$, define $$\widetilde g_n=\sum_{i=1}^{q} a_{i,n}\ind_{f^{-1}(I_{i,n})}.$$ If $i,j\in\{1,...,q\}$, we have that $|f(x)-a_{i,n}|\geq |f(x)-a_{j,n}|$ for all $x\in C_j$. It follows that for all $n\in\mathbb N$
\begin{align*}
\|f-g_n\|_p^p&=\sum_{i=1}^q\int_{A_i}|f(x)-a_{i,n}|^pd\mu(x)=\sum_{j=1}^q\sum_{i=1}^q\int_{A_i\cap C_j}|f(x)-a_{i,n}|^pd\mu(x) \\
&\geq \sum_{j=1}^q\sum_{i=1}^q\int_{A_i\cap C_j}|f(x)-a_{j,n}|^pd\mu(x)\\
&=\sum_{j=1}^q \int_{C_j} |f(x)-a_{j,n}|^pd\mu(x)=\|f-\widetilde g_n\|_p^p
\end{align*}
proving that $(\widetilde g_n)_n$ is also a minimizing sequence.

For all $i\in\{1,...,q\}$, the sequence $(a_{i,n})_n$ has a convergent subsequence 
in $\overline \RR$.
Then we can also assume that $a_{i,n}\to a_i\in \overline \RR$ for all $i\in \{1,...,q\}$. We denote by $z_1<...<z_\ell$ the different
values in $\{a_1,...,a_q\}$, where $\ell\le q$. We point out that $z_t=0$ for some $1\le t\le \ell$.
For each $1\le m\le \ell$, we denote $L_m=\{i: \ 1\le i\le q \hbox{ and } a_i=z_m\}$. 
Each $L_m$ is an interval in $\NN$, because we have assumed $a_{1,n}<...< a_{i,n} <...< a_{q,n}$, for each $n$. 
We define $i_m^-=\min\{L_m\}$ and $i_m^+=\max\{L_m\}$ for all $m\in \{1,...,\ell\}$ and also $i_{\ell+1}^-=\ell+1$. Note that $L_m=\{i_m^-,...,i_m^+\}$ for all $m\in \{1,...,\ell\}$.

Assume that $z_{\ell}=\infty$ or $z_1=-\infty$. In this situation $\ell\ge 2$, because $z_t=0$.
As in the case of finite measure we can modify $(\widetilde g_n)_n$ to get a uniformly 
bounded minimizing sequence. 
Consider first the case $z_{\ell}=\infty$ and recall that
$i_\ell^-=\min\{L_{\ell}\}\in\{s+1,...,q\}$. Then, we have 
$$
r_{i_\ell^-,n}=\frac{a_{i_\ell^--1,n}+a_{i_\ell^-,n}}{2}\to \infty,
$$
because $a_{i_\ell^- -1,n}\ge a_{s,n}=0$ and then $a_{i_\ell^--1,n}\to a_{i_\ell^--1}=z_{\ell-1}\in [0,\infty)$.
Consider 
$$
\widehat g_n=\sum_{i< i_\ell^-} a_{i,n} \ind_{f^{-1}(I_{i,n})}+ a_{i_\ell^--1,n} \ind_{f^{-1}([r_{i_\ell^-,n},\infty))}=
\sum_{i< i_\ell^--1} a_{i,n} \ind_{f^{-1}(I_{i,n})}+ a_{i_\ell^--1,n} \ind_{f^{-1}([r_{i_\ell^--1,n},\infty))}
$$
An important fact is that $\int_{\{f\ge r_{i_\ell^-,n}\}} |f(x)-a_{i_\ell^--1,n}|^p\ d\mu(x)\to 0$,
because $(a_{i_\ell^--1,n})_n$ is a bounded sequence and $r_{i_\ell^-,n}\to \infty$.
Then
$$
\|f-\widehat g_{n}\|_p^p\le \|f-\widetilde g_{n}\|_p^p+\int_{\{f>r_{i_\ell^-,n}\}} \!\!\!|f(x)-a_{i_\ell^--1,n}|^p\ d\mu(x)
\to (\mathscr{D}_{p,k}(f))^p.
$$
Then, the sequence $(\widehat g_{n})_n$ is a minimizing sequence, which is uniformly upper bounded. Similarly, we can modify this sequence 
to get a minimizing sequence, which is uniformly bounded.
Then, in what follows we assume $(\widetilde g_n)_n$ is uniformly bounded and $-\infty< z_1, z_\ell<\infty$.
\smallskip

Now we consider 2 different cases.

\underline{Case 1:} $\ell=1$. In this situation $a_1=...=a_q=0$. Notice that $r_{q,n}=\frac {a_{q-1,n}+a_{q,n}}{2}\to 0$, so if $0<f(x)<\infty$, then 
$\widetilde g_n(x)=a_{q,n}$ for all large $n$ and  $\widetilde g_n(x)\to 0$. In the same way, if $f(x)<0$, then 
$\widetilde g_n(x)=a_{1,n}$ for all large $n$ and $\widetilde g_n(x)\to 0$. On the other hand, if $f(x)=0$, then
$\widetilde g_n(x)=0$. Then by Fatou's Lemma we conclude
$$
\liminf_n \|f-\widetilde g_n\|_p^p \ge \int \liminf_n |f(x)-\widetilde g_n|^p \ d\mu(x)=\|f\|_p^p,
$$
and we obtain $\mathscr{D}_{p,k}(f)\ge \|f\|_p$, showing that $h=0$ is a minimizer.

\underline{Case 2:} $\ell\ge  2$. For all $m\in\{1,...,\ell\}$ recall that $i_m^-=\min\{L_m\}$ and
$i_{\ell+1}^-=\ell+1$. Then, for all $2\le m\le \ell$
$$
r_{i_m^-,n}\to r_m:=\frac{z_{m-1}+z_m}{2}.
$$
and $r_1=-\infty<r_2<...<r_{\ell}<r_{\ell+1}:=\infty$. Now, we choose a particular subsequence $(n')_{n'}$. 
We start with $(r_{i_2^-,n})_n$. If there exist an
increasing subsequence of $(r_{i_2^-,n})_n$, we fix one of these subsequences as $(n^{(2)})$ and we put $T(2)=\text{in}$, for increasing. 
Otherwise we take $(n^{(2)})$ so that $(r_{i_2^-,n^{(2)}})_{n^{(2)}}$ is strictly decreasing, and we put $T(2)=\text{sd}$, for strictly decreasing.
We repeat this procedure for $(r_{i_3^-,n^{(2)}})_{n^{(2)}}$, to obtain, if possible,  $(n^{(3)})$ a subsequence of $(n^{(2)})$ so 
$(r_{i_3^-,n^{(3)}})_{n^{(3)}}$ is increasing, and put $T(3)=\text{in}$. Otherwise we take $(n^{(3)})$ a subsequence of $(n^{(2)})$ so that
$(r_{i_3^-,n^{(3)}})_{n^{(3)}}$ is strictly decreasing, and we put $T(3)=\text{sd}$. We continue until $m=\ell$. 
We also put $T(1)=\text{in}$ and $T(\ell+1)=\text{in}$. Denote by $(n')=(n^{(\ell)})$.

Now, we define the intervals that give a minimizer. For all $m\in\{1,...,\ell\}$ let
\begin{equation}
\label{eq:intervals}
I_m=\begin{cases} [r_m,r_{m+1})  &\hbox{if } T(m)=\text{in},\ T(m+1)=\text{in}\\
                  [r_m,r_{m+1}]  &\hbox{if } T(m)=\text{in},\ T(m+1)=\text{sd}\\
                  (r_m,r_{m+1})  &\hbox{if } T(m)=\text{sd},\ T(m+1)=\text{in}\\
                  (r_m,r_{m+1}]  &\hbox{if } T(m)=\text{sd},\ T(m+1)=\text{sd}\\
\end{cases}
\end{equation}
We notice that $\cup_{m=1}^{\ell} I_m=[-\infty,\infty)$, and for all $m\in\{1,...,\ell\}$ and all $n'$, we define 
$J_{m,n'}=\cup_{i\in L_m} [r_{i,n'},r_{i+1,n'})=[r_{i_m^-,n'},r_{i_{m+1}^-,n'})$. Then, it holds
$$
\ind_{f^{-1}(J_{m,n'})}\to \ind_{f^{-1}(I_m)}\ a.e.
$$
The last piece of information we need is that the set $\cup_{m\neq t} f^{-1}(J_{m,n'})$
is contained in a fixed set of finite measure $\widetilde A$ for large $n'$. 
If $t=\ell$, then 
$\cup_{m\neq t} f^{-1}(J_{m,n'})\subset f^{-1}((-\infty,r_{t,n'}])\subset 
\widetilde A=f^{-1}((-\infty,\frac{r_t}{2}])$, for large $n'$, because $r_{t,n'}\to r_t=\frac{z_{t-1}}{2}<z_t=0$, 
and then $\widetilde A$ has finite measure.
Similarly, if $t=1$, then $\cup_{m\neq t} f^{-1}(J_{m,n'})\subset f^{-1}([r_{2,n'},\infty])\subset \widetilde A=f^{-1}([\frac{r_2}{2},\infty))$,
for large $n'$. This set has finite measure because $r_2>0$. In the general case, $1<t<q$, we have for large $n'$
$$
\cup_{m\neq t} f^{-1}(J_{m,n'})\subset \widetilde A=f^{-1}((-\infty,r_t/2])\cup f^{-1}([r_{t+1}/2,\infty)),
$$
which has finite measure because $r_t<0<r_{t+1}$.

Now, consider the decomposition
$$
\begin{array}{ll}
\|f-\widetilde g_{n'}\|_p^p &\hspace{-0.2cm}=\int |f(x)-\widetilde g_{n'}(x)|^p\ind_{f^{-1}(J_{t,n'})} \ d\mu(x)
+\sum_{m\neq t}\int_{f^{-1}(J_{m,n'})} |f(x)-\widetilde g_{n'}(x)|^p \ d\mu(x).
\end{array}
$$
We use now Fatou's Lemma for the first term and the Dominated Convergence Theorem for the second term.
In the first term, we have the a.e. convergence
$$
|f-\widetilde g_{n'}|\ind_{f^{-1}(J_{t,n'})} \to |f|\ind_{f^{-1}(I_t)}
$$
With respect to the second term, for large $n'$, we have $\max_{1\le i\le q}|a_{i,n'}|\le\max\{|z_1|,z_{\ell}\}+1:=C$, also
$$
|f(x)-\widetilde g_{n'}(x)|^p\ \ind_{\cup_{m\neq t} f^{-1}(J_{m,n'})} 
\le 2^{p-1}\left(|f(x)|^p+C^p\right) \ind_{\widetilde A}\in L^1(\Omega,\F,\mu),
$$
$\widetilde g_{n'} \ind_{\cup_{m\neq t} f^{-1}(J_{m,n'})} \to \sum_{m\neq t} z_m \ind_{f^{-1}(I_m)}$ a.e.
and $f\ind_{\cup_{m\neq t} f^{-1}(J_{m,n'})} \to f \ind_{f^{-1}(\cup_{m\neq t}I_m)}$ a.e.
So, we get
$$
\liminf_{n'} \|f-\widetilde g_n\|_p^p \ge \int_{f^{-1}(I_t)} |f(x)|^p \ d\mu(x)+\sum_{m\neq t}\int_{f^{-1}(I_m)} |f(x)-z_m|^p \ d\mu(x)
$$
and then $h=\sum_{m=1}^{\ell} z_m\ind_{f^{-1}(I_m)}$ is a minimizer,
where the intervals $\{I_m\}_{1\leq m\leq t}$ are either open, closed or semi-closed, they are 
disjoint and $\cup_{m=1}^\ell I_m=\RR$ (see \eqref{eq:intervals}).

From here it is clear that a minimizer exists in $f$-special form as we have done in the finite measure case. Also notice that if $m\neq t$ and 
$0<\mu({f^{-1}(I_m)})$, we must have $z_m$ is a $p$-th mean for $f$ in $f^{-1}(I_m)$, since $h$ is a minimizer.
\end{proof}

\subsection{The case $p=\infty$}
\label{sec:case_p=infinity}
In this section we shall prove that $\mathscr{G}_{\infty,k}$ is proximinal. We start with a lemma.

\begin{lemma}
Let $(\Omega,\F,\mu)$ be a measure space and $f\in L^\infty(\Omega,\F,\mu)$. Then, for all $k\geq 1$, we have that $\mathscr{D}_{\infty,k}(f)=\eta_k(f)$ 
where 
$$
\eta_k(f)=\inf_{h=f\ a.e}\inf\{\alpha>0\ |\ h(\Omega)\ \text{can be 
covered by at most}\ k\ \text{closed balls of radius}\ \alpha\}.
$$
\end{lemma}

\begin{proof}
Let $\varepsilon>0$ and let $g\in\mathscr{G}_{\infty,k}$ such that $\|f-g\|_\infty\leq \mathscr{D}_{\infty,k}(f)+\varepsilon$. Write $g=\sum_{i=1}^ka_i\ind_{A_i}$ 
where $\{A_i\}_{1\leq i\leq k}$ is a partition of $\Omega$. For every $i$, the set $C_i=\{x\in A_i: \, |f(x)-a_i|>\|(f-g)\ind_{A_i}\|_\infty\}$ 
has measure $0$ and therefore $h=f\ind_{\Omega\setminus \cup_j C_j}+\sum\limits_{j=1}^k a_j\ind_{C_j}$ satisfies $h=f$ a.e. and 
$$
h(\Omega)\subset\bigcup_{i=1}^k[a_i-\mathscr{D}_{\infty,k}(f)-\varepsilon,a_i+\mathscr{D}_{\infty,k}(f)+\varepsilon].
$$ 
It follows that $\eta_k(f)\leq\mathscr{D}_{\infty,k}(f)+\varepsilon$ and since $\varepsilon$ is arbitrary, we obtain that $\eta_k(f)\leq\mathscr{D}_{\infty,k}(f)$.
To prove the other inequality, let again $\varepsilon>0$ and pick $l\leq k$, $a_1,...,a_l\in\mathbb R$ and $h=f$ a.e. such that 
$$
h(\Omega)\subset\bigcup_{i=1}^l[a_i-\eta_k(f)-\varepsilon,a_i+\eta_k(f)+\varepsilon].
$$ 
For $1\leq i\leq l$, define $A_i=h^{-1}([a_i-\eta_k(f)-\varepsilon,a_i+\eta_k(f)+\varepsilon])\in\F$. Now define 
$B_1=A_1$ and $B_i=A_i\setminus\bigcup_{j=1}^{i-1}A_i$ for $i\in\{2,...,l\}$. Then $\{B_i\}_{1\leq i\leq l}$ is a partition of $\Omega$. 
Defining $g=\sum_{i=1}^la_i\ind_{B_i}\in\mathscr{G}_{\infty,k}$, it is clear that $\|f-g\|_\infty\leq\eta_k(f)+\varepsilon$. 
It follows that $\mathscr{D}_{\infty,k}(f)\leq\eta_k(f)+\varepsilon$ and then $\mathscr{D}_{\infty,k}(f)\leq\eta_k(f)$.
\end{proof}

\begin{proposition}
\label{pro:minimum3} 
Let $(\Omega,\F,\mu)$ be a measure space. Then $\mathscr{G}_{\infty,k}$ is proximinal for all $k\ge 1$.
\end{proposition}

\begin{proof}
For all $n\in\mathbb N$, let $\alpha_n=\eta_k(f)+\frac{1}{n}$. So, for all $n\in\mathbb N$, there exist 
$a_1^n,...,a_{l_n}^n\in\mathbb R$ with $1\leq l_n\leq k$ 
and $h_n=f$ a.e. such that $h_n(\Omega)\subset\bigcup_{i=1}^{l_n}[a_i^n-\alpha_n,a_i^n+\alpha_n]$. Of course there exists $i_0$ such that
$\mu(h_n^{-1}([a_{i_0}^n-\alpha_n,a_{i_0}^n+\alpha_n]))>0$. If for some $i$ it holds $\mu(h^{-1}([a_i^n-\alpha_n,a_i^n+\alpha_n]))=0$, 
we can redefine $h_n$ on a set 
of measure $0$, to have $h_n(w)=a_{i_0}$ for all $w\in h_n^{-1}([a_i^n-\alpha_n,a_i^n+\alpha_n])$. So, we can assume for all $i$ it holds 
$\mu(h_n^{-1}([a_i^n-\alpha_n,a_i^n+\alpha_n]))>0$. Consider 
$$
t_{i,n}=\frac{1}{\mu(h_n^{-1}([a_i^n-\alpha_n,a_i^n+\alpha_n]))}\int_{h_n^{-1}([a_i^n-\alpha_n,a_i^n+
\alpha_n])} h_n(x)\ d\mu(x)\in [a_i^n-\alpha_n,a_i^n+\alpha_n],
$$
which obviously satisfies $|t_{i,n}|\le \|h_n\|_\infty=\|f\|_\infty$. Then, for all $i,n$, it holds
$$
|a_i^n|\le |t_{i,n}|+|a_i^n-t_{i,n}|\le \|f\|_\infty+\alpha_n\le \|f\|_\infty+\eta_k(f)+1,
$$
which implies that the set $\{a_i^n\ :\ 1\leq i\leq l_n \ n\in\mathbb N \}$ is bounded.

Considering a subsequence if necessary, we can suppose that 
$l_n=l\in\{1,...,k\}$ for all $n\in\mathbb N$. By compactness and taking a further subsequence, we can also assume 
that $a^n_i\to a_i$ for all $i\in\{1,...,l\}$. Define 
$$
C=\{\omega\in\Omega\ |\ \forall n\in\mathbb N\ f(\omega)=h_n(\omega)\}\in\F
$$ 
and note that $\mu(C^c)=0$. Let us show that $f(C)\subset\bigcup_{i=1}^{l}[a_i-\eta_k(f),a_i+\eta_k(f)]$. In fact, if 
$\omega\in C$ then for all $n\in\mathbb N$ there exists $i(\omega,n)\in\{1,...,l\}$ such that 
$f(\omega)=h_n(\omega)\in [a_{i(\omega,n)}^n-\alpha_n,a_{i(\omega,n)}^n+\alpha_n]$. There exists a subsequence $\phi(n)=\phi(n)(\omega)$ 
such that the sequence $(i(\omega,\phi(n)))_n$ is constant and equal to some $i_0(\omega)\in\{1,...,l\}$. 
It follows that $f(\omega)\in [a_{i_0(\omega)}-\eta_k(f),a_{i_0(\omega)}+\eta_k(f)]\subset\bigcup_{i=1}^{l}[a_i-\eta_k(f),a_i+\eta_k(f)]$. 
Define $h=f\ind_C+t\ind_{C^c}$ where $t$ is any real belonging to $\bigcup_{i=1}^{l}[a_i-\eta_k(f),a_i+\eta_k(f)]$. 
We have that $f=h$ a.e. and $h(\Omega)\subset\bigcup_{i=1}^{l}[a_i-\eta_k(f),a_i+\eta_k(f)]$. For $1\leq i\leq l$, 
define $A_i=h^{-1}([a_i-\eta_k(f),a_i+\eta_k(f)])\in\F$. Now define $B_1=A_1$ and $B_i=A_i\setminus\bigcup_{j=1}^{i-1}A_i$ for $i\in\{2,...,l\}$. 
Then $\{B_i\}_{1\leq i\leq l}$ is a partition of $\Omega$. Defining $g=\sum_{i=1}^l a_i\ind_{B_i}\in\mathscr{G}_{\infty,k}$, 
it is clear that $\|f-g\|_\infty=\|h-g\|_\infty\leq\eta_k(f)$. Moreover, we have that $\mathscr{D}_{\infty,k}(f)=\eta_k(f)$ 
by the previous Lemma and so we conclude that $\|f-g\|_\infty=\mathscr{D}_{\infty,k}(f)$.

\end{proof}

\subsection{Extra properties of minimizers and the sets $(\mathscr{G}_{p,k})_{p,k}$}
\label{sec:extra_properties}
In this section we include some extra properties of the sets $(\mathscr{G}_{p,k})_{p,k}$ as well as some natural questions 
like uniqueness of minimizers and the existence of a continuous selection for $P_{\mathscr{G}_{p,k}}$.

Let us start by proving that $\mathscr{G}_{p,k}$ is a closed set, for all $p\ge 1, k\ge 1$, something that it is not 
straightforward to do. Nevertheless, this is 
a direct consequence of the previous results.

\begin{corollary} Let $(\Omega,\F,\mu)$ be a measure space, $p\in [1,\infty]$ and $k\ge 1$. Then
$\mathscr{G}_{p,k}$ is closed.
\end{corollary}

\begin{proof} Assume $(g_n)_n\subset \mathscr{G}_{p,k}$ converges in $L^p(\Omega,\F,\mu)$ to
$g$. Then 
$$
\inf\{\|g-h\|_p: \ h\in \mathscr{G}_{p,k}\}=0.
$$
From the previous results, there exists a minimizer $\bar h\in \mathscr{G}_{p,k}$, 
that is $g=\bar h$ a.e. and the result is shown.
\end{proof}

A question that appears when proving the existence of minimizers is the following. Assume there exists a best approximation of $f$ 
by an element of $\mathscr{G}_{p,k}$ which is in fact an element of $\mathscr{G}_{p,m}$ for some $m<k$, then it is natural to 
think that $f$ should belong to $\mathscr{G}_{p,m}$.
This is true when $p\in[1,\infty)$ and it is not true for $p=\infty$. Before doing that we require the following lemma.

\begin{lemma}
Assume that $f\in L^p(\Omega,\F,\mu)$, for $1\le p<\infty$, and $A=f^{-1}(I)$ is a set of positive and finite measure, where
$I$ is an interval. Assume $b$ is a $p$-th mean of $f$ on $A$, then $b\in \bar I$.
\end{lemma}

\begin{proof} 
Assume the interval $\bar I=[c,d]$, where $c,d\in \overline\RR$ and let us prove that $b\geq c$. If $c=-\infty$ it is clear that $c<b$. 
So assume $c$ is finite. By contradiction, if $b<c$ we have $|f(x)-b|=f(x)-c+(c-b)>f(x)-c=|f(x)-c|$, for all $x\in f^{-1}(I)$
and then, since $\mu(f^{-1}(I))>0$, we get
$$
\int_{f^{-1}(I)} |f(x)-b|^p \ d\mu(x) > \int_{f^{-1}(I)} |f(x)-c|^p \ d\mu(x),
$$
which is contradiction. Similarly, it is shown that $b\le d$.
\end{proof}

\begin{proposition} 
\label{pro:f_inG_k}
Assume $f\in L^p(\Omega,\F,\mu)$ with $p\in [1,\infty)$. Let $m,k\in\mathbb N$ such that $1\leq m<k$. 
Suppose that there exists $g\in\mathscr{G}_{p,m}\cap P_{\mathscr{G}_{p,k}}(f)$. Then $f\in\mathscr{G}_{p,m}$.
\end{proposition}
\begin{proof} 
Suppose that the measure is finite. We can assume that 
$$
g=\sum\limits_{i=1}^r b_i \ind_{f^{-1}(I_i)},
$$ 
where $r\leq m$, $\{I_i\}_{1\leq i\leq r}$ is a 
family of disjoint intervals such that $\{f^{-1}(I_i)\}_{1\leq i\leq r}$ is a partition of $\Omega$ and $b_1<...<b_r$. Suppose by contradiction that $f\notin\mathscr{G}_{p,m}$. Then in particular it holds that $\mu(f^{-1}(\{b_1,...,b_r\}^c)>0$. Since $\Theta=\{b_1,...,b_r\}^c$ is open, it is a 
countable union of open intervals $(J_n)_n$ and therefore
for some $n_0$ we should have $\mu(f^{-1}(J_{n_0}))>0$. By the continuity of the measure, there exists a closed bounded interval 
$J\subset J_{n_0}$ such that $\mu(f^{-1}(J))>0$, and 
therefore $\mu(f^{-1}(J\cap I_{i_0}))>0$, for some $i_0$. Hence, we obtain 
$$
\int_{f^{-1}(I_{i_0}\cap J)} |f-\M_p(f,f^{-1}(I_{i_0}\cap J))|^p \ d\mu(x) <
\int_{f^{-1}(I_{i_0}\cap J)} |f-b_{i_0}|^p \ d\mu(x),
$$ 
since an equality in the previous formula would imply that $b_{i_0}\in\overline{I_{i_0}\cap J}\subset J\subset \{b_1,...,b_r\}^c$, 
by the previous Lemma. If we define 
$$
h=\sum\limits_{i=1,i\neq i_0}^r b_i \ind_{f^{-1}(I_i)}+b_{i_0}\ind_{f^{-1}(I_{i_0}\cap J^c)}+\M_p(f,f^{-1}(I_{i_0}\cap J))
\ind_{f^{-1}(I_{i_0}\cap J)}\in\mathscr{G}_{p,r+1}\subset\mathscr{G}_{p,k},
$$ 
we have that $\|f-h\|_p<\|f-g\|_p$ which contradicts the minimality of $g$. We conclude that $f\in\mathscr{G}_{p,m}$.

In case the measure is infinite, with the same notation as above, we know that $b_{i_1}=0$ for some $i_1$. As above there 
exists a closed and bounded interval $J\subset \{b_1,...,b_r\}^c \subset \{0\}^c$, such that $\mu(f^{-1}(J))>0$. 
Without loss of generality we can assume that $J\subset [a,\infty)$, for some $a>0$. Then
$$
\mu(f^{-1}(J)) a^p\le \|f\|^p,
$$
proving that $f^{-1}(J)$ has finite and positive measure. The argument now goes as in the case of finite measure.
\end{proof}

The following result shows that, for $p\in [1,\infty)$, the error in the approximation by functions in $\mathscr{G}_{p,k}$  
decreases strictly with $k$ until eventually reaching zero.

\begin{corollary} Assume that $f\in L^p(\Omega,\F,\mu)$ with $p\in [1,\infty)$ and consider $\mathscr{D}_{p,\infty}(f)=0$.
Define $k^*=\min\{k:\ \mathscr{D}_{p,k}(f)=0\}\in [1,\infty]$. Then,
$(\mathscr{D}_{p,k}(f))_{k\le k^*}$ is strictly decreasing and $\mathscr{D}_{p,k}(f)=0$ for all $k\ge k^*$, that is 
$$
k^*=\min\{k:\ \mathscr{D}_{p,k+1}(f)=\mathscr{D}_{p,k}(f)\}=\min\{k:\ \mathscr{D}_{p,k}(f)=0\}.
$$
\end{corollary}

The previous results are not true for $p=\infty$. In fact, we have the following example:

\begin{example}
Consider the Lebesgue measure in $[0,1]$, 
the function
$$
f(x)=\begin{cases} x   &\hbox{for } x\notin (\frac13,\frac23)\\
                \frac13 &\hbox{for } x\in (\frac13,\frac23)
\end{cases}
$$
and $k=3$. It is not difficult to show that 
$
\mathscr{D}_{\infty,3}(f)=\frac16,
$
where there are multiple minimizers, for example 
$$
h=\frac16\, \ind_{[0,\frac13]}+\frac12\, \ind_{(\frac13,\frac23)}+\frac56\, \ind_{[\frac23,1]}
$$
is a minimizer, but also 
$$
g=\frac16\, \ind_{[0,\frac23)}+\frac56\, \ind_{[\frac23,1]}\in \mathscr{G}_{\infty,2}
$$
is a minimizer, in particular $\mathscr{D}_{\infty,3}(f)=\mathscr{D}_{\infty,2}(f)>0$. Nevertheless, 
$f\notin \mathscr{G}_{\infty,k}$ for all $k$. This also shows that $\mathscr{G}_{\infty,3}$ is not Chebyschev.
\end{example}

In Proposition \ref{pro:uniqueness}, we have shown that under certain conditions on $f$, there exists a unique minimizer. 
An important question then is if $\mathscr{G}_{p,k}$ is Chebyschev, that is, if there is a unique minimizer for all $f$. As we have seen in 
the previous example this is not true for $p=\infty$, and we complement this for all $p$.

\begin{example}
Consider again the Lebesgue measure in $[0,1]$. Then $\mathscr{G}_{p,2}$ is not Chebyshev for any $p\in[1,\infty]$. To see that, let $f=-\ind_{[0,\frac13)}+0\ind_{[\frac13,\frac23)}+\ind_{[\frac23,1]}$. A possible
minimizer in $\mathscr{G}_{2,2}$ has the form $g=a\ind_{f^{-1}((-\infty,r_2))}+b\ind_{f^{-1}([r_2,\infty])}$, 
for suitable $a,b,r_2$ (see Theorem \ref{pro:minimum}). 
If $r_2\le -1$ or $r_2>1$, a candidate to be a minimizer is $g_1=0$. 
For $-1< r_2\le 0$ the candidate is $g_2=-\ind_{[0,\frac13)}+\frac12 \ind_{[\frac13,1]}$. 
Finally, for $0< r_2<1$ the candidate is
$g_3=-\frac12 \ind_{[0,\frac23)}+ \ind_{[\frac23,1]}$. The corresponding errors are 
$$
\|f-g_1\|_2^2=\frac23,\ \|f-g_2\|_2^2=\|f-g_3\|_2^2=\frac16,
$$
showing that $g_2$ and $g_3$ are two minimizers and then $\mathscr{G}_{2,2}$ is not Chebyschev. Finally, for every $p\in [1,\infty]$ 
both $g_2$ and $g_3$ are minimizers in $\mathscr{G}_{p,2}$, showing that this set is not Chebyschev for any $p$. Moreover,
for $1<p<\infty$, it can be proved that $g_2, g_3$ are the only minimizers. For $p=1$, there is a continuum of minimizers since 
$$
g_a=-\ind_{[0,\frac13)}+a\ind_{[\frac13,1]}
$$ 
is a minimizer for all $a\in[0,1]$. For $p=\infty$, there is also a continuum of 
minimizers since
$$
h_b=b\ind_{[0,\frac13)}+\frac12\ind_{[\frac13,1]}
$$
is a minimizer for all $b\in[\frac{-3}{2},\frac{-1}{2}]$.
\end{example}

\begin{remark}
We have proved that $\mathscr{G}_{p,k}$ is proximinal and closed for all 
$k\geq 1$ and $p\in [1,\infty]$. However, $\mathscr{G}_{p,k}$ is not Chebyshev in general as we have shown in the previous examples. 
Then, it is natural to ask if $P_{\mathscr{G}_{p,k}}$ admits a 
continuous selection. If such continuous selection exists, then $\mathscr{G}_{p,k}$ has to be almost-convex (see Lemma 5 in  \cite{Vlasov}).
Remember that a subset $K$ of a Banach space is said to be \textit{almost-convex} (see \cite{Vlasov}) if for every closed ball $B$ such that $K\cap B=\emptyset$, 
there exists a closed ball $B'$ of arbitrary large radius such that $K\cap B'=\emptyset$ and $B\subset B'$. If $p\in(1,\infty)$, 
a subset $K$ is almost-convex if and only if $K$ is convex (see Lemma 2 in  \cite{Vlasov}). So, the question is if $\mathscr{G}_{p,k}$
can be convex. For $k\ge 2$ and $p<\infty$, $\mathscr{G}_{p,k}$ is convex if and only if $L^p(\Omega,\F,\mu)$ is finite dimensional and
$L^p(\Omega,\F,\mu))=\mathscr{G}_{p,k}$. Indeed, 
assume $k\ge 2$ and that $\mathscr{G}_{p,k}$ is convex. Then it is direct to show that $\mathscr{G}_{p,k}$ is a vector space, because it is homogeneous. Then
$\mathscr{G}_{p,\ell}=\mathscr{G}_{p,k}$, for all $\ell\ge k$. This is done by induction, so the only interesting case 
is $\ell=k+1$. Take $g=\sum_{i=1}^{k+1} a_i\ind_{A_i}$, which
can be seen as the sum of three elements $g_1,g_2,g_3 \in \mathscr{G}_{p,k}$
$$
g_1=\sum\limits_{i=1}^{k-1} a_i \ind_{A_i}+0\ind_{A_{k}\cup A_{k+1}},\ g_2=a_k\ind_{A_k}+0\ind_{\cup_{j\neq k} A_j},
\ g_3=a_{k+1}\ind_{A_{k+1}}+0\ind_{\cup_{j\neq k+1} A_j}.
$$
Therefore, $\mathscr{G}_{p,k}=\cup_\ell \mathscr{G}_{p,l}$ is dense and closed in $L^p(\Omega,\F,\mu)$, 
which implies $\mathscr{G}_{p,k}=L^p(\Omega,\F,\mu)$. The conclusion is that
the unit ball of $L^p(\Omega,\F,\mu)$ is UA and then $L^p(\Omega,\F,\mu)$ is finite dimensional (see Theorem \ref{ballUA} in Section \ref{UballisUA}).

\end{remark}

\section{The $p$-variation}
\label{sec:pvariation}

In this part we introduce a new notion of variation for functions in $L^p(\Omega,\F,\mu)$. There are several notions of variation or oscillation for functions.
Our notion notion could be contrasted  with the definition of oscillation given in \cite{Bogachev} (p.296), which helps 
to characterize compact sets in $L^1$. However, both concepts are not comparable, in general. 

\begin{definition}
Let $p\in[1,\infty)$. For $f\in L^p(\Omega,\F,\mu)$ and $A$ a measurable set of finite measure, we define $var_p(f,A)$, 
the \textit{$p$-variation} of $f$ in $A$, as
$$
var_p(f,A)^p=\begin{cases} \frac{1}{\mu(A)} \int_{A\times A} |f(x)-f(y)|^p \, d\mu(x)d\mu(y) &\hbox{if } \mu(A)>0\\
                           0  &\hbox{otherwise}.
              \end{cases}               
$$
Given $\mathcal{P}=(A_i)_i$, a finite collection of disjoint measurable sets each one
of finite measure, which we also assume
it contains at least one set of positive measure,
we define the total $p$-variation of $f$ in $\mathcal{P}$ as 
$$
var_p(f,\mathcal{P})=\left(\sum_i var_p(f,A_i)^p\right)^{1/p}=
\left(\sum\limits_{i:\, \mu(A_i)>0} \frac{1}{\mu(A_i)} \int_{A_i\times A_i} |f(x)-f(y)|^p \, d\mu(x)d\mu(y)\right)^{1/p}.
$$
For a measurable set $A$ of finite measure, we define the \textit{$k$-th total $p$-variation} of $f$ as 
$$
\V_{p,k}(f,A)=\inf\Big\{var_p(f,\mathcal{P}): \, \mathcal{P} \hbox{ is a partition of } A, \, 
|\mathcal{P}|\le k\Big\}
$$
where the infimum is taken over the set of finite measurable partitions of $A$ consisting of 
at most $k$ measurable sets. Finally, we define the \textit{total $p$-th variation} of $f$ as
$$
\V_{p,k}(f)=\sup_{{A\in\F\atop\mu(A)<\infty}} \V_{p,k}(f,A)
$$
Note that if $\mu$ is finite then $\V_{p,k}(f,\Omega)\leq \V_{p,k}(f)$, and it is not clear if both 
measures of total variation are equivalent, something that we study below (see Proposition \ref{variation}).
\end{definition}

\begin{remark} Notice that the sets in $\mathcal{P}$ that have measure $0$ can be removed by 
gluing them to an element of
$\mathcal{P}$ with positive measure. We redefine a new 
collection $\widetilde{\mathcal{P}}$, which has fewer elements
and $var_p(f,\mathcal{P})=var_p(f,\widetilde{\mathcal{P}})$. So, in what follows,
we can always assume that 
$\mathcal{P}$ is a collection with sets of positive and finite measure.
\end{remark}

\medskip

We compile some basic properties of $\V_{p,k}(\bullet,\Omega)$ and $\V_{p,k}(\bullet)$ in the next result:

\begin{proposition}
Let $(\Omega,\F,\mu)$ be a measure space and $p\in[1,\infty)$. Then $(\V_{p,k})_{k\geq 1}$ 
is a decreasing family of continuous semi-norms on $L^p(\Omega,\F,\mu)$ such that 
$\V_{p,k}(\bullet)\leq 2\|\bullet\|_p$ for all $k\geq 1$. The same properties hold for 
$(\V_{p,k}(\bullet,\Omega))_{k\geq 1}$, in the case $\mu$ is a finite measure.
\end{proposition}

\begin{proof}
The fact that $\V_{p,k}$ is a semi-norm is easy and is left to the reader. 
The monotony of $(\V_{p,k})_{k\geq 1}$ follows directly from the definition. 
Let $f\in L^p(\Omega,\F,\mu)$ and $A$ be a measurable set of finite and positive measure. 
First note that $var_p(f,A)\le 2\|f\ind_A\|_p$. In fact, using that $(a+b)^p\le 2^{p-1}(a^p+b^p)$
holds for all nonnegative numbers $a,b$ and Fubini's theorem, we have that
\begin{align}
\label{eq:desi_var}
var_p(f,A)^p &= \frac{1}{\mu(A)} \int_{A\times A} |f(x)-f(y)|^p \, d\mu(x)d\mu(y) \nonumber\\
&\leq \frac{2^{p-1}}{\mu(A)} \int_{A\times A} |f(x)|^p+|f(y)|^pd\mu(x)d\mu(y)\\
&=2^p\|f\ind_A\|_p^p.\nonumber
\end{align}
It follows that if $\mathcal{P}$ is a finite measurable partition of $A$ then 
$\V_p(f,\mathcal{P})\le 2\|f\ind_A\|_p$ and therefore, we deduce that $\V_{p,k}(f)\leq 2\|f\|_p$. 
In particular, $\V_{p,k}$ is continuous. In case the measure is finite we have
$$
\V_{p,k}(f,\Omega)\le \V_{p,k}(f)\le 2\|f\|_p.
$$
\end{proof}

\begin{remark}
Assume that $\mu$ is a finite measure. We notice that for a fixed function $f\in L^p(\Omega,\F,\mu)$, 
we have $\lim_{k\to \infty} \V_{p,k}(f,\Omega)=0$. Indeed, let $k\in \NN$ and define the sets
$$
A_{i}=\left\{x: \, \frac{i}{k} \le f(x)<\frac{i+1}{k}\right\}, \hbox{ for } i\in\{-k^2,...,k^2-2\},
$$
$A_{k^2-1}=\left\{x: \, k-\frac{1}{k} \le f(x)\le k\right\}$ and $A_{k^2}=\{x: |f(x)|>k\}$.
Then, we have, for all $i\in\{-k^2,...,k^2-1\}$
$$
var_p(f,A_i)\le \frac{\mu(A_i)^{\frac{1}{p}}}{k}
$$
and for $i=k^2$
$$
var_p(f,A_{k^2})\le 2\|f\ind_{A_{k^2}}\|_p.
$$
Thus,
$$
\V_{p,k}(f,\Omega)^p\le \frac{1}{k^p}\mu(\Omega)+2^p \int_{|f|>k} |f(x)|^p \,d\mu(x),
$$
and then $\lim_k \V_{p,k}(f,\Omega)=0$.

We also notice that the same property holds for $\left(\V_{p,k}(f)\right)_k$, in general measure spaces, but its proof is more involved and
we postponed to Corollary \ref{cor_limitvar}.
\end{remark}

The following lemma proves that the variation of a function can always be computed on a $\sigma$-finite set if the measure has no atoms of infinite mass.

\begin{lemma}\label{sigmafinite_variation}
Assume $(\Omega,\F,\mu)$ is a measurable space such that $\mu$ has no atoms of infinite mass and $p\in[1,\infty)$. Let $f\in L^p(\Omega,\F,\mu)$ and fix $k\geq 1$. Then there exists an increasing sequence of finite measure sets $(\Omega^*_n)_n\subset\F$ such that $$\V_{p,k}(f)=\lim_n\V_{p,k,n}(f)=\V_{p,k,*}(f),$$  where 
\begin{itemize}
    \item $\V_{p,k,*}(f)$ is the total variation of $f|_{\Omega^*}$ computed in $(\Omega^*,\F|_{\Omega^*}, \mu|_{\Omega^*})$ with $\Omega^*=\bigcup_n\Omega^*_n;$
    \item $\V_{p,k,n}(f)$ is the total variation of $f|_{\Omega_n^*}$ computed in $(\Omega_n^*,\F|_{\Omega_n^*}, \mu|_{\Omega_n^*})$.
\end{itemize}
\end{lemma}

\begin{proof}
We can obviously suppose that $\mu$ is infinite. Define $F=\{x: \ f(x)=0\}$. In case $F$ has finite measure, we define $\widetilde{D}=F$. If $F$ has infinite measure, we consider a subset $\widetilde{D}\subset F$ which is $\sigma$-finite and of infinite measure. Note that such a set exists. Indeed, take
$$
a=\sup_{D\in \F, D\subset F, \mu(D)<\infty} \mu(D).
$$
Let us prove that $a=\infty$. Consider a sequence $(D_l)_l$ of subsets of $F$, each one of finite measure
such that $\lim_l \mu(D_l)=a$. It is clear
that $\widetilde{D}_l=\cup_{i\le l} D_i$ is an increasing sequence of sets of finite measure, included in $F$ which satisfies
$
\mu(D_l)\le \mu(\widetilde{D}_l),
$
proving that $\mu(\widetilde{D}_l)\uparrow a$ and $\widetilde{D}=\cup_l \widetilde{D}_l$ satisfies $\mu(\widetilde{D})=a$. If $a$ is finite
then, $F\setminus\widetilde{D}$ has infinite measure. By hypothesis this set contains a set $H$ of finite and positive measure. 
Then $\mu(\widetilde{D}\cup H)=\mu(\widetilde{D})+\mu(H)>a$, which is a contradiction.

Now consider a sequence of sets of finite measure $(A_n)_n$ such that
$$
\V_{p,k}(f,A_n)\ge \V_{p,k}(f)-\frac{1}{n}.
$$ For every $m\ge 1$ the set  $C_m=\left\{x: |f(x)|>\frac{1}{m}\right\}$ has finite measure. The set $\Omega^*=\cup_n A_n \cup \cup_m C_m \cup\widetilde{D}$ is $\sigma$-finite
and it has infinite measure, because $\mu(\bigcup_m C_m \cup F)=\mu(\Omega)=\infty$. We consider 
$$
\Omega_n^*=\begin{cases} \cup_{i\le n} A_i \cup C_i \cup \widetilde{D}     &\hbox{if } \mu(\widetilde{D})<\infty\\
                         \cup_{i\le n} A_i \cup C_i \cup \widetilde{D}_i   &\hbox{if } \mu(\widetilde{D})=\infty
\end{cases},
$$
which is an increasing sequence of sets of finite and positive measure, such that $\Omega_n^*\uparrow \Omega^*$. Define $\V_{p,k,n}(f)$ the total variation of $f|_{\Omega_n^*}$ computed in $(\Omega_n^*,\F|_{\Omega_n^*}, \mu|_{\Omega_n^*})$, that is
$$
\V_{p,k,n}(f)=\sup\limits_{A\in \F, A\subset \Omega_n^*} \V_{p,k}(f|_{\Omega_n^*},A)=\sup\limits_{A\in \F, A\subset \Omega_n^*} \V_{p,k}(f,A)\le \V_{p,k}(f).
$$
Similarly, we define $\V_{p,k,*}(f)$, which is the total variation of $f|_{\Omega^*}$ computed in $(\Omega^*,\F|_{\Omega^*}, \mu|_{\Omega^*})$. It is clear that for every $n$, by construction,
$$
\V_{p,k}(f)-\frac{1}{n}\le \V_{p,k}(f,A_n)\le \V_{p,k,n}(f)\le \V_{p,k,*}(f)\le \V_{p,k}(f),
$$
and also that $(\V_{p,k,n}(f))_n$ is increasing, showing that 
$$
\V_{p,k,n}(f)\uparrow \V_{p,k}(f)
$$
and $\V_{p,k,*}(f)=\V_{p,k}(f)$.  
\end{proof}

The next proposition shows that the variation and $\mathscr{D}_{p,k}$ have the same behaviour. This will be a fundamental tool to caracterize the uniform approximability of sets.

\begin{proposition}
\label{variation} 
Assume $(\Omega,\F,\mu)$ is a measurable space and $p\in[1,\infty)$. For any $k\ge 1$ and any
$f\in L^p(\Omega,\F,\mu)$, we have
\begin{enumerate}
\item[(i)] 
$$
\mathscr{D}_{p,k+1}(f)\le \V_{p,k}(f)\le 2\mathscr{D}_{p,k}(f).
$$
\item[(ii)] If the measure $\mu$ is finite, it holds 
$$
\mathscr{D}_{p,k}(f)\le \V_{p,k}(f,\Omega)
\le \V_{p,k}(f) \le 2\mathscr{D}_{p,k}(f)\le 2\V_{p,k}(f,\Omega).
$$
\item[(iii)] If $\mu$ has no atoms of infinite mass, we have that
$$
\mathscr{D}_{p,k}(f)\le \V_{p,k}(f)\le 2\mathscr{D}_{p,k}(f).
$$
\end{enumerate}
\end{proposition}
\begin{proof} $(i)$ For the upper bound, consider $g\in \mathscr{G}_k\cap L^p(\Omega,\F,\mu)$ a function such that
$$
\mathscr{D}_{p,k}(f)^p=\|f-g\|_p^p.
$$
Assume that
$
g=\sum_{i=1}^k c_i \ind_{A_i},
$
where $\{A_i\}_{1\leq i\leq k}$ is a finite partition of $\Omega$. Clearly if $\mu(A_i)=\infty$, then $c_i=0$. For $A$ a set of finite measure define a partition of $A$ by 
$\mathcal P=\{A\cap A_i\}_{1\leq i\leq k}$. Using that $(a+b)^p\le 2^{p-1}(a^p+b^p)$ for all positive numbers $a$ and $b$, 
we get for all $i$ such that $\mu(A\cap A_i)>0$
$$
var_p(f,A\cap A_i)^p=\frac{1}{\mu(A\cap A_i)}\int_{(A\cap A_i)\times (A\cap A_i)} 
|f(x)-f(y)|^p d\mu(x)d\mu(y)\le 2^p \int_{A\cap A_i} |f(z)-c_i|^p d\mu(z).
$$
Then,
$$
var_p(f,\mathcal{P})^p\le 2^p \sum_i \int_{A\cap A_i} |f(z)-c_i|^p d\mu(z)\le 2^p \|f-g\|_p^p.
$$ 
Therefore, we get
$$
\V_{p,k}(f,A)\le 2 \|f-g\|_p=2\mathscr{D}_{p,k}(f).
$$

For the lower bound let $\varepsilon>0$ and take a set $A$ of finite measure such that
$\|f\ind_{A^c}\|_p^p<\e$. By definition of $\V_{p,k}(f,A)$, there exists a finite 
partition $\mathcal{P}=\{A_i\}_{1\leq i\leq n}$  of $A$, with $n\le k$, such
that (we assume all the sets in $\mathcal{P}$ has positive measure)
\begin{align*}
var_p(f,\mathcal{P})^p&=\sum_i \frac{1}{\mu(A_i)}\int_{A_i\times A_i} 
|f(x)-f(y)|^p d\mu(x)d\mu(y)\\
&\le \left(\V_{p,k}(f,A)\right)^p +\e
\le\left(\V_{p,k}(f)\right)^p +\e.  
\end{align*}
For every $i\le n$ by the definition of $\M_p(f,A_i)$, we have
$$
\ind_{A_i}(y)\int_{A_i} |f(x)-\M_p(f,A_i)|^p \, d\mu(x) 
\le \ind_{A_i}(y) \int_{A_i} |f(x)-f(y)|^p \, d\mu(x),
$$
and therefore, integrating over $y$ we get
$$
\int_{A_i} |f(x)-\M_p(f,A_i)|^p \, d\mu(x) \le \frac{1}{\mu(A_i)} \int_{A_i\times A_i} |f(x)-f(y)|^p \, d\mu(x),
$$
and then
$$
\sum_i \int_{A_i} |f(x)-\M_p(f,A_i)|^p \, d\mu(x)  \le var_p(f,\mathcal{P})^p.
$$
Finally, define $g=\sum_i \M_p(h,A_i) \ind_{A_i}+0\ind_{A^c} \in \mathscr{G}_{p,k+1}\cap L^p(\Omega,\F,\mu)$ to obtain that
$$
\|f\ind_A-g\|^p_p\le var_p(f,\mathcal{P})^p\le  \left(\V_{p,k}(f)\right)^p+\e
$$ 
To finish this part, notice that
\begin{align*}
\|f-g\|_p^p&=\|f\ind_A-g\ind_A\|_p^p+\|f\ind_{A^c}-g\ind_{A^c}\|_p^p \\
&=\|f\ind_A-g\|_p^p+\|f\ind_{A^c}\|_p^p\le  \left(\V_{p,k}(f)\right)^p+2\e
\end{align*}
which implies that 
$
\mathscr{D}_{p,k+1}(f)\le \V_{p,k}(f).
$
\medskip

$(ii)$ The proof is similar to $(i)$. The upper bound follows immediately from the lower bound to be proved.
For the lower estimate, in the above proof we can take $A=\Omega$.
\medskip

$(iii)$ Let $(\Omega^*_n)_n\subset\F$ and $\Omega^*=\bigcup_n \Omega^*_n$ given by Lemma \ref{sigmafinite_variation}, such that 
$$
\V_{p,k}(f)=\lim_n\V_{p,k,n}(f)=\V_{p,k,*}(f).
$$

We first assume that $f$ is bounded by some constant $C>0$. Then, using the result we have shown for the finite measure case, we have on $\Omega_n^*$
$$
\inf\{\|f|_{\Omega_n^*}-g\|_p:\ g\in \mathscr{G}_{p,k}(\Omega_n^*)\}\le \V_{p,k,n}(f).
$$
By Theorem \ref{pro:minimum}, the left hand side is attained at some function $g_n$ defined in $\Omega_n^*$, which is 
also bounded by $C$. We can assume this minimizer has the following form
$$
g_n=\sum\limits_{i=1}^{q(n)} b_{i,n} \ind_{B_{i,n}},
$$
where $\{B_{i,n}\}_{1\leq i\leq q(n)}$ is a partition of sets of positive measure of $\Omega_n^*$ and
$$
-C\le b_{1,n}<...< b_{q(n),n}\le C,
$$ 
$$
r_{1,n}=-C-1,\ r_{q(n)+1,n}=C+1,\ r_{i,n}=\frac{b_{i-1,n}+b_{i,n}}{2}\ \ \text{for}\ i\in\{2,..., q(n)\},
$$
$$
B_{i,n}=f^{-1}([r_{i,n},r_{i+1,n}))\cap \Omega_n^*\ \ \text{for}\ i\in\{1,...,q(n)-1\}\ \ \text{and}\ \ 
B_{q(n),n}=f^{-1}([r_{q(n),n},r_{q(n)+1,n}])\cap \Omega_n^*,
$$ 
$$
b_{i,n}=\M_p(f,B_{i,n})\ \ \text{for}\ i\in\{1,..., q(n)\}
$$
and $q(n)\le k$. As before, we can assume by passing to a subsequence that $q(n)=q$ is constant and 
the vector $v_n=(r_{1,n},b_{1,n},r_{2,n}, ..., r_{q,n},b_{q,n},r_{q+1,n})$ converges to a vector
in $[-C-1,C+1]^{3q}$, which we denote by $v=(r_{1},b_{1},r_{2}, ..., r_{q},b_{q},r_{q+1})$. Also
we denote by $L=(b_1,...,b_q)$ and $w_1<...<w_m$ the different values in $L$, where $m\le q$.

Let us show that $w_{t^*}=0$ for some $t^*$. For that, remark that $\Omega_n^*=\cup_{i=1}^q B_{i,n}$ and therefore, 
there exists an index $i(n)$, such that
$$
\mu(B_{i(n),n})\ge \frac1q \mu(\Omega_n^*),
$$
showing that $\lim_n \mu(B_{i(n),n})=\infty$. We can assume that $i(n)=i$ is constant, by passing to a subsequence if necessary. Using the optimality of $b_{i,n}=\M_p(f,B_{i,n})$, we get
$$
\begin{array}{ll}
|b_{i,n}|^p \mu(B_{i,n})&\hspace{-0.2cm}\le 2^{p-1} \left(\int_{B_{i,n}} |f(x)-\M_p(f,B_{i,n})|^p\ d\mu(x)+
\int_{B_{i,n}} |f(x)|^p\ d\mu(x)\right)\le 2^p \|f\|_p^p.
\end{array}
$$
This shows that $b_{i,n}\to b_i=0$, and the claim holds by taking $t^*$ such that $w_{t^*}=b_i=0$. 

Consider $I_t=\{j \in \{1,... q\}: \ b_j=w_t\}$ for $t\in\{1,...,m\}$. Notice that each $I_t$ is a nonempty interval of $I=\{1,...,q\}$. 
Assume that $I_t=\{l(t),...,u(t)\}$, then we have
$r_{l(t)+1,n}\to w_t, ..., r_{u(t),n}\to w_t, b_{l(t),n}\to w_t, ..., b_{u(t),n}\to w_t$ and 
$$
\lim_n r_{l(t),n}=r_{l(t)}=\frac{b_{l(t)-1}+w_t}{2} < w_t <\frac{b_{u(t)+1}+w_t}{2}=r_{u(t)+1}=\lim_n r_{u(t)+1,n},
$$
with the obvious modifications in the case $l(t)=1$ or $u(t)=q$. By construction we have
for all $i< l(t^*)$ it holds $r_{i+1,n}\le r_{l(t^*),n}<\frac{r_{l(t^*)}}{2}=r_-$, for all large $n$, because $r_{l(t^*)}<0$. Similarly,
for all $i\ge u(t^*)$ we have $r_{i+1,n}\ge r_{u(t^*)+1,n}>\frac{r_{u(t^*)+1}}{2}=r_+>0$, for all large $n$. This implies that, for large $n$
$$
\bigcup_{i<l(t^*)} B_{i,n}\subset f^{-1}([-C-1,r_-]),
$$
which is a set of finite measure: $\mu(f^{-1}([-C-1,r_-]))<\infty$. Consider a modification of $g_n$ given by
$$
\ell_n=\sum\limits_{i\notin [l(t^*),u(t^*)]} b_i \ind_{B_{i,n}}+\sum_{l(t^*)\le i\le u(t^*)} b_{i,n} \ind_{B_{i,n}}.
$$
We have $\|g_n-\ell_n\|_p$ converges to zero. Indeed, this follows from the inequality
$$
\begin{array}{ll}
\|g_n-\ell_n\|_p^p &\hspace{-0.2cm}=\sum\limits_{i\notin \{l(t^*),...,u(t^*)\}} |b_i-b_{i,n}|^p \mu(B_{i,n})\\
&\hspace{-0.2cm}\le \max_j|b_j-b_{j,n}|^p \mu\left(f^{-1}([-C-1,r_-])\cup f^{-1}([r_+,C+1])\right)\to 0.
\end{array}
$$
Using the triangular inequality and the optimality of $g_n$, we get 
$$
\|f|_{\Omega_n^*}-g_n\|_p\le \|f|_{\Omega_n^*}-\ell_n\|_p\le \|f|_{\Omega_n^*}-g_n\|_p+\|g_n-\ell_n\|_p,
$$
and we plan to use Fatou's Lemma. Before doing that, we will fix a subsequence with certain monotonic properties.
Since $r_{1,n}, r_{q+1,n}$ are constant, there is no restriction here. For $i\in\{2,...,q\}$ we choose a subsequence 
in the following order. If $(r_{2,n})_n$ has an strictly decreasing subsequence, we consider this as $n^{(2)}$ and define $T(2)=\text{sd}$ (for strictly decreasing)
otherwise, we consider $n^{(2)}$ so that  $(r_{2,n})_n$ is increasing along this subsequence and $T(2)=\text{in}$ (for increasing). 
Now, we construct $n^{(3)}$. If $(r_{3,n^{(2)}})_{n^{(2)}}$ 
has an strictly decreasing subsequence we take this as $n^{(3)}$ and $T(3)=\text{sd}$, 
otherwise we take $n^{(3)}$ so that $(r_{3,n^{(3)}})_{n^{(3)}}$ is increasing, and $T(3)=\text{in}$. We continue
in this way until we define $n^{(q)}$. We put $T(1)=\text{in}$ and $T(q+1)=\text{sd}$.
We call $n'=n^{(q)}$. In this way we have the a.e. convergence
$$
\ind_{f^{-1}([r_{i,n'},r_{i+1,n'}))\cap\Omega_{n'}^*}\to \begin{cases} \ind_{f^{-1}([r_{i},r_{i+1}])\cap \Omega^*}         &\hbox{if } T(i)=\text{in}, T(i+1)=\text{sd}\\
                                                                        \ind_{f^{-1}([r_{i},r_{i+1}))\cap \Omega^*}         &\hbox{if } T(i)=\text{in}, T(i+1)=\text{in}\\
                                                                        \ind_{f^{-1}((r_{i},r_{i+1}])\cap \Omega^*}         &\hbox{if } T(i)=\text{sd}, T(i+1)=\text{sd}\\
                                                                        \ind_{f^{-1}((r_{i},r_{i+1}))\cap \Omega^*}         &\hbox{if } T(i)=\text{sd}, T(i+1)=\text{in}
\end{cases}
$$
We call $\J_i$ the interval, with extremes $r_i, r_{i+1}$, according to the above classification. An important remark is that $\cup_i \J_i=[-C-1,C+1]$.

Using the Dominated Convergence Theorem we conclude that
$$
\sum\limits_{i\notin \{l(t^*),...,u(t^*)\}} \int_{B_{i,n'}} |f(x)-b_i|^p \ d\mu(x)\to \sum\limits_{i\notin \{l(t^*),...,u(t^*)\}} \int_{f^{-1}(\J_i)} |f(x)-b_i|^p \ d\mu(x).
$$
On the other hand, using Fatou's Lemma we conclude
$$
\begin{array}{ll}
\liminf\limits_{n'}\sum\limits_{i\in \{l(t^*),...,u(t^*)\}} \int_{B_{i,n'}} |f(x)-b_{i,n'}|^p \ d\mu(x)
&\hspace{-0.2cm}\ge \int \liminf\limits_{n'} |f(x)-g_n(x)|^p \ind_{B_{n'}} \ d\mu(x)\\
&\hspace{-0.2cm}\ge \int |f(x)|^p \ind_{f^{-1}(\bar \J)\cap\Omega^*} \ d\mu(x).
\end{array}
$$
where $B_{n'}=\bigcup_{i\in \{l(t^*),...,u(t^*)\}} B_{i,n'}$
and  $\widetilde{\J}=\bigcup_{i\in \{l(t^*),...,u(t^*)\}} \J_i$. Here, we have used that for all $x\in B_{n'}$ we have
$$
|g_{n'}(x)|\le \max_{i\in \{l(t^*),...,u(t^*)\}}|b_{i,n'}|\to 0.
$$
Hence, 
$$
\begin{array}{l}
|f(x)|\ind_{B_{i,n'}}\le |f(x)-g_{n'}(x)|\ind_{B_{i,n'}}+\max_{i\in \{l(t^*),...,u(t^*)\}}|b_{i,n'}|\ind_{B_{i,n'}}\\
|f(x)-g_{n'}(x)|\ind_{B_{i,n'}}\le |f(x)|\ind_{B_{i,n'}}+\max_{i\in \{l(t^*),...,u(t^*)\}}|b_{i,n'}|\ind_{B_{i,n'}},
\end{array}
$$
showing that
$$
\liminf_{n'} |f(x)-g_{n'}(x)|\ind_{B_{i,n'}}=|f(x)| \liminf_{n'} \ind_{B_{i,n'}}.
$$
Putting all together, we conclude that
$$
\V_{p,k}(f) \ge \liminf_{n'} \|f|_{\Omega_{n'}^*}-g_{n'}\|_p\ge \|f|_{\Omega^*}-\ell\|_p,
$$
where the function $\ell\in \mathscr{G}_{p,q}(\Omega^*)$ is defined on $\Omega^*$ as
$$
\ell=\sum\limits_{i\notin \{l(t^*),...,u(t^*)\}} b_i \ind_{f^{-1}(\J_i)\cap\Omega^*}+0\ind_{f^{-1}(\widetilde{\J})\cap\Omega^*}.
$$
Notice that for $i\notin \{l(t^*),...,u(t^*)\}$, we have $\mu(f^{-1}(\J_i)\cap\Omega^*)\le \mu\left(f^{-1}([-C-1,r_-]\cup f^{-1}([r_+,C+1])\right)<\infty$ and
so $\mu(f^{-1}(\widetilde{\J})\cap\Omega^*)=\infty$.

Since $b_i=0$ for some $i$ and $f=0$ outside $\Omega^*$, we can extend $\ell$ by $0$ outside $\Omega^*$ and still
this extension $\bar \ell$ belongs to $\mathscr{G}_{p,q}\subset \mathscr{G}_{p,k}$. So, we get that
$$
\mathscr{D}_{p,k}(f)\le \|f-\bar \ell\|_p\le \V_{p,k}(f).
$$
and the result is shown in the case $f$ is bounded.

Now, for the general case, consider $\e>0$ and a large $C>0$, such that $\|f\ind_{|f|>C}\|_p\le \e$. From the domination
$\V_{p,k}(\bullet)\le 2\|f\|_p$, and the seminorm property of $\V_{p,k}$ we conclude
$$
\V_{p,k}(f\ind_{|f|\le C})\le \V_{p,k}(f)+\V_{p,k}(f-f\ind_{|f|\le C})\le \V_{p,k}(f)+2\e.
$$
Using what we have shown, we get there exists and $\ell \in \mathscr{G}_{p,k}$ such that
$$
\|f\ind_{|f|\le C}-\ell\|_p\le \V_{p,k}(f\ind_{|f|\le C})\le \V_{p,k}(f)+2\e.
$$
On the other hand, we have
$$
\|f-\ell\|_p\le \|f\ind_{|f|\le C}-\ell\|_p+\|f\ind_{|f|\le C}-f\|_p\le \|f\ind_{|f|\le C}-\ell\|_p+\e,
$$
which shows that
$$
\mathscr{D}_{p,k}(f)\le \V_{p,k}(f)+3\e,
$$
and the result is shown.
\end{proof}

\begin{remark} Examples that satisfies $(iii)$ in the previous Proposition are the $\sigma$-finite measures. 
In particular, it can be applied to $\ell^p=L^p(\NN,\mathcal{P}(\NN),\delta)$, where $\delta$
is the counting measure.
But, there are non $\sigma$-finite measures that satisfies that hypothesis as well, the counting measures on any 
uncountable space. 
\end{remark}

\begin{corollary}\label{cor_limitvar}
Assume $(\Omega,\F,\mu)$ is a measurable space and $p\in[1,\infty)$. For all $f\in L^p(\Omega,\F,\mu)$
it holds that
$
\lim\limits_{k\to \infty} \V_{p,k}(f)=0.
$
\end{corollary}

\begin{proof}
This follows directly from the previous proposition since $\lim_k\mathscr{D}_{p,k}(f)=0$ by density of the simple functions.
\end{proof}

A question of some interest is when $\V_{p,k}(f)=0$, for a function $f\in L^p(\Omega,\F,\mu)$. 
Clearly, if $f\in \mathscr{G}_{p,k}$ then $\V_{p,k}(f)=0$. The next result answers the converse.

\begin{proposition} 
\label{pro:Kernel} Let $(\Omega,\F,\mu)$ be a measure space and $p\in[1,\infty)$. Let $k\geq 1$. We have
\begin{enumerate}
    \item[(i)] if $\mu$ is a general measure, then $\mathscr{G}_{p,k}\subset \V_{p,k}^{-1}(\{0\})\subset \mathscr{G}_{p,k+1}$.
     \item[(ii)] if $\mu$ has no atoms of infinite mass, then $\V_{p,k}^{-1}(\{0\})=\mathscr{G}_{p,k}$.
\end{enumerate}
\end{proposition}
\begin{proof} $(i)$. Let $f\in L^p(\Omega,\F,\mu)$ satisfying $\V_{p,k}(f)=0$. From $(i)$ of Proposition \ref{variation}, we have
$$
\mathscr{D}_{p,k+1}(f)\le \V_{p,k}(f)=0,
$$
which implies that $f=g$ a.e. for some $g\in \mathscr{G}_{p,k+1}$ (see Theorems \ref{pro:minimum} and \ref{pro:minimum2}). 
The other inclusion is obvious.

$(ii)$ The proof is similar to $(i)$ and uses $(iii)$ in Proposition \ref{variation}.
\end{proof}

\begin{remark} Notice that if $\mu$ has an atom of infinite 
mass it may happens that $\V_{p,k}(f)=0$, but $f\in \mathscr{G}_{p,k+1}\setminus \mathscr{G}_{p,k}$. Indeed, assume $\Omega=\{1,...,k+1\}$, 
where the mass of each atom in $\{1,...,k\}$ is one and the mass at atom $\{k+1\}$ is infinite. Every function
in $L^p(\Omega,\F,\mu)$ for $p\in [1,\infty)$ satisfies $f(k+1)=0$. The function $f$ given by $f(i)=i$ for
$i\in\{1,...,k\}$ and $f(k+1)=0$ belongs to $\mathscr{G}_{p,k+1}\setminus\mathscr{G}_{p,k}$. Nevertheless, $\V_{p,k}(f)=0$, 
which is exactly the case $(i)$ in Proposition \ref{pro:Kernel}. Also, this example explains why the lower bound in Proposition \ref{variation}
$(i)$ is computed over $\mathscr{G}_{p,k+1}$ and not over $\mathscr{G}_{p,k}$, in general.
\end{remark}

\section{Uniform approximability}
\label{sec:UAp}

In this section, we investigate some properties of uniformly approximable sets (see Definition \ref{def_UA}).

\subsection{Uniform integrability}

In this subsection, we prove that the class of uniform approximable sets is strictly larger than the class of uniform integrable sets. Assume that $(\Omega,\F,\mu)$ is a measure space and let $p\in[1,\infty)$. Remember that a subset $\mathscr{A}\subset L^p(\Omega,\F,\mu)$ is \textit{uniformly integrable} (in short, UI) if
$$
\inf\limits_{g\in L^p_+(\Omega,\F,\mu)} \sup\limits_{f\in \mathscr{A}} \int_{|f|>g} |f(x)|^p \ d\mu(x) =0,
$$
where $L^p_+(\Omega,\F,\mu)$ is the set of nonnegative functions in $L^p(\Omega,\F,\mu)$. Note that if $\mu$ is a finite measure, then this definition coincides with the usual one, that is $\mathscr{A}$ is UI in $L^p(\Omega,\F,\mu)$ if and only if (see \cite{Hunt}, page 254)
$$
\lim\limits_{a\to \infty} \sup\limits_{f\in \mathscr{A}} \,\,\int\limits_{\,|f(x)|\ge a} \!\!\! |f^p(x)|\, d\mu(x)=0.
$$

\begin{proposition}
\label{UI} Let $(\Omega,\F,\mu)$ be a measure space and $p\in[1,+\infty)$. Assume $\mathscr{A}\subset L^p(\Omega,\F,\mu)$ is UI. Then, $\mathscr{A}$ is UA.
\end{proposition}
\begin{proof} Consider $\e>0$, and take $g\in L^p_+(\Omega,\F,\mu)$, such that
$$
\sup\limits_{f\in \mathscr{A}} \int_{|f|>g} |f(x)|^p \ d\mu(x) \le \frac{\e^p}{3}.
$$
Fix $f\in\mathscr{A}$. Consider $n\in \NN$, large enough such that $\int_{g>n} g^p(x) \ d\mu(x)+\int_{g<\frac{1}{n}} g^p(x) \ d\mu(x)\le \frac{\e^p}{3}$. 
The set $B_n=\left\{x: \ \frac{1}{n} \le g(x) \le n\right\}\cap \{x: |f(x)|\le g(x)\}\subset C_n=\left\{x: \ \frac{1}{n} \le g(x) \le n\right\}$ has finite measure. 
Notice that over $B_n$ we have 
$|f|\le n$. Take now $k\ge 2$ such that $\left(\frac{n}{k}\right)^p \mu(C_n)\le \frac{\e^p}{3}$ and define
$$
A_i=B_n\cap\left\{x: \frac{ni}{k}\le f(x) < \frac{n(i+1)}{k}\right\}
$$
for $i\in\{-k,...,k-2\}$, $A_{k-1}=B_n\cap\left\{x: \frac{n(k-1)}{k}\le f(x) \le n\right\}$ and 
$$
h=\sum_{i=-k}^{k-1} \frac{in}{k}\ind_{A_i}+0\ind_{B_n^c}\in \mathscr{G}_{p,2k+1}.
$$ 
Then, we have
$$
\begin{array}{ll}
\int_{B_n} |f(x)-h(x)|^p \ d\mu(x)&\hspace{-0.2cm}=
\sum\limits_{i=-k}^{k-1} \int_{A_i} |f(x)-h(x)|^p \ d\mu(x)
\le \left(\frac{n}{k}\right)^p \sum\limits_{i=-k}^{k-1} \mu(A_i)\\
&\hspace{-0.2cm}\le \left(\frac{n}{k}\right)^p \mu(B_n)\le \left(\frac{n}{k}\right)^p \mu(C_n)\leq \frac{\e^p}{3} .
\end{array}
$$
On the other hand, $B_n^c=(C_n^c\cap \{x: |f(x)|\le g(x)\})\cup \{x: |f(x)|> g(x)\}$ and so
$$
\begin{array}{ll}
\int_{B_n^c} |f(x)-h(x)|^p \ d\mu(x)&\hspace{-0.2cm}=\int_{B_n^c} |f(x)|^p \ d\mu(x)=\int_{C_n^c \cap |f|\le g} |f(x)|^p \ d\mu(x)+
\int_{|f|>g} |f(x)|^p \ d\mu(x)\\
\\
&\hspace{-0.2cm}\le \int_{C_n^c} g^p(x) \ d\mu(x)+\int_{|f|>g} |f(x)|^p \ d\mu(x)\le \frac{2\e^p}{3}
\end{array}
$$
Finally, we have $\|f-h\|_p\le \e$, and the result is shown.
\end{proof}

\begin{remark}
\label{not_UI}
Note that the converse of Proposition \ref{UI} is not true in general. In fact $\mathscr{G}_{p,2}$ is UA, but this
set is not UI in $L^p(\Omega,\F,\mu)$ in general. Indeed, assume the space has finite measure and there exists a sequence 
$(B_n)_n$ of measurable sets with positive 
measure such that $\mu(B_n)\to 0$. Then $f_n=\mu(B_n)^{-1/p}\ind_{B_n}$ belongs to $\mathscr{G}_{p,2}$, each one has norm 1 and the 
subfamily $(f_n)_n$ is not UI, since for all $a\ge 0$, we have
$$
\sup_n \int\limits_{f_n>a} f_n^p \, d\mu=1.
$$
The conclusion is that UA is weaker than UI.
\end{remark}

\begin{remark}
If $(\Omega,\F,\mu)$ is a finite measure space, the following examples are UI in $L^p(\Omega,\F,\mu)$, for $p\in [1,\infty)$,
\begin{enumerate}
    \item[-] $\mathscr{A}$ is bounded in $L^q(\Omega,\F,\mu)$ for some $q>p$;
    \item[-] $\mathscr{A}$ is bounded by a fixed function $g\in L^p(\Omega,\F,\mu)$.
\end{enumerate}
\end{remark}

The following result can be prove using the fact that totally boundedness implies UI in $L^p(\Omega,\F,\mu)$ for $p\in [1,\infty)$. 
However, since the case $p=\infty$ needs a proof, we give a more direct argument:

\begin{proposition}
\label{totally bounded}
Let $(\Omega,\F,\mu)$ be a measure space and $p\in[1,+\infty]$. 
If $\mathscr{A}$ is totally bounded in $L^p(\Omega,\F,\mu)$ then $\mathscr{A}$ is UA.
\end{proposition}

\begin{proof} Let $\varepsilon>0$. There exist finitely many functions $f_1,...,f_n$ such that
$\mathscr{A}\subset\bigcup_{j=1}^n B(f_j,\varepsilon)$. By density of the simple functions, there exist $k\in\mathbb N$ and
$g_j\in\mathscr{G}_{p,k}$ such that $\|f_j-g_j\|_p\le \varepsilon$ for all $j\in\{1,...,n\}$. Now if $f\in\mathscr{A}$ then, there exists
$j_0=j_0(f)\in\{1,...,n\}$ such that $\|f-f_{j_0}\|_p\leq\varepsilon$. 
It follows that 
$$
\|f-g_{j_0}\|_p\leq\|f-f_{j_0}\|_p+\|f_{j_0}-g_{j_0}\|_p\le 2\varepsilon
$$ 
and the proof is complete.
\end{proof}

\subsection{Characterization of the uniform approximability}

If $M$ is a metric space, we recall that the covering numbers of $M$ are defined 
for every $\varepsilon>0$ by 
$$\mathcal{N}(M,\varepsilon)=\inf\left\{N\geq 1\ :\ M\ 
\text{can be covered by}\ N\ \text{closed balls of radius}\ \varepsilon\right\}.
$$ 
For more informations about covering numbers and its applications to Machine Learning, we refer the reader to \cite{Wainwright} and \cite{Anthony}.

If $(\Omega,\F,\mu)$ is a measure space, we define the covering numbers of a measurable function $f$ by 
$$
\mathcal{N}(f,\varepsilon)=\inf\{\mathcal{N}(g(\Omega),\varepsilon)\ :\ g\ \text{measurable function such that}\ f=g\ \text{a.e.}\}.
$$
This notion allows us to caracterize the uniform approximability in 
$L^\infty(\Omega,\F,\mu)$ in terms of uniformly bounded covering numbers. 
Before doing that, we notice that if $f\in L^\infty(\Omega,\F,\mu)$, then 
$\mathcal{N}(f,\varepsilon)<\infty$. Indeed, we know that 
$|f|\le \|f\|_\infty$ holds a.e., so by considering $g=f\ind_{\{|f|\le
\|f\|_\infty\}}$, we have $g=f$ a.e. and
$$
\mathcal{N}(f,\varepsilon)\le \frac{2}{\e}\|f\|_\infty+1.
$$
If $f$ is a measurable function and 
$\mathcal{N}(f,\varepsilon)<\infty$ then $f\in L^\infty(\Omega,\F,\mu)$. 
On the other hand, by definition of infimum, there exists a measurable function 
$g$ such that $f=g$ a.e., and
$$
\mathcal{N}(f,\varepsilon)\le \mathcal{N}(g(\Omega),\varepsilon)
\le \mathcal{N}(f,\varepsilon)+\frac12,
$$
showing that $\mathcal{N}(f,\varepsilon)= \mathcal{N}(g(\Omega),\varepsilon)$.

\begin{theorem}\label{L_infty}
Let $(\Omega,\F,\mu)$ be a measure space and let $\mathscr{A}\subset L^\infty(\Omega,\F,\mu)$. 
The following assertions are equivalent:
\begin{enumerate}
    \item[(i)] $\mathscr{A}$ is UA;
    \item[(ii)] $\sup_{f\in\mathscr{A}}\mathcal{N}(f,\varepsilon)<\infty$ for all $\varepsilon>0$.
\end{enumerate}
In this case, we have that $N_{\infty,\varepsilon}(\mathscr{A})= \sup_{f\in\mathscr{A}}\mathcal{N}(f,\varepsilon)$ 
for all $\varepsilon>0$.
\medskip
\end{theorem}

\begin{proof} 
Let $\varepsilon>0$ and suppose $\sup_{f\in\mathscr{A}}\mathcal{N}(f,\varepsilon)=\infty$. 
Fix $k\geq 1$ and  choose $f\in\mathscr{A}$ such that $\mathcal{N}(f,\varepsilon)\ge 10(k+1)$. 
Changing the representant of $f$ is necessary, we can suppose that 
$m_\varepsilon:=\mathcal{N}(f,\varepsilon)=\mathcal{N}(f(\Omega),\varepsilon)$. 
So there exists a collection $\mathcal{J}$ of closed balls $\{I_i=[a_i,b_i]\}_{1\leq i\leq m_\e}$, 
of radius $\e$, such that 
$$
f(\Omega)\subset\bigcup_{i=1}^{m_\varepsilon} I_i
$$ 
Using the minimality of this covering each interval cannot be covered by the 
other intervals, so for $i\neq j$ we have $I_i\setminus I_j\neq \emptyset$. Consider the
measurable sets $A_i=f^{-1}(I_i)$. We shall prove that the minimality of $\mathcal{J}$ implies that
$\mu(A_i)>0$. Indeed, assume that for some $i$ we have $\mu(A_i)=0$. Take any $j\neq i$ (notice that we have assumed
that $m_\e$ is at least $10(k+1)>2$) and 
$a\in I_j\setminus I_i$. The measurable function
$$
h=f\ind_{A_i^c}+a\ind_{A_i}
$$ 
coincides with $f$ up to measure $0$ and $h(\Omega)\subset\bigcup\limits_{r\neq i} I_r$, so 
$\mathcal{N}(h(\Omega),\varepsilon)\le m_\e-1$, which is a contradiction. 

We say that a subcollection $\mathcal{C}\subset \mathcal{J}$ is $\e$-separated if for two
different intervals $I,J \in \mathcal{C}$, we have the distance between them $d(I,J)$ is
greater than $\e$. Notice that a collection with only one interval from $\mathcal{J}$ is
$\e$-separated. Take $\mathcal{C}^*$ a maximal $\e$-separated
subcollection with respect to inclusion. Now, if $I\in\mathcal{J} \setminus \mathcal{C}^*$ 
there exists an interval
$L=[a,b]\in \mathcal{C}^*$ such that $d(I,L)\le \e$, otherwise the maximality of 
$\mathcal{C}^*$ is contradicted.
Then,
$$
I\subset [a-4\e,a-2\e]\cup [a-2\e,a]\cup [a,b]\cup [b,b+2\e]\cup [b+2\e,b+4\e],
$$
showing that the collection
$$
\mathcal{D}=\{[a_i-4\e,a_i-2\e],[a_i-2\e,a_i],[a_i,b_i],[b_i,b_i+2\e],
[b_i+2\e,b_i+4\e]: [a_i,b_i]\in \mathcal{C}^*\},
$$
is a covering of $f(\Omega)$ with closed balls of radius $\e$. Therefore
$$
m_\e\le |\mathcal{D}|\le 5 |\mathcal{C}^*|,
$$
showing that $n=|\mathcal{C}^*|\ge \frac15 m_\e$ (here $|\mathcal{C}^*|$
is the cardinal of $\mathcal{C}^*$).
\medskip

Consider now $g\in \mathscr{G}_{\infty,k}$. We say that an interval $I\in \mathcal{C}^*$ is unmarked if $d(g(\Omega),I)> \e$. 
There are at least
$n-2k\ge \frac15 m_\e-2k=2(\frac{1}{10} m_\e-k)>1$ unmarked intervals in $\mathcal{C}^*$. 
Consider $I_i\in \mathcal{C}^*$ 
any unmarked interval, then for all $x\in A_i=f^{-1}(I_i)$, we have
$$
|f(x)-g(x)|> \e.
$$
Since $\mu(A_i)>0$, we conclude that $\|f-g\|_\infty>\e$ and therefore 
$N_{\infty,\varepsilon}(\mathscr{A})> \frac{m_\e}{10}-1$, showing that 
$N_{\infty,\varepsilon}(\mathscr{A})=\infty$. So we have proved that $(i)$ implies $(ii)$.
\medskip

Now, let us show that $(ii)$ implies $(i)$. So, we are assuming that
$M_\varepsilon=\sup_{f\in\mathscr{A}}\mathcal{N}(f,\varepsilon)<\infty$ for all
$\varepsilon>0$. Fix $\varepsilon>0$ and let $f\in\mathscr{A}$. Suppose that
$m_\varepsilon:=\mathcal{N}(f,\varepsilon)=\mathcal{N}(f(\Omega),\varepsilon)\leq
M_\varepsilon$. Again we can write
$$
f(\Omega)\subset\bigcup_{i=1}^{m_\varepsilon}I_i
$$ 
where $I_i=[a_i,b_i]$ are closed balls of radius $\e$. We assume that the left extremes 
are ordered increasingly: $a_1<a_2< ... <a_{m_\e}$. We define recursively  $\tilde 
a_1=a_1, \tilde b_1=b_1$ and for $i\ge 2$
$$
\tilde a_i=\max\{b_{i-1}, a_i\},\, \tilde b_i=b_i.
$$
Define $\tilde I_i=[\tilde a_i,\tilde b_i]$ for $i\in\{1,..., m_\e\}$.
The fact that every interval $I_i$ cannot be covered by the intervals $\{I_j\}_{j\neq i}$
allows us to show the following facts about the new intervals $\{\tilde I_i\}_{1\leq i\leq m_\e}$
$$
\begin{array}{l}
\forall i\,\ \tilde I_i=I_i\setminus \bigcup\limits_{j=1}^{i-1} [a_j,b_j), \, 
\forall i\, \bigcup\limits_{j=1}^{i} I_j=\bigcup\limits_{j=1}^{i} \tilde I_j,\\
\\
\text{int}(\tilde I_i)=(\tilde a_i,\tilde b_i)\neq \emptyset,\\
\\
\forall i<j: \  \tilde I_i \cap \tilde I_j\subset 
\begin{cases} \emptyset &\hbox{if } j-i\ge 2\\
\\
              \{\tilde b_i\} &\hbox{if } j=i+1
\end{cases}.
\end{array}
$$
Thus, $\{\tilde I_i\}_{1\leq i\leq m_\e}$ is a collection of closed balls of radii at most $\e$, that 
covers $f(\Omega)$, which are disjoint except for consecutive intervals that can intersects at one 
extreme.
\medskip

With this new intervals we can produce a partition of $f(\Omega)$, by choosing
$\widehat I_1=\tilde I_1$ and for $i\ge 2$
$$
\widehat I_i=\begin{cases} (\tilde a_i,\tilde b_i] &\hbox{if } \tilde I_{i}\cap\tilde I_{i-1}\neq \emptyset\\
\\
\tilde I_i &\hbox{otherwise}
\end{cases}
$$
We now define $\widehat A_i=f^{-1}(\widehat I_i)$, which is a partition of $\Omega$ (maybe some of
them are empty). If $x\in \widehat A_i$ then 
$f(x) \in \widehat I_i\subset I_i=[a_i,b_i]$ and therefore $|f(x)-\frac{a_i+b_i}{2}|\le \e$. 
Define the simple function
$$
g(x)=\sum\limits_{i=1}^{m_\e} \frac{a_i+b_i}{2} \, \ind_{\widehat A_i},
$$
that belongs to $\mathscr{G}_{\infty,m_\e}$ and satisfies for all $x\in \Omega$
$$
|f(x)-g(x)|\le \e,
$$
showing that $\|f-g\|_\infty\le \e$. We conclude that 
$N_{\infty,\e}(\mathscr{A})\le \sup_{f\in\mathscr{A}}\mathcal{N}(f,\varepsilon)$.
\medskip

To finish, we prove that $N_{\infty,\e}(\mathscr{A})= \sup_{f\in\mathscr{A}}\mathcal{N}(f,\varepsilon)$.
For that purpose consider $k=N_{\infty,\e}(\mathscr{A})$, which means
that for all $f\in \mathscr{A}$, there exists $g\in \mathscr{G}_{\infty,k}$, such that
$\|f-g\|_\infty\le \e$. We assume that $g=\sum_{i=1}^k c_i \ind_{B_i}$, where $(B_i)_{i=1}^k$ is
a partition of $\Omega$. For any $i\in\{1,...,k\}$ we have 
$$
\|(f-c_i)\ind_{B_i}\|_\infty \le \|f-g\|_\infty\le \e,
$$
which means that $A_i=\{x\in B_i:\, |f(x)-c_i|>\e\}$ is a measurable set of measure $0$. 
Since $\mu(\Omega)>0$, not all the sets $B_j$ can have measure $0$, so we assume without 
loss of generality that $\mu(B_1)>0$. Consider 
$B=\Omega\setminus \bigcup\limits_{i=1}^k A_i$, 
$h=f\ind_B+c_1\ind_{B^c}$ and $\widetilde g=c_1\ind_{B_1\cup B^c}+\sum_{i=2}^k c_i \ind_{B_i\setminus A_i}$.
We notice that $f=h$ a.e. and $\widetilde g=g$ a.e. On the other hand, 
$B_1\cup B^c, B_2\setminus A_2,..., B_k\setminus A_k$ is a partition and $\widetilde g\in \mathscr{G}_{\infty,k}$. 
Also, it is clear that
$B_1\setminus A_1,..., B_k\setminus A_k, B^c$ is a partition and
$$
\widetilde g= g\ind_B+c_1\ind_{B^c}.
$$
With these modifications, we have for all $x\in \Omega$
$$
|h(x)-\widetilde g(x)|\le \e.
$$
This is clear for $x\in B$. For $x\in B^c$, we have $h(x)=c_1=\widetilde g(x)$ and the claim is shown. 
Finally, the collection
of closed ball of radius $\e$ given by: $\{[c_i-\e,c_i+\e]\}_{1\leq i\leq k}$ is an 
$\e$-cover of $h(\Omega)$, showing that $\mathcal{N}(f,\varepsilon)\le k$. The conclusion is that
$$
\sup_{f\in\mathscr{A}}\mathcal{N}(f,\varepsilon)\le k=N_{\infty,\e}(\mathscr{A}),
$$
and the result is shown.
\end{proof}

\begin{corollary}\label{UA_lq}
Let $(\Omega,\F,\mu)$ be a finite measure space and $\mathscr{A}$ be a 
set of measurable functions. Assume that $\mathscr{A}$ is UA in $L^q(\Omega,\F,\mu)$ for
some $q\in [1,\infty]$, then $\mathscr{A}$ is UA in $L^p(\Omega,\F,\mu)$ for
all $p\in [1,q]$ and for all $\varepsilon>0$ it holds 
$$
N_{p,\varepsilon}(\mathscr{A})\leq N_{q,\varepsilon\mu(\Omega)^{-r}}(\mathscr{A}),
$$
where $r=\frac{1}{p}-\frac{1}{q}$.
\medskip

In particular if $\sup_{f\in\mathscr{A}}\mathcal{N}(f,\varepsilon)<\infty$ for all 
$\varepsilon>0$, then $\mathscr{A}$ is UA in $L^p(\Omega,\F,\mu)$ for all 
$p\in [1,\infty]$ and for all $\varepsilon>0$ it holds $N_{p,\varepsilon}(\mathscr{A})\leq
\sup_{f\in\mathscr{A}}\mathcal{N}(f,\varepsilon\mu(\Omega)^{\frac{-1}{p}})$. 
\end{corollary}

\begin{proof} This is a direct consequence of H\"older's inequality. In fact, assume that $p\le q$ and consider
$g\in \mathscr{G}_{p,k}, f\in \mathscr{A}$ then, we have
$$
\|f-g\|_p\le \|f-g\|_q (\mu(\Omega))^{r}
$$
where $r=\frac{1}{p}-\frac{1}{q}$. From this it follows that $\mathscr{A}$ is UA in $L^p(\Omega,\F,\mu)$
and 
$$
N_{p,\varepsilon}(\mathscr{A})\leq N_{q,\varepsilon\mu(\Omega)^{-r}}(\mathscr{A}).
$$

The second assertion follows from Theorem \ref{L_infty}
\end{proof}

The previous result gives a large class of UA sets when the measure is finite. 
For example suppose that $\Omega$ is a bounded metric space, $\F$ is the Borel
$\sigma$-algebra and $\mu$ is a finite measure on $\F$. Then the set of $1$-Lipschitz
functions is UA in $L^p(\Omega,\F,\mu)$ for any $p\in[1,+\infty]$.\\

The following result is a characterization of UA in $L^p$ for $p\in [1,\infty)$, where we shall prove that a class is UA if and only 
$\V_{p,k}(f)$ converges toward $0$, when $k\to \infty$, uniformly in the class. 

\begin{theorem}
Let $(\Omega,\F,\mu)$ be a measure space, $p\in[1,\infty)$ and let $\mathscr{A}\subset L^p(\Omega,\F,\mu)$. 
Then, the following are equivalent
\begin{itemize}
\item[(i)] $\mathscr{A}$ is UA in $L^p(\Omega,\F,\mu)$;
\item[(ii)] $\lim_{k\to \infty} \sup_{f\in \mathscr{A}} \V_{p,k}(f) =0$.
\end{itemize}
In this case if we define $r_\varepsilon(\mathscr{A})=\min\{k\in\NN\ :\ \sup_{f\in \mathscr{A}} 
\V_{p,k}(f)\le\varepsilon\}$, we have that for all $\varepsilon>0$
$$
r_{2\varepsilon}(\mathscr{A})\leq 
N_{p,\varepsilon}(\mathscr{A})\leq r_{\varepsilon}(\mathscr{A})+1.
$$ 

Moreover, if the measure $\mu$ is finite both properties (i), (ii)  are equivalent to
\begin{itemize}
\item[(iii)] $\lim_{k\to \infty} \sup_{f\in \mathscr{A}} \V_{p,k}(f,\Omega) =0.$
\end{itemize}
In this case if we define $m_\e(\mathscr{A})=\min\{k\in\NN\ :\ \sup_{f\in \mathscr{A}} 
\V_{p,k}(f,\Omega)\le\varepsilon\}$, we have that for all $\varepsilon>0$
$$
m_{2\varepsilon}(\mathscr{A})\leq r_{2\e}(\mathscr{A})\leq  
N_{p,\varepsilon}(\mathscr{A})\leq m_{\varepsilon}(\mathscr{A})\le r_{\varepsilon}(\mathscr{A})
$$ 
\end{theorem}

\begin{proof}
Suppose that $\mathscr{A}$ is UA and fix $\varepsilon>0$. Then we have that $\mathscr{D}_{p,k}(f)\leq\varepsilon$ 
for all $f\in\mathscr{A}$, where $k=N_{p,\varepsilon}(\mathscr{A})$. By Proposition \ref{variation}, we deduce that 
$\V_{p,k}(f)\leq 2\varepsilon$ for all $f\in\mathscr{A}$. It follows that 
$r_{2\varepsilon}(\mathscr{A})\leq N_{p,\varepsilon}(\mathscr{A})$, implying that $(ii)$ holds. Now suppose that $(ii)$ holds. 
Using Proposition \ref{variation} again, it is easy to see that $N_{p,\varepsilon}(\mathscr{A})\leq r_{\varepsilon}(\mathscr{A})+1$, 
from what we deduce that $(i)$ is true. In the case that $\mu$ is finite, the equivalence between $(ii)$ and $(iii)$ and 
the last assertion of the theorem follow directly from Proposition \ref{variation}.
\end{proof}

\subsection{The unit ball of $L^p$}
\label{UballisUA}

Now we investigate when the unit ball of $L^p(\Omega,\F,\mu)$, denoted by 
$B_{L^p(\Omega,\F,\mu)}=\{f\in L^p:\, \|f\|_p\le 1\}$, is UA. The case $p=\infty$ is simple:

\begin{proposition}\label{UA_infty}
Let $(\Omega,\F,\mu)$ be a measure space. Then $B_{L^\infty(\Omega,\F,\mu)}$ is UA. More precisely 
we have that $N_{\infty,\varepsilon}(B_{L^\infty(\Omega,\F,\mu)})\leq\left[\frac{2}{\varepsilon}\right]+1$ 
(where $[.]$ is the integer part) for all $\varepsilon>0$.
\end{proposition}

\begin{proof}
It is a direct consequence of theorem \ref{L_infty}.
\end{proof}

The main objective of this section is to prove the following result:

\begin{theorem}\label{ballUA}
Let $(\Omega,\F,\mu)$ be a measure space and $p\in[1,+\infty)$. The following assertions are equivalent:
\begin{enumerate}
    \item[(i)] $B_{L^p(\Omega,\F,\mu)}$ is UA;
    \item[(ii)] $L^p(\Omega,\F,\mu)$ is finite dimensional;
    \item[(iii)] $\mu$ is atomic and has only a finite number of atoms with finite measure, up to measure $0$.
\end{enumerate}
More precisely, if the previous assertions are false then $N_{p,\varepsilon}(B_{L^p(\Omega,\F,\mu)})=\infty$ for all $\varepsilon\in(0,1)$.
\end{theorem}

This theorem will be proved thanks to several intermediary results. We start with the following result:

\begin{proposition}\label{The:1}
Let $(\Omega,\F,\mu)$ be a measure space and $p\in[1,+\infty)$. Suppose that there exists a sequence of 
disjoint measurable sets $(A_n)_n$ of positive measure such that $\mu(A_n)\to 0$. 
Then $B_{L^p(\Omega,\F,\mu)}$ is not UA. More precisely, we have that 
$N_{p,\varepsilon}(B_{L^p(\Omega,\F,\mu)})=\infty$ for all $\varepsilon\in(0,1)$.
\end{proposition}

\begin{proof}
We are going to prove that for all $k\geq 1$
\begin{equation}
\label{eq:1}
\sup\limits_{f\in L^p(\Omega,\F,\mu), f\neq 0} \inf\limits_{h\in \mathscr{G}_{p,k}} \frac{\|f-h\|_p^p}{\|f\|_p^p}=1.
\end{equation}
Note that this equality implies easily that $N_{p,\varepsilon}(B_{L^p(\Omega,\F,\mu)})=\infty$ for all $\varepsilon\in(0,1)$.
Consider $r\geq 2$ and consider a subsequence $(n_k)_k$ such that $\mu(A_{n_k})\in (r^{-n_k-1},r^{-n_k}]$.
We further assume that $n_{k+1}-n_k\ge 2$.
With this sequence we consider
$$f_N(x)=\sum\limits_{j=1}^N \mu(A_{n_j})^{\frac{-1}{p}}\ind_{A_{n_j}}(x).$$ for all $N\geq 1$ and note that $\|f_N\|_p^p=N$. 
Let $h\in \mathscr{G}_{p,k}$ and $N>2+2k$. We say that an index $1<j< N$ is unmarked if 
$\text{Im}(h)\cap \left(r^{\frac{n_{j-1}}{p}},r^{\frac{n_{j+1}}{p}}\right)=\emptyset$. Note that there are at least $N-2-2k$ unmarked indexes. 
For such unmarked index $j$, we have for $x\in A_{n_j}$
$$
\begin{array}{l}
f_N(x)= \mu(A_{n_j})^{\frac{-1}{p}}\ge r^{\frac{n_j}{p}}>r^{\frac{n_j-1}{p}}> r^{\frac{n_{j-1}}{p}}\ge h(x), \hbox{ or}\\
\\
f_N(x)= \mu(A_{n_j})^{\frac{-1}{p}}< r^{\frac{n_j+1}{p}}< r^{\frac{n_{j+1}}{p}}\le h(x).
\end{array}
$$
In the first case we have 
$$f_N(x)-h(x)\ge f_N(x)-r^{\frac{n_j-1}{p}}\ge f_N(x)-r^{\frac{-1}{p}}f_N(x)=\left(1-r^{\frac{-1}{p}}\right)f_N(x).$$
In the second case we get
$$ h(x)-f_N(x)\ge r^{\frac{n_{j+1}}{p}}-f_N(x)\ge r^{\frac{n_j+2}{p}}-f_N(x)=r^{\frac{1}{p}} r^{\frac{n_j+1}{p}}-f_N(x)> \left( r^{\frac{1}{p}}-1\right)f_N(x).$$ Notice that
$$\theta=1-r^{\frac{-1}{p}}=\frac{r^{\frac{1}{p}}-1}{r^{\frac{1}{p}}}<r^{\frac{1}{p}}-1.$$
So, we have $|f_N(x)-h(x)|\ge \theta f_N(x)$ on $A_{n_j}$.
Then, we conclude that
$$\|f_N-h\|_p^p\ge \sum\limits_{j: \hbox{unmarked}} \int_{A_{n_j}} |f_N(x)-h(x)|^p \, dx
\ge \theta^p \sum\limits_{j: \hbox{unmarked}} \int_{A_{n_j}} |f_N(x)|^p \, dx\ge \theta^p(N-2-2k).$$
Hence, we have
$$\frac{\|f_N-h\|_p^p}{\|f_N\|_p^p}\ge \theta^p \frac{N-2-2K}{N},$$
and we get 
$$\sup\limits_{f\in L^p(\Omega,\F,\mu), f\neq 0} \inf\limits_{h\in \mathscr{G}_{p,k}} \frac{\|f-h\|_p^p}{\|f\|_p^p}\ge \theta^p=(1-r^{\frac{-1}{p}})^p.$$
Now, it is enough to make $r\uparrow \infty$.
\end{proof}

An inmediate corollary is obtained for diffuse measures.

\begin{corollary}
\label{UA_diffuse}
Assume that $\mu$ is a non trivial diffuse measure, 
then for all $p\in[1,\infty)$ the unit ball $B_{L^p(\Omega,\F,\mu)}$ 
is not UA and  $N_{p,\varepsilon}(B_{L^p(\Omega,\F,\mu)})=\infty$,
for all $\varepsilon\in(0,1)$.
\end{corollary}

\begin{proof} This follows directly from Sierpiński's theorem (see \cite{Sierpinski}). In fact, consider a measurable 
set $B_0$ such that $0<\mu(B_0)=a<\infty$ (if such set
does not exists then $\Omega$ is an atom of $\mu$). Then, there
exists $B_1\subset B_0$ such that $\mu(B_1)=\frac{a}{2}$. Applying
the same idea to $B_0\setminus B_1$, there exists 
$B_2\subset (B_0\setminus B_1)$ such that 
$\mu(B_2)=\frac{\mu( B_0\setminus B_1)}{2}=\frac{a}{4}$. 
Inductively, we construct a sequence of disjoint subsets 
$(B_k)_k$ such that 
$$
B_{k+1}\subset B_0\setminus \bigcup\limits_{i=1}^k B_i,
$$
and 
$$
\mu(B_{k+1})=\frac{\mu\left( B_0\setminus \bigcup\limits_{i=1}^k
B_i\right)}{2}=\frac{a}{2^{k+1}}
$$ 
for all $k\in\mathbb N$. The result follows from the previous Proposition.
\end{proof}

\begin{proposition}
\label{atomic_UA}
Assume that $(\Omega,\F,\mu)$ is an atomic measure space and $p\in [1,\infty)$. 
Then the following are equivalent:
\begin{itemize}
    \item[(i)] $\mu$ has a finite number of atoms of finite measure, up to measure $0$;  
    \item[(ii)] The space  $L^p(\Omega,\F,\mu)$ is finite dimensional; 
    \item[(iii)] The unit ball $B_{L^p(\Omega,\F,\mu)}$ is UA.
\end{itemize}

Moreover, if the previous assertions are false then
$N_{p,\varepsilon}(B_{L^p(\Omega,\F,\mu)})=\infty$ for 
all $\varepsilon\in(0,1)$.
\end{proposition}

\begin{proof} Assume $(i)$ holds. Denote by $\{A_k\}_{1\leq k\leq n}$ a finite collection
of atoms of finite measure, such that all other atom $C$ of finite measure coincides
with some of them up to measure $0$. Take $B=\Omega\setminus \cup_{k=1}^n A_k$. If $\mu(B)>0$
there there exists an atom $C\subset B$. This atom $C$
satisfies that $\mu(C\setminus A_k)=\mu(C)>0$ and it cannot
coincide with $A_k$ up to measure $0$. Then $C$ has infinite
measure. Then either $\mu(B)=0$ or $\mu(B)=\infty$ and
contains no measurable subset of positive finite measure.
Then, $L^p(\Omega,\F,\mu)$ is generated
by the finite collection $\{\ind_{A_k}\}_{1\leq k\leq n}$, 
so $(ii)$ holds.
Clearly  $(ii) \Rightarrow (iii)$. 

So, for the rest of the proof we assume that there exists a countable collection of disjoint atoms 
$(A_n)_n$ each one of finite positive measure. 
Here there are two different situations. The first one 
is the existence of an infinite subsequence of atoms
$(A_{n_k})_k$ such that $\mu(A_{n_k})\to 0$. Then, we 
can apply Theorem \ref{The:1}, to conclude that the 
unit ball $B_{L^p(\Omega,\F,\mu)}$ is not UA.

The second possibility is the existence of a constant $a>0$ such that 
$\mu(A_n)\ge a$, for all $n$. We now procede to prove that 
$B_{L^p(\Omega,\F,\mu)}$ is not UA. We do it for $p=1$, the other cases are 
treated similary.

In what follows we fix $k\ge 2$ and $R>1$, and we 
consider the partial sums 
$$
S_i=\sum\limits_{2^{i-1}\le j<2^i} \mu(A_j)\ge a 2^{i-1},
$$
for $i\ge 1$, and we construct a strictly increasing 
sequence of integers $(t_q)_q$ such that the interval 
$[R^{t_q},R^{t_q + 1})$ contains at least one of these 
partial sums. We call
$S_{i_q}$ any such partial sums, for example the smallest
one, that is, for $q$ such that 
$[R^{t_q},R^{t_q + 1})\cap \{S_i\}_{i\geq 1}\neq \emptyset$, we take
$$
i_q=\min\{r\in \NN: R^{t_q}\le S_{r}<R^{t_q+1} \}.
$$
We also define 
$$
B_q=\bigcup\limits_{j=2^{i_q-1}}^{2^{i_q}-1} A_j,
$$
the union of the atoms that has mass $S_{i_q}$. We consider the function
$$
f=\sum\limits_{q=3}^{M+2} R^{-t_q}\, \ind_{B_q},
$$
where $M$ is a large integer. For the moment we choose 
$M>2kR$. Take $h\in \mathscr{G}_{1,k}$ and as before we say 
that $3\le q\le M+2$ is an unmarked index if 
$$
\text{Im}(h) \cap (R^{-t_q-1},R^{-t_q+1})=\emptyset.
$$
There are at least $M-2k$ unmarked indexes. For an 
unmarked index $q$ and $x\in B_q$, we either have 
$$
\begin{array}{l}
f(x)-h(x)\ge R^{-t_q}-R^{-t_q-1}\ge f(x)\left(1-\frac{1}{R}\right)=f(x) 
\frac{R-1}{R}, \hbox{ or}\\
\\
h(x)-f(x)\ge R^{-t_q+1}-R^{-t_q}\ge f(x)(R-1).
\end{array}
$$
In any case, we have for $x\in B_q$
$$
|f(x)-h(x)|\ge f(x) \frac{R-1}{R},
$$
and then
$$
\begin{array}{ll}
\|f-h\|_1& \ge \frac{R-1}{R}
\sum\limits_{q: \small{\hbox{ unmarked }}} 
R^{-t_q} \mu(B_q)
= \frac{R-1}{R} \left( \sum\limits_{q} R^{-t_q} \mu(B_q)
-\sum\limits_{q: \small{\hbox{marked }}} 
R^{-t_q} \mu(B_q)\right)\\
\\
&\ge \frac{R-1}{R}(\|f\|_1-2kR)=\|f\|_1\frac{R-1}{R}
\left(1-\frac{2kR}{\|f\|_1}\right)
\end{array}
$$
Now, we estimate the norm of $f$.
Clearly, we have $\|f\|_1=\sum_q R^{-t_q} \mu(B_q)$, 
which gives the lower estimate
\begin{equation}
\label{eq:1a}
M\le \|f\|_1,
\end{equation}
and then the lower bound
\begin{equation}
\label{eq:2}
\|f-h\|_1\ge \|f\|_1 
\frac{R-1}{R}\left(1-\frac{2kR}{M}\right).
\end{equation}
So, we conclude that for $\tilde f=f/\|f\|$
$$
\inf\{\|\tilde f-g\|_1:\, g\in \mathscr{G}_{1,k}\}\ge 
\frac{R-1}{R}\left(1-\frac{2kR}{M}\right),
$$
and therefore
$$
\sup\limits_{f\in B_{L^1(\Omega,\F,\mu)}}
\inf\{\|f-g\|_1:\, g\in \mathscr{G}_{1,k}\}\ge 
\frac{R-1}{R}\left(1-\frac{2kR}{M}\right)
$$
Taking $M\uparrow \infty$, we conclude that
$$
\sup\limits_{f\in B_{L^1(\Omega,\F,\mu)}}
\inf\{\|f-g\|_1:\, g\in \mathscr{G}_{1,k}\}\ge 
\frac{R-1}{R}.
$$
Now we take $R\uparrow \infty$, to get finally that
$$
\sup\limits_{f\in B_{L^1(\Omega,\F,\mu)}}
\inf\{\|f-g\|_1:\, g\in \mathscr{G}_{1,k}\}\ge 1
$$
independently of $k\ge 2$. For $k=1$, we point out that $\mathscr{G}_{1,1}=\{0\}$ and so
$$
\sup\limits_{f\in B_{L^1(\Omega,\F,\mu)}}
\inf\{\|f-g\|_1:\, g\in \mathscr{G}_{1,1}\}=1.
$$
Hence, $N_{1,\varepsilon}(B_{L^1(\Omega,\F,\mu)})=\infty$, for all $\varepsilon<1$.
\end{proof}
\medskip

In order to prove Theorem \ref{ballUA}, we shall use a result in \cite{decomposition}, where the notion of atomic and nonatomic are different from the (standard)
notions we are using. In this discussion we add an $*$ to distinguish the notions we are using and the corresponding in \cite{decomposition}. 
According to \cite{decomposition} a measurable set $A$ is an \textit{$*$-atom} if $\mu(A)>0$ and for all $E\in \F$ either $\mu(A\cap E)=0$ or
$\mu(A\setminus E)=0$. It is direct to show that if $A$ is an $*$-atom for $\mu$, then it is an atom for $\mu$. Indeed, assume that 
$B\subset A$ satisfies $\mu(B)<\mu(A)$, then $\mu(A\setminus B)=\mu(A)-\mu(B)>0$ and we conclude that $0=\mu(A\cap B)=\mu(B)$, 
proving that $A$ is an atom for $\mu$.
The converse is not always true (see the example below). It is true if $A$ has finite measure. In fact, suppose that $A$ is 
an atom of finite measure and let $E$ be a measurable set. If $\mu(A\cap E)>0$ then $\mu(A)=\mu(A\cap E)$, showing that 
$\mu(A\setminus E)=0$ since $E\cap A$ has finite measure, and therefore $A$ is an $*$-atom.

A measure is $*$-atomic if every measurable set $A$ of positive measure contains an $*$-atom. A measure that has no $*$-atoms
is said $*$-nonatomic. Here is an example of a $*$-nonatomic measure which is atomic in the standard sense. Consider $(\RR,\mathcal{P}(\RR))$ as a measurable space
and 
$$
\mu(A)=\begin{cases} \infty &\hbox{if $A$ is uncountable}\\
                     0      &\hbox{otherwise}
\end{cases}
$$
If $\mu(A)>0$, then $A$ is uncountable and can be splitted into two uncountable disjoint sets $B$ and $C$. Then $\mu(A\cap B)=\infty$ and $\mu(A\setminus B)=\infty$.
So, there are no $*$-atoms and then according to the above definition $\mu$ is $*$-nonatomic. 

The other concept we need is the notion of $*$-singular. Two measures $\nu$ and $\lambda$ are said $*$-singular if for all measurable sets $E$, there exist
two measurable sets $F$ and $G$ contained in $E$ such that
$$
\nu(F)=\nu(E),\ \lambda(F)=0, \hbox{ and } \lambda(G)=\lambda(E), \ \nu(G)=0.
$$
The main theorem we need is the following.\\

\noindent {\bf Theorem 2.1} in \cite{decomposition}.  {\it Assume $(\Omega,\F,\mu)$ is a measure space. Then $\mu$ can be decomposed as
$\mu=\nu+\lambda$, where $\nu$ is $*$-atomic and $\lambda$ is $*$-nonatomic. We can assume that $\nu,\lambda$ are $*$-singular, 
in which case the decomposition is unique.}\\

We are now ready to prove the main result of this subsection:

\begin{proof}[Proof of Theorem \ref{ballUA}] It is clear that 
$(iii)\implies (ii)\implies (i)$. Now suppose that $(i)$ holds, that is, $B_{L^p(\Omega,\F,\mu)}$ is UA.  
By Theorem 2.1 in \cite{decomposition}, 
there is a unique decomposition $\mu=\nu+\lambda$ where $\nu$ is $*$-atomic measure, $\lambda$ is $*$-nonatomic  
and $\nu$ and $\lambda$ are $*$-singular.
Consider 
$$
\mathcal{C}=\{ [A]:\ A \hbox{ is an $*$-atom for $\nu$ of finite $\nu$-measure }\}
$$ 
where $[A]$ is the equivalence class of measurable sets $B$ such that $\nu(A\Delta B)=0$. Notice that $[A]\in \mathcal{C}$ if  and only if $A$ is an atom of 
finite $\nu$-measure.
Therefore, if $[A]\neq [B] \in \mathcal{C}$ then $\nu(A\cap B)=0$, that is, $A$ and $B$ are disjoint up to $\nu$-measure $0$. 

If $\mathcal{C}$ is infinite, we take a countable collection  $(E_n)_n$ of atoms for $\nu$, which are disjoint up to $\nu$-measure zero, 
and each one has finite and positive $\nu$-measure. For every $n$ there exists
$F_n\subset E_n$, such that $\nu(F_n)=\nu(E_n)$ and $\lambda(F_n)=0$. Clearly, $(F_n)_n$ is a countable class of 
disjoint atoms for $\nu$, which have positive and finite measure. The measurable set $A=\cup_{n=1}^\infty F_n$ satisfies $\lambda(A)=0$. 
This shows that $\mu|_A=\nu|_A$, so $L^p(A,\F|_A,\mu|_A)$ and $L^p(A,\F|_A,\nu|_A)$ can be identified.

On the other hand, the measure $\nu|_A$ is atomic. Indeed, assume that $D\subset A$ has positive measure. Then for some $n$ it holds $\nu(D\cap F_n)>0$ and then
$D\cap F_n$ contains an $*$-atom $H$ of $\nu$, which has finite measure, and therefore it is an atom for $\nu$. We can apply 
Proposition \ref{atomic_UA} to conclude that $B_{L^p(A,\F|_A,\nu|_A)}$ is not UA, and a fortiori $B_{L^p(\Omega,\F,\mu)}$ is not UA, which is a contradiction. 

The conclusion is that $\nu$ has a finite number
of atoms $(A_n)_{n\in J}$ of finite measure, up to measure $0$, where $J$ is a finite (eventually empty) set. 
Therefore, if $B=\Omega\setminus \cup_{n\in J} A_n$, then any measurable $C\subset B$ has $0$ or infinite $\nu$-measure.

On the other hand, there exists $G\subset B$ such that $\nu(G)=0$ and $\lambda(G)=\lambda(B)$. If there exists $H\subset B$ a measurable set
such that $0<\lambda(H)<\infty$, then we arrive to a contradiction. Indeed, consider $K\subset H$ such that $\nu(K)=0$ and $\lambda(K)=\lambda(H)$. Since
$\lambda(H)$ is finite, this means that $\lambda(H\setminus K)=0$. Now, $\lambda|_K$ is a diffuse measure, because if there exists $L\subset K$
an atom for $\lambda$, then this atom has finite measure and therefore it is an $*$-atom for $\lambda$, which is not possible. 
The contradiction is obtained because $B_{L^p(K,\F|_K,\mu|_K)}$ and $B_{L^p(K,\F|_K,\lambda|_K)}$ 
can be identified and the latter is not UA, according to Corollary \ref{UA_diffuse}.

The conclusion is that $\lambda(H)$ is $0$ or infinite for every $H\subset B$. Since $\lambda(B^c)=0$, we conclude that $\lambda(H)$ is
either $0$ or infinite for every measurable set $H$. Also, $\mu(H)$ is $0$ or infinite, for any $H\subset B$
and $\mu=\nu$ on $A=B^c$. Therefore, $\mu$ is an atomic measure and it has a finite collection of disjoint atoms with finite measure, up to measure zero.

The last part of the Theorem follows from either Corollary \ref{UA_diffuse} or Proposition \ref{atomic_UA}.
\end{proof}

\subsection{Stability of the class of  UA sets}

In this subsection, we study the image of a UA set under classical operations. We start with the following easy proposition:

\begin{proposition}
Let $(\Omega,\F,\mu)$ be a measure space and $p\in[1,+\infty]$. Let $\mathscr{A},\mathscr{B}\subset L^p(\Omega,\F,\mu)$ and $\varepsilon>0$. Then:
\begin{enumerate}
    \item[(i)] if $\mathscr{A}\subset\mathscr{B}$ then $N_{p,\varepsilon}(\mathscr{A})\leq N_{p,\varepsilon}(\mathscr{B})$;
    \item[(ii)] $N_{p,\varepsilon}(\mathscr{A})=N_{p,\varepsilon}(\overline{\mathscr{A}})$;
    \item[(iii)] $N_{p,|\lambda|\varepsilon}(\lambda \mathscr{A})=N_{p,\varepsilon}(\mathscr{A})$ for all $\lambda\in\mathbb R$;
    \item[(iv)] $N_{p,\varepsilon}(\mathscr{A}+\mathscr{B})\leq\min_{t,s>0,t+s\leq \varepsilon}N_{p,t}(\mathscr{A})N_{p,s}(\mathscr{B})$;
\end{enumerate}
In particular if $\mathscr{A}$ and $\mathscr{B}$ are UA then $\overline{\mathscr{A}}$, $\lambda \mathscr{A}$ and $\mathscr{A}+\mathscr{B}$ are UA.
\end{proposition}

\begin{proof}
The proof is left to the reader.
\end{proof}

In the next result, we prove that the closed convex hull of a bounded UA set is still UA.

\begin{theorem}\label{convex_hull}
Let $(\Omega,\F,\mu)$ be a measure space and $p\in(1,+\infty)$. If $\mathscr{A}\subset L^p(\Omega,\F,\mu)$ is a UA set, 
then $\mathscr{A}_K=\{f\in\overline{\text{co}}(\mathscr{A})\ |\ \forall g\in\mathscr{A}\ \|f-g\|_p\leq K\}$ is also UA 
for all $K\geq 0$. More precisely, we have that 
$$
N_{p,\varepsilon}(\mathscr{A}_K)\leq \min_{\eta\in(0,1)}\left(N_{p,(1-\eta)\varepsilon}(\mathscr{A})\right)^{s(\eta)}
$$ 
for all $\varepsilon>0$, where $s(\eta)=\left[\frac{CK}{\eta\varepsilon}\right]^{\frac{\min\{p,2\}}{\min\{p,2\}-1}}+1$ 
and $C$ is a constant depending on $(\Omega,\F,\mu)$ and $p$. In particular, if $\mathscr{A}$ is 
bounded then $\overline{\text{co}}(A)$ is UA.
\end{theorem}

\begin{proof}
Fix $K\geq 0$. For $n\in\mathbb N$, define $$\text{co}_n(\mathscr{A})=\left\{\sum_{i=1}^{n}a_if_i\ |\
a_i\geq 0,\ \sum_{i=1}^{n}a_i=1,\ f_i\in\mathscr{A}\right\}.$$ Remember that $L^p(\Omega,\F,\mu)$ has
non-trivial Rademacher type $r=\min\{p,2\}$ (see Theorem 6.2.14 in \cite{AlbiacKalton}). By Theorem $2.5$ of \cite{ref1}, one has that
$$
d(\text{co}_n(\mathscr{A}),f)\leq\frac{CK}{n^{1-\frac{1}{r}}}
$$ 
for all $f\in\mathscr{A}_K$ and $n\in\mathbb N$ where $C$ is a constant depending on $(\Omega,\F,\mu)$ and $p$.
Therefore, if we take $\varepsilon>0, \eta \in  (0,1)$ and $n_0=\left[\frac{CK}{\eta\varepsilon}\right]^{\frac{r}{r-1}}+1$, 
we will have that $d(\text{co}_{n_0}(\mathscr{A}),f)<\eta\varepsilon$, for all $f\in\mathscr{A}_K$.

Thus, if $f_0\in\mathscr{A}_K$ there exists $g_0=\sum_{i=1}^{n_0}a_if_i\in\text{co}_{n_0}(\mathscr{A})$ such that
$\|f_0-g_0\|_p<\eta\varepsilon$. On the other hand, since $\mathscr{A}$ is UA, there exists $h_i\in\mathscr{G}_{p,k}$ 
where $k=N_{p,(1-\eta)\varepsilon}(\mathscr{A})$ such that
$$
\|f_i-h_i\|_p\le (1-\eta)\varepsilon
$$ 
for all $i\in\{1,...,n_0\}$. One can deduce that 
$$
\left\|f-\sum_{i=1}^{n_0}a_ih_i\right\|_p\leq\|f_0-g_0\|_p+\left\|
\sum_{i=1}^{n_0}a_if_i-\sum_{i=1}^{n_0}a_ih_i\right\|_p\le \eta\varepsilon+
\sum_{i=1}^{n_0} a_i (1-\eta)\varepsilon=\e
$$ 
with $\sum_{i=1}^{n_0}a_ih_i\in\mathscr{G}_{p,k^{n_0}}$. 
We conclude that $N_{p,\varepsilon}(\mathscr{A}_K)\leq k^{n_0}$.
\end{proof}

\begin{remark}
Note that if $\mathscr{A}$ is an unbounded UA set then $\overline{\text{co}}(\mathscr{A})$ may not be UA.
In fact, $\mathscr{G}_{p,2}$ is UA but $\overline{\text{co}}(\mathscr{G}_{p,2})=L^p(\Omega,\F,\mu)$ (since
$\text{co}(\mathscr{G}_{p,2}) $ is the set of simple functions) is not UA in general for any $p\in[1,\infty]$.
Remark that the previous theorem is not interesting if $p=\infty$ since any bounded set is UA by Proposition \ref{UA_infty}.

\end{remark}

\begin{remark}
The previous theorem is false if $p=1$. In fact remember that $B_{\ell_1}=\overline{co}(Ext(B_{\ell_1}))$ and $Ext(B_{\ell_1})=
\{\pm\delta_n\}_{n\in\mathbb N}$, where $Ext(B_{\ell_1})$ is the set of extreme points of $B_{\ell_1}$. It follows that 
$Ext(B_{\ell_1})$ is UA but we 
have seen that $B_{\ell_1}$ is not UA (see Theorem \ref{ballUA}). More generally, using the previous result, it is easy to 
show that there exists a UA set $\mathscr{A}\subset\ell_p$ such that 
$B_{\ell_p}=\overline{co}(\mathscr{A})$ if and only if $p\in\{1,\infty\}$.
\end{remark}

In the next result we study stability properties of UA classes under H\"older transformations.
Recall that a real function $\Psi$ is uniformly $\alpha$-H\"older if there exists a constant $K$, 
such that
$$
|\Psi(x)-\Psi(y)|\le K |x-y|^\alpha.
$$
With this definition, the identity function is not uniformly $\alpha$-H\"older for $\alpha<1$. 
To enlarge the class
of uniformly $\alpha$-H\"older functions we consider the following classes of H\"older functions, 
denoted $\mathbb{H}(K,\alpha)$
for $0<\alpha\le 1$, which
consists of real functions $\Psi$ such that for all $x,y$, it holds
$$
|\Psi(x)-\Psi(y)|\le K (|x|+|y|+1)^{1-\alpha} |x-y|^{\alpha}.
$$
We can assume without loss of generality that $K\ge 1$.
We notice that $\mathbb{H}(K,1)$ is the set of $K$-Lipschitz functions. If $0<\beta\le \alpha$, then 
$\mathbb{H}(K,\alpha)\subset \mathbb{H}(K,\beta)$. Also $\mathbb{H}(K,\alpha)$ contains the class of
uniformly H\"older
functions.

\begin{proposition} Assume that $(\Omega,\F,\mu)$ is a finite measure space and $\mathscr{A}$ is UA
in $L^q(\Omega,\F,\mu)$, for some $q\in [1,\infty]$. Consider $\alpha\in(0,1]$, and we assume further
that $\mathscr{A}$ is bounded in $L^q(\Omega,\F,\mu)$ when $\alpha<1$. Then, 
the $\mathbb{H}(K,\alpha)$-transform of $\mathscr{A}$ given by
$$
\mathbb{H}(K,\alpha)(\mathscr{A})=\{\Psi(f):\, f\in \mathscr{A}, \, \Psi \in \mathbb{H}(K,\alpha)\}
$$
is UA in $L^p(\Omega,\F,\mu)$ for any $1\le p\le q$. Moreover, for $\e\in(0,1]$
$$
N_{p,\e}(\mathbb{H}(K,\alpha)(\mathscr{A}))\le N_{q,(\e/\Gamma)^{\frac{1}{\alpha}}}(\mathscr{A}),
$$
where 
$$
\Gamma=\Gamma(\alpha,p,q)=K\mu(\Omega)^{-r} \begin{cases} 
\left(2B+1+\mu(\Omega)^{\frac{1}{q}}\right)^{1-\alpha} &\hbox{if } \alpha<1\\
1 &\hbox{if } \alpha=1
\end{cases}
$$ 
with $B$ a bound for $\mathscr{A}$ in $L^q(\Omega,\F,\mu)$ and $r=\frac{1}{p}-\frac{1}{q}$.
\end{proposition}

\begin{proof} The case $\alpha=1$ is straightforward so, we assume $\alpha<1$.  
We assume first that $p=q$. 
Consider $\e\in(0,1]$, $k=N_{q,\e}(\mathscr{A}), f\in \mathscr{A}, g\in \mathscr{G}_{q,k}$ such that 
$\|f-g\|_q\le \e$ and $\Psi\in \mathbb{H}(K,\alpha)$. We have
$$
\int |\Psi(f(x))-\Psi(g(x))|^q d\mu(x)\le K^q \int (|f(x)|+|g(x)|+1)^{q(1-\alpha)}\,
|f(x)-g(x)|^{q\alpha} d\mu(x).
$$
Now, we apply H\"older's inequality for $s=\frac{1}{\alpha}$ and its conjugated index $t=\frac{1}{1-\alpha}$ to get
$$
\int |\Psi(f(x))-\Psi(g(x))|^q d\mu(x)\le K^q \left(\int (|f(x)|+|g(x)|+1)^{q} d\mu(x)\right)^{1-\alpha}
\left(\int |f(x)-g(x)|^q d\mu(x)\right)^{\alpha}
$$
which implies
$$
\|\Psi(f)-\Psi(g)\|_q \le K \left(\|f\|_q+\|g\|_q+\|1\|_q\right)^{(1-\alpha)} \|f-g\|_q^{\alpha}.
$$
If $B$ is a bound for $\mathscr{A}$, we conclude that $\|g\|_q\le B+1$, which shows
$$
\|\Psi(f)-\Psi(g)\|_q \le K \left(2B+1+\mu(\Omega)^{\frac{1}{q}}\right)^{(1-\alpha)} \e^{\alpha}
=\Gamma \e^{\alpha}.
$$
Since $\Psi(g)\in \mathscr{G}_{q,k}$, we deduce that
$$
N_{q,\Gamma\e^\alpha}(\mathbb{H}(K,\alpha)(\mathscr{A}))\le N_{q,\e}(\mathscr{A}),
$$
and the result is shown in this case.
The case $p<q$ follows from Corollary \ref{UA_lq}.
\end{proof}
We point out that under the hypothesis of the Theorem, we have 
$
\mathscr{A}\subset \mathbb{H}(K,\alpha)(\mathscr{A}).
$

\vspace{1cm}
\textbf{Acknowledgements:}
G. Grelier was supported by the Grants of Ministerio de Econom\'ia, Industria y Competitividad MTM2017-83262-C2-2-P; 
Fundaci\'on S\'eneca Regi\'on de Murcia 20906/PI/18; 
and by MICINN 2018 FPI fellowship with reference PRE2018-083703, associated to grant MTM2017-83262-C2-2-P. 
J. San Mart\'in was supported in part by BASAL ANID FB210005.

\nocite{*}

\end{document}